\documentclass[a4paper, english]{amsart}
\usepackage[T1]{fontenc}
\usepackage[latin9]{inputenc}
\usepackage{color}
\usepackage[english]{babel}
\usepackage{amstext}
\usepackage{amsthm}
\usepackage{amsmath}
\usepackage{amssymb}
\usepackage{yfonts}

\usepackage{tikz}
\usetikzlibrary{
  decorations.pathmorphing,%
  decorations.markings,%
  decorations.pathreplacing,%
  decorations.text,%
  calc,%
  patterns,%
  shapes.geometric,%
  arrows,
  decorations.shapes,
  positioning,
  plotmarks,
  shadings,
  math,
  intersections
}
\tikzset{>=latex} 
\tikzset{font=\small}
\tikzset{mark size=1.5pt, mark options=thin}
\tikzset{pin distance=4pt,
  every pin edge/.style={<-, thin, shorten <= -2pt}}

\usepackage{graphicx}
\usepackage{subfigure}
\usepackage[unicode=true,pdfusetitle,
 bookmarks=true,bookmarksnumbered=false,bookmarksopen=false,
 breaklinks=false,pdfborder={0 0 0},pdfborderstyle={},backref=false,colorlinks=true]
 {hyperref}
\definecolor{myblue}{rgb}{0,0,0.6}
\hypersetup{colorlinks=true, linkcolor=myblue,citecolor=myblue,filecolor=myblue,urlcolor=myblue}

\usepackage[a4paper]{geometry}
\newcommand{\e}{\epsilon}

\DeclareMathOperator{\Op}{Op}

\definecolor{uipoppy}{RGB}{221,128,71}
\definecolor{uipaleblue2}{RGB}{179,196,215}
\definecolor{uiviolet}{RGB}{86,86,99}
\definecolor{uiblack}{RGB}{0, 0, 0}
\definecolor{azul}{RGB}{0,128,255}
\definecolor{verde}{RGB}{50,180,50}
\definecolor{uipaleblue}{RGB}{108,199,220}
\definecolor{light-gray}{gray}{0.8}
\definecolor{light-blue}{rgb}{0.53,.8,98}
\definecolor{green1}{RGB}{50,180,50}

\definecolor{jeffColor}{RGB}{102, 0, 204}
\definecolor{yaizaColor}{RGB}{0, 153, 153}
\definecolor{pale-verde}{RGB}{155,207,145}
\definecolor{med-blue}{rgb}{0.29,0.28,.87}

\definecolor{periodColor}{RGB}{255, 167, 105}
\definecolor{dark-green}{RGB}{135, 194, 130}

\usepackage[show]{ed}

\makeatletter

\newcommand{\supp}{\operatorname{supp}}

\newcommand{\Ell}{\operatorname{Ell}}
\newcommand{\WF}{\operatorname{\WFh }}
\newcommand{\Id}{I}
\newcommand{\comp}{\operatorname{comp}}
\renewcommand{\Im}{\operatorname{Im}}
\renewcommand{\Re}{\operatorname{Re}}

\newtheorem{theorem}{Theorem}[section]

\newtheorem{lem}[theorem]{Lemma}
\newtheorem{prop}[theorem]{Proposition}

\newtheorem*{claim*}{Claim}

\numberwithin{equation}{section}
\numberwithin{table}{section}
\numberwithin{figure}{section}

\newtheorem{rem}[theorem]{Remark}
\newtheorem*{rem*}{Remark}

\newcommand{\tfa}{\text{ for all }}

\newcommand{\tas}{\text{ as }}
\newcommand{\tand}{\text{ and }}

\newcommand{\bre}{\begin{rem}}
\newcommand{\ere}{\end{rem}}
\newcommand{\bit}{\begin{itemize}}
\newcommand{\eit}{\end{itemize}}
\newcommand{\ben}{\begin{enumerate}}
\newcommand{\een}{\end{enumerate}}
\newcommand{\beq}{\begin{equation}}
\newcommand{\eeq}{\end{equation}}
\newcommand{\beqs}{\begin{equation*}}
\newcommand{\eeqs}{\end{equation*}}
\newcommand{\bpf}{\begin{proof}}
\newcommand{\epf}{\end{proof}}
\newcommand{\ble}{\begin{lem}}
\newcommand{\ele}{\end{lem}}

\definecolor{dalcol}{rgb}{0,0.3,0}

\definecolor{jecol}{rgb}{0.4,0,0}
\definecolor{escol}{rgb}{0,0,0.8}
\definecolor{estcol}{rgb}{0,0.5,0}
\definecolor{esnewcol}{rgb}{0,0.5,0}

\newcommand{\tendi}{\rightarrow \infty}

\usepackage{soul}

\newcommand{\PML}{{\rm pml}}

\newcommand{\tr}{{\rm tr}}
\newcommand{\Rea}{\mathbb{R}}

\newcommand{\cH}{\mathcal{H}}
\newcommand{\cD}{\mathcal{D}}

\newcommand{\WFh}{\operatorname{WF}_{\hbar}}

\newcommand{\mc}[1]{\mathcal{#1}}
\newcommand{\loc}{\operatorname{loc}}

\newcommand{\newchi}{\psi}
\newcommand{\indicator}{1}
\newcommand{\Arg}{\operatorname{Arg}}

\usepackage{scalerel,stackengine}
\stackMath
\newcommand\reallywidehat[1]{%
\savestack{\tmpbox}{\stretchto{%
  \scaleto{%
    \scalerel*[\widthof{\ensuremath{#1}}]{\kern-.6pt\bigwedge\kern-.6pt}%
    {\rule[-\textheight/2]{1ex}{\textheight}}
  }{\textheight}%
}{0.5ex}}%
\stackon[1pt]{#1}{\tmpbox}%
}
\newcommand{\mythmname}[1]{\emph{(#1.)}}

\newcommand{\rhs}{g}
\newcommand{\resolvent}{\Upsilon}

\begin{document}
\title[Perfectly-matched-layer truncation is exponentially accurate]{Perfectly-matched-layer truncation is exponentially accurate at high frequency}
\author{Jeffrey Galkowski} 
\address{Department of Mathematics, University College London, London, WC1H 0AY, UK}
\email{J.Galkowski@ucl.ac.uk}

\author{David Lafontaine} 
\address{CNRS and Institut de Math\'ematiques de Toulouse, UMR5219;
Universit\'e de Toulouse, CNRS; UPS, F-31062 Toulouse Cedex 9, France}
\email{david.lafontaine@math.univ-toulouse.fr}

\author{Euan A.~Spence}
\address{Department of Mathematical Sciences, University of Bath, Bath, BA2 7AY, UK}
\email{E.A.Spence@bath.ac.uk}

\date{\today}

\begin{abstract}
We consider a wide variety of Helmholtz scattering problems including scattering by Dirichlet, Neumann, and penetrable obstacles. We 
consider a radial perfectly-matched layer (PML) and
show that for any fixed PML 
width and a steep-enough scaling angle, the PML solution is exponentially close, both in frequency and the tangent of the scaling angle, to the true scattering solution. Moreover, for a fixed 
scaling angle and large enough PML width, the PML solution is exponentially close to the true scattering solution in both frequency and the PML width. In fact, the exponential bound holds with rate of decay $c(w\tan\theta -C) k$ where $w$ is the PML width and $\theta$ is the scaling angle. More generally, the results of the paper hold in the framework of black-box scattering under the assumption of an exponential bound on the norm of the cutoff resolvent, thus including problems with strong trapping. These are the first results on the exponential accuracy of PML at high-frequency with non-trivial scatterers.
\end{abstract}

\maketitle


\section{Introduction}
\subsection{Context and background}

Since the work of Berenger~\cite{Be:94}, perfectly matched layers (PMLs) have become a standard tool in the numerical simulation of frequency-domain wave problems such as the Helmholtz equation. This method approximates the solution of a scattering problem in an unbounded domain by making a complex change of variables in a layer away from the region of interest and truncating the problem with a Dirichlet condition. 

It is well known that, for fixed frequency, the error in the truncation decreases exponentially with the width of the perfectly matched layer (PML); see \cite[Theorem 2.1]{LaSo:98}, \cite[Theorem A]{LaSo:01}, \cite[Theorem 5.8]{HoScZs:03}, \cite[Theorem 3.4]{BrPa:07}.
However these error bounds are not explicit in the frequency.

The only frequency-explicit error bounds on the accuracy of PMLs obtained up till now are for the model problem of no scatterer. In this case, the error is known to decrease exponentially in the width of the PML, the tangent of the scaling angle, and the frequency; this was proved in \cite[Lemma 3.4]{ChXi:13} (for $d=2$) and \cite[Theorem 3.7]{LiWu:19} (for $d=2,3$) using 
 the fact that the solution of this problem can be written explicitly. 
 
In this paper, we consider a wide variety of Helmholtz scattering problems, including scattering by Dirichlet, Neumann, and penetrable obstacles in any dimension,
and including problems with strong trapping. 
We consider a radial PML and prove that, provided that the PML change of variables is $C^{3}$, the error decreases exponentially in frequency, the PML width, and the scaling angle with a rate that, at least in one dimension, is sharp.

We first state these results applied to the particular problem of plane-wave scattering by an impenetrable obstacle in \S\ref{s:1.2}, and then state the results in the ``black-box scattering'' framework in \S\ref{sec:introBB}.

\subsection{The main results applied to plane-wave scattering by an impenetrable obstacle}\label{s:1.2}
Let $\Omega_-\subset \mathbb{R}^d$ be bounded and open with Lipschitz boundary $\Gamma_-:=\partial\Omega_-$ and connected open complement, $\Omega_+:=\mathbb{R}^d\setminus \Omega_-$. 
Truncation by a PML is widely used to compute approximations to the exterior Helmholtz problem
\begin{equation}
\label{e:helmholtz}
\begin{gathered}(-\Delta-k^2)u^S=0\text{ in } \Omega_+,\qquad
Bu^S=-B\exp(ikx\cdot a)\text{ for }x\in \Gamma_- ,\\
(\partial_r-ik) u^S=o(r^{\frac{1-d}{2}})\text{ as }r:=|x|\to \infty.
\end{gathered}
\end{equation}
Here, $B$ is an operator on the boundary giving either the Dirichlet (sound-soft) condition, $u\mapsto u|_{\Gamma_- }$ or Neumann (sound-hard) condition $u\mapsto (\partial_\nu u)|_{\Gamma_- }$, and $\nu(x)$ is the outward unit normal to $\Omega_-$.
Physically, $u^S$ corresponds to the scattered wave generated when the plane wave $\exp(ikx\cdot a)$ hits the obstacle $\Omega_-$.

Let $R_{P}(k)$ denote the solution operator for~\eqref{e:helmholtz} (see Proposition~\ref{p:cutoff} for the precise definition); the letter $R$ stands for ``resolvent'', and the subscript $P$ is there because we use this notation for the solution operator for the more-general operator $P$ in \S\ref{sec:introBB} below.
Let $\chi \in C_c^\infty(\mathbb{R}^d)$ with $\chi \equiv 1$ in a neighbourhood of the convex hull of $\Omega_-$. We define the exponential rate of growth for the solution operator through a subset $J\subset \mathbb{R}$ that is  unbounded above:
\begin{equation}
\label{e:Lambda1}
\Lambda(P,J):=\limsup_{\substack{k\to \infty\\k\in J}}\frac{1}{k}\log \|\chi R_P(k)\chi\|_{L^2\to L^2}.
\end{equation}
We write $\Lambda(P)$ for $\Lambda(P,\mathbb{R})$. If $\Gamma_-$ is $C^\infty$ then $\Lambda(P)<\infty$.
If, in addition, $\Gamma_-$ is nontrapping, then $\Lambda(P)=0$.  Finally, if $\Gamma_-$ is \emph{only} Lipschitz, then for all $\delta>0$ there is a set $J\subset \mathbb{R}$ with $|\mathbb{R}\setminus J|\leq \delta$ such that $\Lambda(P,J)=0$; see \S\ref{sec:introBB} and~\S\ref{s:examples} for details and references.

\begin{figure}
\begin{tikzpicture}
\def \R{2.5};
\def \a{72};
\def \PML{1.5}
\def \n{5};
\def \rate{1.25}
\begin{scope}[scale=.7]
\draw[pattern=north west lines] (0:\PML*\R)\foreach \x in {1,...,\n} {
            -- ({\x*360/\n}:\PML*\R)
        } -- cycle  ;
\draw[fill=white] (0,1)--(1,1)--(1,-1)--(.25,-1)--(.25,-.5)--(.5,-.5)--(.5,-.75)--(.75,-.75)--(.75,.25)--(0,.25)-- (-.75,.25)--(-.75,-.75)--(-.5,-.75)--(-.5,-.5)--(-.25,-.5)--(-.25,-1)--(-1,-1)--(-1,1)--cycle;
\draw[fill=white, opacity=.65](0,0) circle ({\PML*\R*cos(360/(2*\n))});
\draw (0,0) circle (\R);
\draw[->] (0,0)--(0,-\R/2)node[right]{$R_1$}--(0,-\R);
\draw[->] (0,0)--(40:{\PML*2*\R/3*cos(360/\n/2)})node[below]{$R_{\tr}$}--(40:{\PML*\R*cos(360/\n/2)});
\draw[fill=light-gray,opacity=.6] (0,1)--(1,1)--(1,-1)--(.25,-1)--(.25,-.5)--(.5,-.5)--(.5,-.75)--(.75,-.75)--(.75,.25)--(0,.25)-- (-.75,.25)--(-.75,-.75)--(-.5,-.75)--(-.5,-.5)--(-.25,-.5)--(-.25,-1)--(-1,-1)--(-1,1)--cycle;
\node at(.1,.5){$\Omega_-$};


\draw (10:.95*\PML*\R)node[right]{$\Omega_{\tr,+}$};

\end{scope}
\end{tikzpicture}
\caption{\label{f:diagram}The diagram shows the obstacle, $\Omega_-$, the ball of radius $R_1$ (outside of which the scaling begins), the ball of radius $R_{\tr}$, and $\Omega_{\tr,+}$ (shaded in the hatched lines) where the domain exterior to $\Omega_-$ is truncated.}
\end{figure}
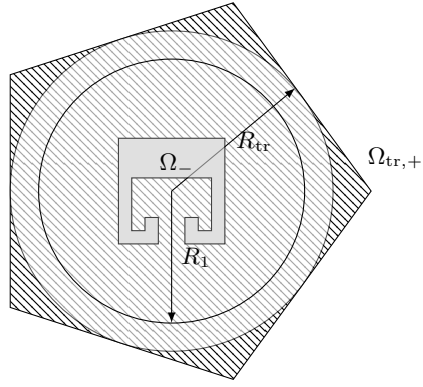
We now describe the geometric set-up for the PML truncation; see Figure~\ref{f:diagram} for a schematic.
Let  $R_2>R_1>0$, such that $\Omega_-\Subset B(0,R_1)$. Next, let $R_{\tr}>R_1$ and $\Omega_{\tr}\subset \mathbb{R}^d$ be a bounded Lipschitz open subset with $B(0,R_{\tr})\subset \Omega_{\tr}$. Finally, let $\Omega_{\tr,+}:=\Omega_{\tr}\cap \Omega_+$, $\Gamma_{\tr}:=\partial\Omega_{\tr}$, and $0\leq \theta<\pi/2$. The PML method replaces~\eqref{e:helmholtz} by the following problem
\begin{equation}
\label{e:PML}
(-\Delta_\theta-k^2)v^S=0\text{ in } \Omega_{\tr,+},\qquad
Bv^S=-B\exp(ikx\cdot a)\text{ for }x\in \Gamma_- ,\qquad
v^S=0\text{ for }x\in \Gamma_{\tr}.
\end{equation}
Here, $-\Delta_\theta$ is a second order differential operator that is given in spherical coordinates $(r,\omega)\in [0,\infty)\times S^{d-1}$  by
\begin{align}
\label{e:deltaTheta}
\Delta_\theta&= \Big(\frac{1}{1+if_\theta'(r)}\partial_r\Big)^2+\frac{d-1}{(r+if_\theta(r))(1+if'_\theta(r))}\partial_r +\frac{1}{(r+if_\theta(r))^2}\Delta_\omega,\\
&= \frac{1}{(1+if_\theta'(r))(r+ i f_\theta(r))^{d-1}}\frac{\partial}{\partial r}
\left( \frac{ (r+ i f_\theta(r))^{d-1}}{1+ i f_\theta'(r)}\frac{\partial}{\partial r}
\right)+\frac{1}{(r+if_\theta(r))^2}\Delta_\omega,
\nonumber
\end{align}
with $\Delta_\omega$ the surface Laplacian on $S^{d-1}$ and $f_\theta(r)\in C^{3}([0,\infty);\mathbb{R})$  given by $f_\theta(r)=f(r)\tan\theta$ for some $f$ satisfying
\begin{equation}
\label{e:fProp}
\begin{gathered}
\{f(r)=0\}=\{f'(r)=0\}=\{r\leq R_1\},\qquad f'(r)\geq 0,\qquad f(r)\equiv r \text{ on }r\geq R_2;
\end{gathered}
\end{equation}
i.e., the scaling ``turns on'' at $r=R_1$, and is linear when $r\geq R_2$. We emphasize that $R_{\tr}$ can be $<R_2$, i.e., we allow truncation before linear scaling is reached. Indeed, $R_2>R_1$ can be arbitrarily large and therefore, given any bounded interval $[0,R]$ and any function $g\in C^3([0,R])$ satisfying 
\begin{equation}
\label{e:fPropA}
\begin{gathered}
\{g(r)=0\}=\{g'(r)=0\}=\{r\leq R_1\},\qquad g'(r)\geq 0,
\end{gathered}
\end{equation}
our results hold for an $f$ with $f|_{[0,R]}=g$. 
A concrete example of a $g(r)$ satisfying the conditions~\eqref{e:fPropA} is the piecewise degree-three polynomial
\beq\label{eq:examplef}
g(r) =(r-R_1)^31_{[R_1,\infty)}(r).
\eeq

\begin{rem}[Link with notation used in the numerical-analysis literature]\label{rem:notation}In \eqref{e:PML}-\eqref{e:fProp} the PML problem is written using notation from the method of complex scaling (see, e.g., \cite[\S4.5]{DyZw:19}).
In the numerical-analysis literature on PML, the scaled variable is often written as $r(1+ i \widetilde{\sigma}(r))$ with $\widetilde{\sigma}(r)= \sigma_0$ for $r$ sufficiently large, see, e.g., \cite[\S4]{HoScZs:03}, \cite[\S2]{BrPa:07}. To convert from our notation, set $\widetilde{\sigma}(r)= f_\theta(r)/r$ and $\sigma_0= \tan\theta$.
We also highlight that, whereas the numerical-analysis literature on radial PMLs often assumes that the exterior domain is truncated by a ball (i.e., $\Omega_{\tr} = B(0,R_{PML})$ for some $R_{PML}>R_1$), the PML problem \eqref{e:PML} is posed with a general truncation boundary.
One practical advantage of allowing an arbitrary truncation boundary is that, 
when solving the PML problem with the finite-element method (FEM)
one can then use simplicial elements without having to deal with error in approximating the truncation boundary. 
\end{rem}

\begin{theorem}
\label{thm:new}
Let $\Gamma_-$ be Lipschitz and $J\subset \mathbb{R}$ unbounded above with $\Lambda(P,J)<\infty$. 
For all $\eta,\e>0$ there exist $C, C',k_0, c_\epsilon>0$ (independent of $R_{\tr}$ and $R_1$) such that for all $R_{\tr}>R_1+\e$, $B(0,R_{\tr})\subset\Omega_{\tr}\Subset \mathbb{R}^d$ with Lipschitz boundary, there exists $\theta_0(P,J,R_{\tr})<\pi/2$ (with $\theta_0(P,J,R_{\tr})$ a non-increasing function of $R_\tr$) such that if $\theta_0(P,J,R_{\tr})+\e<\theta<\pi/2-\e$, $k>k_0$ with $k\in J$, and $a\in \mathbb{R}^d$, then the solution $v^S$ to \eqref{e:PML} exists, is unique, and satisfies 
\begin{equation}
\begin{gathered}
\frac{\|u^S-v^S\|_{H^1(B(0,R_1)\setminus\Omega_-)}}{\|u^S+e^{ikx\cdot a}\|_{L^2(B(0,R_1)\setminus\Omega_-)}}\leq 
C\exp\Big(-k\big( (2-\eta)c_\epsilon (R_{\tr}-R_1-\epsilon)\tan \theta - 2\Lambda(P,J)\big)\Big),\\
\|u^S-v^S\|_{H^1(B(0,R_1)\setminus\Omega_-)}\leq C' \frac{\|u^S-v^S\|_{H^1(B(0,R_1)\setminus\Omega_-)}}{\|u^S+e^{ikx\cdot a}\|_{L^2(B(0,R_1)\setminus\Omega_-)}},
\label{e:relErrorNew}
\end{gathered}
\end{equation}
where $u^S$ is the solution to \eqref{e:helmholtz}. Furthermore, if $\Lambda(P,J)=0$, then $\theta_0(P,J,R_{\tr})=0$.
\end{theorem} 

Theorem \ref{thm:new} shows that both the absolute and the relative error in the PML approximation of the total field $u^S + e^{ikx\cdot a}$ is exponentially small in $k$, the PML width (i.e., $R_{\tr}-R_1$), and the tangent of the scaling angle (i.e., $\tan \theta$).

Theorem \ref{thm:new} is a consequence of the following more general result, which gives explicitly the rate of decay and $\theta_0(P,J,R_{\tr})$. This result involves the following two functions;

\begin{gather}
\label{e:weightDer}
\Phi_\theta(r):=\begin{cases}\inf_{t\geq 0}\Big|\Im \Big((1+if_\theta'(r))\sqrt{1-\frac{t}{(r+if_\theta(r))^2}}\,\Big)\Big|,&d\geq 2,\\
f_\theta'(r),&d=1,
\end{cases}
\\
\theta_0(P,J,R_{\tr}):=
\sup\Big\{\theta\,:\, \int_{R_1}^{R_{\tr}}\Phi_\theta(r)\,dr\leq \Lambda(P,J)\Big\}.\label{eq:theta0}
\end{gather}
To better understand these functions, we record the following:
\bit
\item $\Phi_\theta(r)\geq 0$ for all $r$ (by definition), and $\Phi_\theta(r)>0$ when $r>R_1$ and $\theta>0$ (by Part (1) of Lemma \ref{lem:Phi} below).
\item If $\Lambda(P,J)=0$, then $\theta_0(P,J,R_{\tr})=0$.
Furthermore, for any $\Lambda(P,J)$, $\theta_0(P,J,R_{\tr})<\pi/2$; this follows from Part (1) of Lemma \ref{lem:Phi} below and the fact that $\tan (\pi/2)=\infty$.
In addition, $\theta_0(P,J,R_{\tr})$ is a non-increasing function of $R_\tr$ (as claimed in Theorem \ref{thm:new}) by the properties of $\Phi_\theta$ above.
\eit

Figure \ref{fig:Phi} plots $\Phi_\theta(r)$ (for $d\geq 2$) and its integral for $f(r)$ given by \eqref{eq:examplef}.

\begin{theorem} 
\label{t:expEst}
Let $\Gamma_-$ be Lipschitz and $J\subset \mathbb{R}$ unbounded above with $\Lambda(P,J)<\infty$. Then for all $\eta,\e>0$ there exist $C, C',k_0>0$ (independent of $R_{\tr}$ and $R_1$) such that for all $R_{\tr}>R_1+\e$, $B(0,R_{\tr})\subset\Omega_{\tr}\Subset \mathbb{R}^d$ with Lipschitz boundary, $\theta_0(P,J,R_{\tr})+\e<\theta<\pi/2-\e$, $k>k_0$ with $k\in J$, and $a\in \mathbb{R}^d$, 
the solution $v^S$ to \eqref{e:PML} exists, is unique, and satisfies 
\begin{equation}
\begin{gathered}
\frac{\|u^S-v^S\|_{H^1(B(0,R_1)\setminus\Omega_-)}}{\|u^S+e^{ikx\cdot a}\|_{L^2(B(0,R_1)\setminus\Omega_-)}}\leq C \exp\bigg(-k\Big( (2-\eta) \int_{R_1}^{R_{\tr}}\Phi_\theta(r)\,dr-2\Lambda(P,J)\Big) \bigg),\\
\|u^S-v^S\|_{H^1(B(0,R_1)\setminus\Omega_-)}\leq C' \frac{\|u^S-v^S\|_{H^1(B(0,R_1)\setminus\Omega_-)}}{\|u^S+e^{ikx\cdot a}\|_{L^2(B(0,R_1)\setminus\Omega_-)}}
\label{e:relError1}
\end{gathered}
\end{equation}
where $u^S$ is the solution to \eqref{e:helmholtz}. 
\end{theorem} 
\noindent Moreover, when $d=1$, explicit calculations show that our estimate is nearly optimal in the sense that the factor $2-\eta$ multiplying {$\int_{R_1}^{R_{\tr}}\Phi_\theta(r)\,dr$} in~\eqref{e:relError1} cannot be replaced by any number larger than $2$. 

To better understand the estimate~\eqref{e:relError1}, we record five properties of the function $\Phi_\theta(r)$;
note that Properties~\eqref{i:thetaDep},~\eqref{i:evLinear} and~\eqref{i:rSmall} are illustrated in the right-hand plots of Figures \ref{fig:Phi} and~\ref{fig:Phi2}.

\begin{figure}[h!]
\begin{center}
    \includegraphics[width=\textwidth]
    {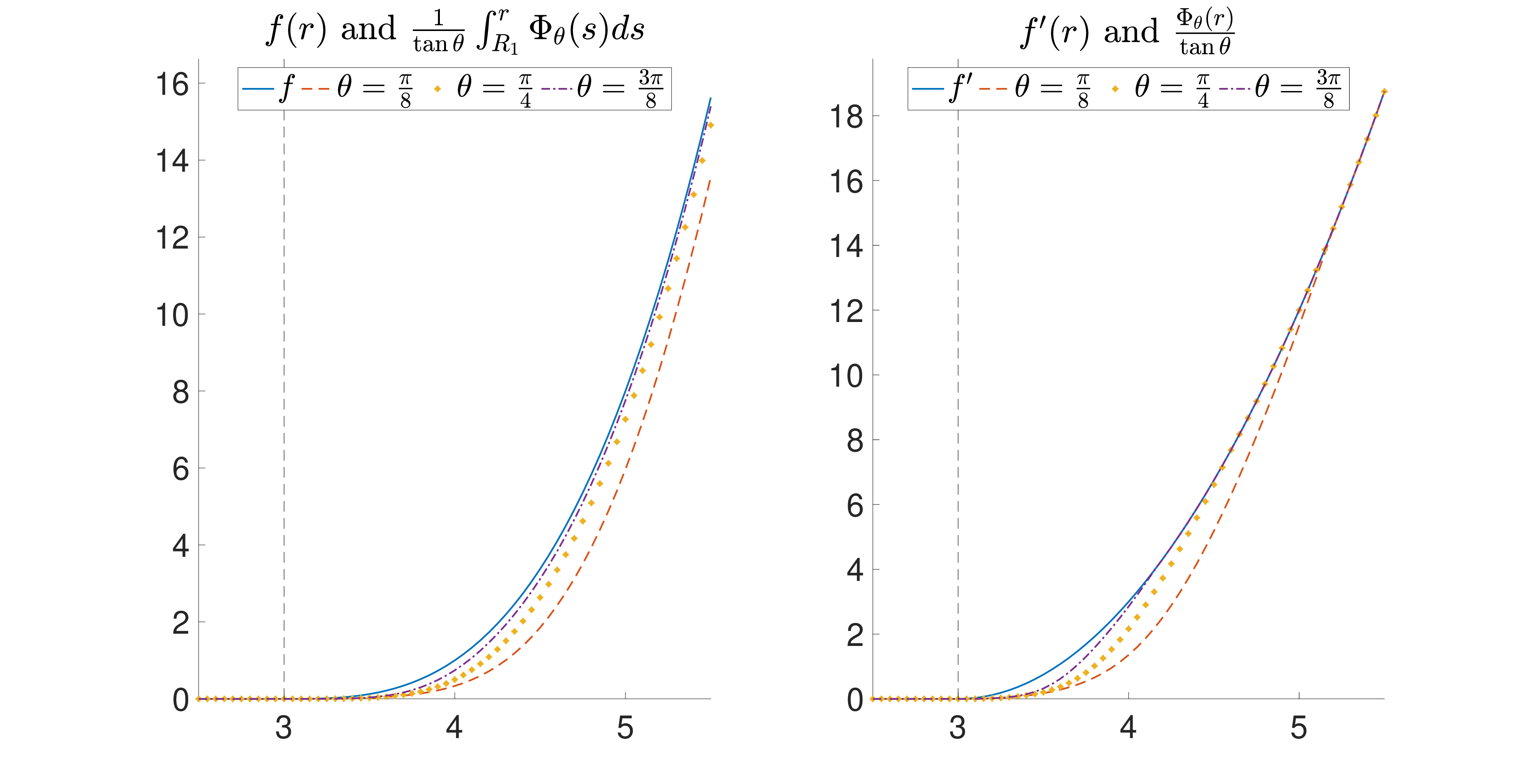}
    \end{center}
    \caption{Plots of $f(r)$ with $f|_{[0,6]}$ given by \eqref{eq:examplef}, $f'(r)$, $\frac{1}{\tan\theta}\Phi_\theta(r)$ (for $d\geq 2$), and $\frac{1}{\tan\theta}\int_R^r \Phi_\theta(r)\,dr$ for $R_1=3$.}
       \label{fig:Phi}
\end{figure}

\begin{figure}[h!]
\begin{center}
    \includegraphics[width=\textwidth]
    {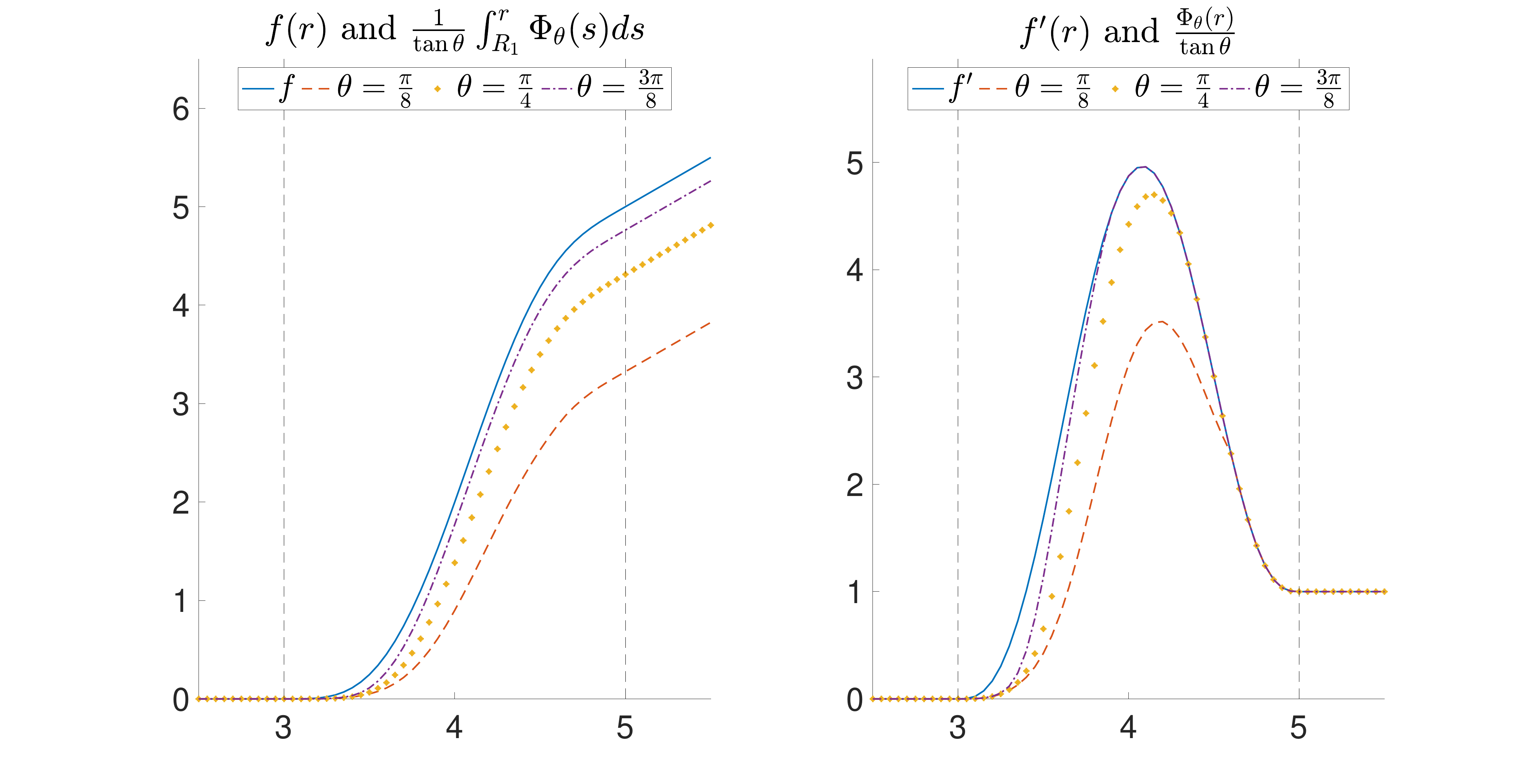}
    \end{center}
    \caption{Plots of $f(r)$ given by \eqref{e:8degrees}, $f'(r)$, $\frac{1}{\tan\theta}\Phi_\theta(r)$ (for $d\geq 2$), and $\frac{1}{\tan\theta}\int_R^r \Phi_\theta(r)\,dr$ for $R_1=3$ and $R_2=5$.}
       \label{fig:Phi2}
\end{figure}

\begin{lem}\label{lem:Phi}
\noindent \vspace{-.5em}
\begin{enumerate}
\item \label{i:thetaDep}For all $\delta>0$, there is $c_\delta>0$ such that $\Phi_\theta(r)>c_\delta\tan \theta$ on $r>R_1+\delta$, $\theta>\delta$. 
\item \label{i:iffTheta} $\Phi_\theta(r)=f_\theta'(r)$ if and only if
\beq\label{e:PhitCond}
\tan^2\theta\geq \frac{r^2}{f(r)^2}-\frac{2r}{f'(r)f(r)}.
\eeq
\item \label{i:evLinear} If $f(r) =r$, $f'(r)=1$, then $\Phi_\theta(r)=f_\theta'(r)$.
\item\label{i:rSmall} For all $\delta>0$, there is $\theta_\delta<\pi/2$ such that for $\theta>\theta_\delta$, $\Phi_\theta(r)=f_\theta'(r)$ on $r\geq R_1+\delta$,
\item \label{i:continuous}The map $(r,\theta)\mapsto\Phi_\theta(r)$ is continuous for $(r,\theta)\in [0,\infty)\times (0,\pi/2)$.
\end{enumerate}
\end{lem}
Point \eqref{i:thetaDep} in Lemma \ref{lem:Phi} implies that, for $R_{\tr}>R_1+ \delta$,
\beq\label{eq:Fri1}
-\int_{R_1}^{R_{\tr}}\Phi_\theta(r)\,dr \leq  -c_\delta (R_{\tr}-R_1-\delta)\tan \theta.
\eeq
Points \eqref{i:thetaDep} and~\eqref{i:evLinear} in Lemma \ref{lem:Phi} imply that, for $R_{\tr}>R_2$, 
\beq\label{eq:Fri2}
-\int_{R_1}^{R_{\tr}}\Phi_\theta(r)\,dr \leq  -c_\delta (R_2-R_1-\delta)\tan \theta- (R_{\tr}-R_2)\tan\theta< - (R_{\tr}-R_2)\tan\theta;
\eeq
Point~\eqref{i:rSmall} in Lemma \ref{lem:Phi} implies that for all $\delta>0$ there is $\theta_\delta<\pi/2$ such that for $\theta>\theta_\delta$, 
\beq\label{eq:Fri3}
-\int_{R_1}^{R_{\tr}}\Phi_\theta(r)\,dr \leq -\big(f(R_{\tr})-f(R_1+\delta)\big)\tan\theta.
\eeq
By \eqref{eq:Fri1}, for $R_{\tr}>R_1+ \delta$, the right-hand side of \eqref{e:relError1} is less than or equal to 
\beq\label{eq:Fri4}
C\exp\Big(-k\big( (2-\eta)c_\delta (R_{\tr}-R_1-\delta)\tan \theta - 2\Lambda(P,J)\big)\Big)
\eeq
for some $c_\delta>0$; analogous bounds follow using~\eqref{eq:Fri2} and \eqref{eq:Fri3}. 
These bounds show that the error between $u^S$ and $v^S$ decreases exponentially in the frequency, the PML width, and the tangent of the scaling angle.
 In particular, the result \eqref{e:relErrorNew} follows from using \eqref{eq:Fri4} in \eqref{e:relError1}.

An example $f$ that satisfies~\eqref{e:fProp} is the piecewise degree-eight polynomial
\begin{equation}
\label{e:8degrees}
 f(r)=r\bigg(\int_{R_1}^r (t-R_1)^3 (R_2-t)^31_{[R_1,R_2]}(t) \, dt\bigg)\bigg(\int_{R_1}^{R_2} (t-R_1)^3 (R_2-t)^3 \, dt\bigg)^{-1};
\end{equation}
see \cite[\S2]{BrPa:07}. See Figure~\ref{fig:Phi2} for plots of $\Phi_\theta(r)$ and its integral in this case.

\subsection{The main results for black-box scattering}\label{sec:introBB}

We now describe our results for black-box operators, namely, operators that are equal to the Laplacian outside a ball and are equal to some self-adjoint operator inside the ball; see \S\ref{s:blackBox} for a careful definition of these operators and associated notation. Black-box operators (a.k.a.~black-box Hamiltonians) include examples such as
scattering by Dirichlet, Neumann, and penetrable obstacles, and scattering by inhomogeneous media. Let $R_0>0$ and $P:\mc{D}\to \mc{H}$ be a black-box operator equal to minus the 
Laplacian outside $B(0,R_0)$ (i.e., $B(0,R_0)$ contains the scatterer); 
here $\mc{D}$ is the domain of the operator (see \eqref{e:blackBox}) and $\mc{H}$ is a Hilbert space coinciding with $L^2$ outside $B(0,R_0)$ (see \eqref{e:hilbert}).
Let $\chi \in C_c^\infty(\mathbb{R}^d)$ with $\chi \equiv 1$ on $B(0,R_0)$. Then, by~\cite[Theorem 4.4]{DyZw:19} (see Proposition~\ref{p:cutoff}), the cutoff resolvent
$$
\chi (P-\lambda^2)^{-1}\chi :\mc{H}\to \mc{D},\qquad -\tfrac{\pi}{2}<\Arg(\lambda)<\tfrac{3\pi}{2},
$$
is meromorphic with finite rank poles. Let $R_P(\lambda):=(P-\lambda^2)^{-1}$. 

The analogue of~\eqref{e:Lambda1} in the black box setting is 
\beq\label{eq:LambdaPJ}
\Lambda(P,J):=\limsup_{\substack{k\to \infty\\k\in J}}\frac{1}{k}\log \|\chi R_P(k)\chi\|_{\mc{H}\to \mc{H}}\in[0,\infty].
\eeq

Many black-box Hamiltonians satisfy $\Lambda(P)<\infty$. They include 
scattering by Dirichlet, Neumann, and penetrable obstacles with smooth boundaries, and scattering by inhomogeneous media with smooth wavespeeds (see \S\ref{s:examples} for details). In addition, for \emph{all} black-box Hamiltonians satisfying a polynomial bound on the number of eigenvalues of the reference operator (see, e.g.,~\cite[Equation 4.3.10]{DyZw:19}) and all $\delta>0$, there is a set $J\subset \mathbb{R}$ with $|\mathbb{R}\setminus J|\leq \delta$ such that $\Lambda(P,J)=0$; see \cite[Theorem 1.1]{LaSpWu:20} or (under an additional assumption about how close the resonances can be to the real axis)~\cite[Proposition 3]{St:01}.

 Let $R_{\tr}>R_1>R_0$ and $\Omega_{\tr}$ bounded and open with Lipschitz boundary such that $B(0,R_{\tr})\subset \Omega_{\tr}$, and define $\theta_0(P,J,R_{\tr})$ as in~\eqref{eq:theta0}. We define the complex-scaled operator $P_\theta$  corresponding to a black-box Hamiltonian as in~\eqref{e:deltaTheta} (for the more general setup, see~\eqref{e:GammaTheta}).
We then study the difference between the solutions
\begin{equation}
\label{e:PMLBox}
(P_\theta-k^2)v=\rhs,\qquad v|_{\Gamma_{\tr}}=0
\end{equation}
and 
\begin{equation}
\label{e:PBox}
(P-k^2)u=\rhs,\qquad  (\partial_r-ik)u=o(r^{\frac{1-d}{2}}) \tas r\to \infty.
\end{equation}
\begin{theorem}
\label{t:blackBoxErr}
Let $J\subset \mathbb{R}$ and $P$ be a black-box Hamiltonian with $\Lambda(P,J)<\infty$. Let $\chi \in C_c^\infty(B(0,R_1))$ with $\chi \equiv 1$ in a neighbourhood of $B(0,R_0)$, and $\eta,\e>0$. Then there are $C,k_0>0$ (independent of $R_{\tr}$ and $R_1$) such that for all $R_{\tr}>R_1+\e$, $B(0,R_{\tr})\subset \Omega_{\tr}\subset \mathbb{R}^d$ with Lipschitz boundary, $\theta_0(P,J,R_{\tr})+\e<\theta<\pi/2-\e$, $\widetilde{\rhs}\in \mc{H}$, $k>k_0$, and $k\in J$, 
the solution $v$ of \eqref{e:PMLBox} with $\rhs=\chi\widetilde \rhs$ exists, is unique, and satisfies
\beq\label{e:blackBoxErr}
\|\chi( u-v)\|_{\mc{D}}+\|(1-\chi)(u-v)\|_{H^2(B(0,R_1))}\leq C
\exp\bigg(-k\Big((2-\eta)\int_{R_1}^{R_{\tr}}\Phi_\theta(r)\,dr-2\Lambda(P,J)\Big)\bigg)
\|\widetilde \rhs\|_{\mc{H}},
\eeq
where $u$ is the solution of ~\eqref{e:PBox} with $\rhs=\chi\widetilde \rhs$.
\end{theorem}
One ingredient of the proof of Theorem~\ref{t:blackBoxErr} is the following resolvent estimate for~\eqref{e:PMLBox}.
\begin{theorem}
\label{t:blackBoxResolve}
Let $J\subset \mathbb{R}$, $P$ be a black-box Hamiltonian with $\Lambda(P,J)<\infty$, $\chi \in C_c^\infty(B(0,R_1))$ with $\chi \equiv 1$ in a neighbourhood of $B(0,R_0)$, and $\e>0$. Then there are $C,k_0>0$ 
 (independent of $R_{\tr}$ and $R_1$) such that 
the following holds. For all $R_{\tr}>R_1+\e$, $B(0,R_{\tr})\subset  \Omega_{\tr}\Subset \mathbb{R}^d$ with Lipschitz boundary,  $\theta_0(P,J,R_{\tr})+\e<\theta<\pi/2-\e$, all $\rhs\in \mc{H}$ with $\supp \rhs\subset \Omega_{\tr}$, all $k>k_0$, and $k\in J$, the solution $v$ to \eqref{e:PMLBox} exists, is unique, and satisfies
\beq\label{e:blackBoxResolve}
\|v\|_{\mc{H}(\Omega_{\tr})}+k^{-2}\|v\|_{\mc{D}(\Omega_{\tr})}\leq C\|\chi R_P(k)\chi\|_{\mc{H}\to \mc{H}}\|\rhs\|_{\mc{H}},
\eeq
where $\mc{H}(\Omega_{\tr})$ and $\mc{D}(\Omega_{\tr})$ are defined in \eqref{e:defPMLa}.
\end{theorem}

Another ingredient of the proof of Theorem~\ref{t:blackBoxErr} that may be of independent interest is that a bound on the cutoff resolvent $\chi R_P\chi$ implies the same bound on the scaled resolvent. 
\begin{theorem}
\label{t:scaledResolve}
Suppose $\chi \in C_c^\infty(B(0,R_1))$ with $\chi \equiv 1$ in a neighbourhood of $B(0,R_0)$. Then, there are $C,k_0>0$ such that for $k>k_0$, $(P_\theta-k^2)^{-1}: \mc{H}\to \mc{D}$ exists and satisfies
$$
\|(P_\theta-k^2)^{-1}\|_{\mc{H}\to \mc{H}}+k^{-2}\|(P_\theta-k^2)^{-1}\|_{\mc{H}\to \mc{D}}\leq C\|\chi R_P(k)\chi\|_{\mc{H}\to \mc{H}}.
$$
\end{theorem}
We also point out that, although it follows the same ideas as the smooth case, complex scaling with $C^{2,\alpha}$ scaling functions as described in Appendix~\ref{a:scale} is new. While the assumption that the scaling function is $C^{2,\alpha}$ is essential for the analysis in Appendix~\ref{a:scale}, and the assumption that it is $C^3$ is used to prove resolvent bounds for the free problem via 
 defect measures, other methods of complex scaling exist, see e.g.~\cite{AgCo:71, Si:78,Si:79}, and apply to, e.g., piecewise linear scaling functions.

\begin{rem}
In numerical analysis, piecewise linear scaling functions of the form $f_\theta (r)=(r-R_1)_+ \tan \theta$ are often used (see \S\ref{s:numerical}). Although our theorems do not apply to this case, we now sketch the key ingredients needed to extend our estimates to this type of scaling function. First, define a modified scaling function $\widetilde{f}_\theta(r)$ satisfying (i) $\widetilde{f}_\theta(r)=f_\theta(r)$ on $r\leq R_{\tr}$, (ii) for some $R_3>R_{\tr}$, $\theta_1>\theta$, $\widetilde{f}_{\theta}(r)\equiv r\tan \theta_1 $ for $r>R_3$, (iii) $\widetilde{f}_\theta(r)$ satisfies~\eqref{e:fProp} on $\{r>R_1\}$, and (iv) $\widetilde{f}_\theta\in C^\infty(\{r>R_1\})$. 
We would then need two results: first, the nontrapping resolvent estimate for the free problem (i.e., the analogue of Theorem~\ref{t:nontrappingScale}) and second, agreement of the scaled resolvent and unscaled resolvent away from scaling (see Proposition~\ref{p:scaledFredholmA}). Provided one has these two results, the bounds in Theorems \ref{t:blackBoxErr} and \ref{t:blackBoxResolve} follow.
 \end{rem}

\subsection{Ideas and method of proof}
PML can be understood as an adaptation (used in numerical analysis) of the method of complex scaling, which originated with~\cite{AgCo:71,BaCo:71} and was developed in its modern form for black-box scatterers in~\cite{SjZw:91} (see \S\ref{s:scaling} or~\cite[\S 4.5]{DyZw:19} for an introductory treatment of the subject). In complex scaling, $\mathbb{R}^d$ is deformed to a submanifold, $\Gamma_\theta\subset \mathbb{C}^d$ in such a way that the radiating solutions of~\eqref{e:helmholtz} deform to $L^2$ bounded solutions, $u_\theta^S$,  of the deformed problem on $\Gamma_\theta:=\{x+if_\theta(|x|)\frac{x}{|x|}\}$:
\begin{equation}
\label{e:complexScale1}
\begin{cases}(-\Delta_{\Gamma_\theta}-k^2)u^S_\theta=0&\text{on }\Gamma_\theta\setminus \overline{\Omega_-}\\
Bu^S_\theta=-B\exp(ikx\cdot a)&x\in \Gamma_-.\\
\end{cases}
\end{equation}
Moreover this deformation has the property that $u^S_\theta|_{B(0,R_1)\setminus \overline{\Omega_-}}\equiv u^S|_{B(0,R_1)\setminus\overline{\Omega_-}}$. The PML equation~\eqref{e:PML} is then the Dirichlet truncation of~\eqref{e:complexScale1}.

Because $u^S_\theta$ and $u^S$ agree on $B(0,R_1)\setminus \overline{\Omega}_-$, we are able to prove Theorem~\ref{t:expEst} by comparing $u_\theta^S$ and $v^S$. The crucial fact (see \S\ref{s:carleman}) that leads to exponentially good estimates on the error between $u_\theta^S$ and $v^S$ is that  both $u_\theta^S$ and $v^S$ are \emph{exponentially decaying} in $R>R_1$ (both in $|x|$ and $k$). Thus, the boundary values for $u_\theta ^S$ on $\Gamma_{\tr}$ are exponentially small and one can expect that $u_\theta^S$ and $v^S$ are exponentially close. Combining these exponential estimates together with a basic elliptic estimate for $v^S$ near $\Gamma_{\tr}$ and bounds on the cutoff resolvent for~\eqref{e:complexScale1}, we can complete the proof of Theorem~\ref{t:expEst}. Naively, this argument leads to an exponential improvement $\approx k \int_{R_1}^{R_{\tr}}\Phi_\theta(r)\,dr $. To obtain the rate $\approx 2k\int_{R_1}^R\Phi_\theta(r)\,dr$, one must then use that errors near the truncation boundary only propagate with exponential damping toward $R_1$. This leads to the second factor in our bound; {see the discussion in the caption of Figure \ref{fig:idea}.}

To understand the appearance of the function $\Phi_\theta(r)$, we recall that the semiclassical principal symbol of $-\hbar^2\Delta_\theta-1$ (where $\hbar:=1/k$) is
$$
p(r,\xi_r,\omega,\xi_\omega):=\Big(\frac{\xi_r}{1+if_\theta'(r)}\Big)^2+\frac{|\xi_\omega|_{S^{d-1}}^2}{(r+if_\theta(r))^2}-1.
$$
Replacing $\xi_r$ by the corresponding operator $\hbar D_r$, $(D_r:=-i\partial_r)$, one obtains a family of ODEs in $r$ depending on $|\xi_\omega|_{S^{d-1}}^2$. The infinitesimal growth/decay of the two possible solutions to this ODE at a point $r$ is then given by the imaginary part of the roots, $s_+$ and $s_-$, of the polynomial $\xi_r\mapsto p(r,\xi_r,\omega,\xi_\omega)$. The function $\Phi_\theta(r)$ is then given by
\beq\nonumber
\Phi_\theta(r)= \inf_{|\xi_\omega|\geq 0} \min \big\{
|\Im s_+|, |\Im s_-|
\big\};
\eeq
thus it is the smallest possible decay obtained in this way (see Lemma \ref{l:carleman1} for more details).

 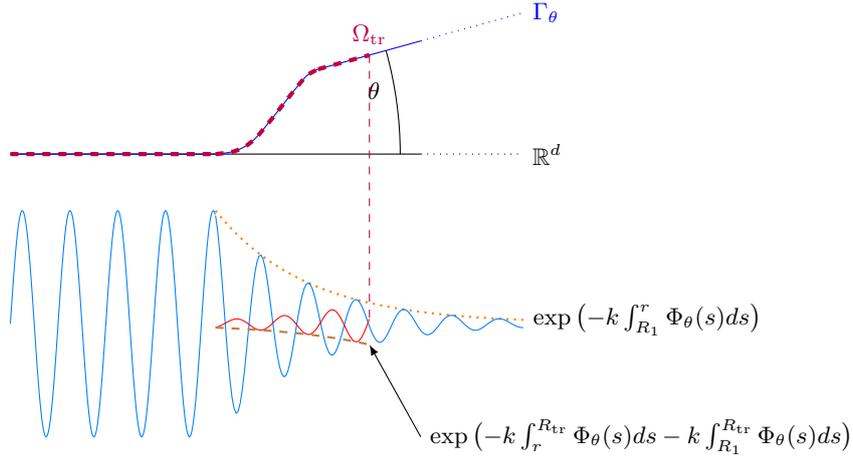
\begin{figure}
 \begin{tikzpicture}
\def \w{1.8};
\def \h{1};
\def \rate{1.25}
\begin{scope}[scale=1.5]
\draw [dotted]({2*\w},0)--({2.5*\w},0)node[right]{$\mathbb{R}^d$};
\draw (0,0)--({2*\w},0);
\draw[blue] (0,0)--({1*\w},0);
\draw [blue]plot[smooth]coordinates{({1*\w},0)({1.1*\w},.03)({1.2*\w},.15)({1.4*\w},.6)({1.45*\w},.7)  ({1.5*\w},.75)};
\draw [blue,dotted]({2*\w},1)--({2.5*\w},1.25)node[right]{$\Gamma_\theta$};
\draw [blue]({1.5*\w},.75)--({2*\w},1);
\draw ({1.9*\w},0) arc(0:atan(1/(2*\w)):{1.9*\w});
\node at (10:{1.8*\w}){$\theta$};
\draw[purple, ultra thick,dashed] (0,0)--({1*\w},0);
\draw [purple, ultra thick,dashed]plot[smooth]coordinates{({1*\w},0)({1.1*\w},.03)({1.2*\w},.15)({1.4*\w},.6)({1.45*\w},.7)  ({1.5*\w},.75)};
\draw [purple, ultra thick,dashed]({1.5*\w},.75)--({1.75*\w},.875)node[above]{$\Omega_{\tr}$};
{\draw [purple, dashed] ({1.75*\w},.875)--({1.75*\w},-1.5);}
{\draw[color=orange,dotted, thick, domain=1*\w:2.5*\w,samples=50]   plot (\x,{exp(-\rate*(\x-\w))-1.5})   node[right] {\textcolor{black}{$\exp\big({-k\int_{R_1}^r{\Phi_{\theta}}(s)ds}\big)$}};}
{\draw[color=azul, domain=0:1*\w,samples=200]   plot (\x,{\h*sin( (15 *\x)  r )-1.5})   node[right] {};}
{\draw[color=azul, domain=1*\w:2.5*\w,samples=150]   plot (\x,{(\h*sin( (15 *\x)  r ))*exp(-\rate*(\x-\w))-1.5})   node[right] {};}
{\draw[color=red, domain=1*\w:1.75*\w,samples=150]   plot (\x,{(\h*sin( (15 *(\x+1.75*\w))  r ))*exp(-\rate*.75*\w)*exp(\rate*(\x-1.75*\w))-1.5})   node[right] {};}
{\draw[color=brown, thick, dashed, domain=1*\w:1.75*\w,samples=50]   plot (\x,{(-1*exp(-\rate*.75*\w)*exp(\rate*(\x-1.75*\w))-1.5})   node[below] {};}
{\draw[->] ({2*\w},-2.5) node[right] {\textcolor{black}{$\exp\big({-k\int_{r}^{R_{\tr}}{\Phi_{\theta}}(s)ds-k\int_{R_1}^{R_{\tr}}{\Phi_\theta}(s)ds}\big)$}}--({1.75*\w},{-1*exp(-\rate*.75*\w)-1.5});}
\end{scope}
\end{tikzpicture}
 \caption{The figure shows a wave $u_\theta^S$ (in blue) propagating toward $\Gamma_{\tr}$ from near the obstacle $\Omega_-$.  The wave $u^S_\theta$ decays exponentially as it enters the scaling region (where $\Gamma_\theta\neq \mathbb{R}^d$); this exponential decay is shown in the orange dotted line. The wave $v_\theta$ then reflects off $\Gamma_{\tr}$. There are two possible solutions: one exponentially growing towards the interior and one exponentially decaying towards the interior. Fortunately, the solution exponentially growing towards the interior corresponds to the exponentially decaying (away from the interior) $u^S_\theta$ and this solution does not produce an error. The exponentially decaying (towards the interior) part of $v_\theta$, however, does produce an error in the interior. This solution is again exponentially damped as it travels toward the interior; this solution is shown in red and the decay rate is shown by the brown dashed line. }\label{fig:idea}
 \end{figure}

\subsection{Immediate implications for the numerical analysis of the finite-element method with PML truncation}
\label{s:numerical}

There have been two recent papers on the $k$-explicit analysis of the $h$-version of the finite-element method (FEM) applied to the Helmholtz equation with PML truncation (recall that in the $h$-version of the FEM, convergence is achieved by decreasing the meshwidth $h$ whilst keeping the polynomial degree $p$ constant).
The paper \cite{LiWu:19} considers the Helmholtz equation in free space (i.e., with no scatterer) and 
$f_\theta(r)=\sigma_0 (r-R_1)_+$ (where $x_+= x$ for $x\geq 0$ and $=0$ for $x<0$). \cite{ChGaNiTo:18} considers the Helmholtz equation posed in the exterior  of a smooth, starshaped Dirichlet obstacle with $f_\theta(r) = r\widetilde{\sigma}/k$ with $\widetilde{\sigma}\in C^1$ (and independent of $k$).

For the $h$-FEM applied to the Helmholtz equation, a fundamental question is:~how must $h$ decrease with $k$ to maintain accuracy of the Galerkin solution as $k$ increases?
Both \cite{LiWu:19} and \cite{ChGaNiTo:18} prove that, for the PML problems they consider, the answer is the same as for the respective Helmholtz problems truncated with the exact outgoing Dirichlet-to-Neumann map.

Indeed, \cite[Theorem 4.4]{LiWu:19} proves 
that if the approximation spaces consist of piecewise linear polynomials and $hk^{3/2}$ is sufficiently small, then
the Galerkin approximation, $v_h$, to $v$ satisfying \eqref{e:PMLBox} (with $P_\theta=-\Delta_\theta$) exists, is unique, and satisfies
\beq\label{eq:FEM}
\|\nabla(v-v_h)\|_{L^2} + k \|v-v_h\|_{L^2} \leq C hk^{3/2} \|\rhs\|_{L^2}
\eeq
(cf.~the results in \cite{LaSpWu:19a} for the Helmholtz problem with the exact outgoing Dirichlet-to-Neumann map). Furthermore, with piecewise polynomial of degree $p$, if $h^p k^{p+1}$ is sufficiently small, then~\cite[Theorem 5.4]{ChGaNiTo:18} proves that, for the exterior Dirichlet problem with starshaped $\Omega_-$, the Galerkin solution exists, is unique, and satisfies a quasioptimal error estimate with quasioptimality constant independent of $k$  (cf.~the results in \cite{MeSa:10, MeSa:11} for the Helmholtz exterior Dirichlet problem truncated with the exact outgoing Dirichlet-to-Neumann map).
 
Combining the results in the present paper with the FEM analysis in \cite{LiWu:19}, we immediately have that the results of \cite{LiWu:19} (i.e., existence, uniqueness, and the error bound \eqref{eq:FEM} for the Galerkin solution when  $hk^{3/2}$ is sufficiently small) extend to the FEM solution of 
any of the Helmholtz problems in \S\ref{s:examples}, provided that (i) $f_\theta(r)$ satisfies the assumptions in \S\ref{s:1.2}, (ii)
\beqs
\|\chi R_P(k)\chi\|_{\mc{H}\to \mc{H}} \leq C/k \quad \tfa k\geq k_0
\eeqs
(which occurs, for example, when the problem is nontrapping) and (iii) the domain of the PML problem $\mc{D}(\Omega_\tr)$, defined by \eqref{e:defPMLa}, equals $H^2(\Omega_\tr)$. 
Indeed, Theorem \ref{t:blackBoxResolve} is a generalisation (modulo the differences in scaling functions) of \cite[Theorem 3.1]{LiWu:19} 
and Theorem \ref{t:blackBoxErr} is a generalisation of \cite[Theorem 3.7]{LiWu:19}.

The results in \cite{ChGaNiTo:18}, however, rely crucially on the fact that $f_\theta(r) \sim 1/k$ (e.g., the comparison with the sponge layer in \cite[\S5]{ChGaNiTo:18} fails if $f_\theta(r) \gg 1/k$); therefore, the results of the present paper cannot be combined with those in \cite{ChGaNiTo:18}.
  We expect the error in the PML solution when $f_\theta(r)\sim1/k$ to be only $O(1)$ as $k\tendi$. This is in contrast to the exponentially small error when $f_\theta(r)\sim 1$ (as shown in Theorem~\ref{t:blackBoxErr}).
  
Finally, we highlight that the results of the present paper are used in the follow-up paper \cite{GLSW1} to obtain results about the $hp$-FEM applied to the Helmholtz PML problem; these results are analogous to those obtained in \cite{LaSpWu:21} about the Helmholtz problem truncated with the exact outgoing Dirichlet-to-Neumann map.
 
\subsection{Outline of the paper} 
\S\ref{s:blackBox} recaps the framework of black-box scattering. 
\S\ref{s:scaling} recaps the method of complex scaling and proves 
Theorem \ref{t:scaledResolve}. 
\S\ref{s:elliptic} proves elliptic estimates in the scaling region.
\S\ref{s:proofs1} proves Theorems  \ref{t:blackBoxErr} and \ref{t:blackBoxResolve} (i.e., the main results in the black-box framework).
\S\ref{s:relError} proves the bound on the relative error in Theorem \ref{t:expEst} for the plane-wave scattering problem.
\S\ref{s:rough_res} proves the nontrapping estimate on the free resolvent for the scaled problem with $C^3$ scaling function.
\S\ref{a:scale} proves results about complex scaling with $C^{2,\alpha}$ scaling function.
\S\ref{app:SCA} recalls results from semiclassical analysis.
\S\ref{app:Phi} proves Lemma \ref{lem:Phi} (i.e., properties of $\Phi_\theta(r)$).

\subsection*{Acknowledgements} The authors thank Maciej Zworski for several helpful conversations and the anonymous referees for their constructive comments. JG was supported by EPSRC grant EP/V001760/1, and 
DL and EAS were supported by EPSRC grant EP/R005591/1.

\section{Black-box Hamiltonians}
\label{s:blackBox}

Throughout this paper we work in the setting of black-box Hamiltonians (see~\cite[\S 4.1]{DyZw:19}); we now review this notion.

 Let $\mc{H}$ be a complex Hilbert space with the orthogonal decomposition
\begin{equation}
\label{e:hilbert}
\mc{H}=\mc{H}_{R_0}\oplus L^2(\mathbb{R}^d\setminus B(0,R_0)),
\end{equation}
where $\mc{H}_{R_0}$ is an arbitrary Hilbert space.
We take the standard convention that if $\chi\in L^\infty(\mathbb{R}^d)$ with $\chi\equiv c_0\in \mathbb{C}$ on $B(0,R_0)$, then for $u\in \mc{H}$ with $u=u|_{B(0,R_0)}+u|_{\mathbb{R}^d\setminus B(0,R_0)}$ , $u|_{R_0}\in \mc{H}_{R_0}$, and $u|_{\mathbb{R}^d\setminus B(0,R_0)}\in L^2(\mathbb{R}^d\setminus B(0,R_0))$,
$$
\chi u=c_0\big(u|_{B(0,R_0)}\big)+(\chi|_{\mathbb{R}^d\setminus B(0,R_0)})\big(u_{\mathbb{R}^d\setminus B(0,R_0)}\big)\in \mc{H}.
$$

We say that $P$ is a \emph{black-box Hamiltonian} if, for $\mc{H}$ as in~\eqref{e:hilbert}, $P:\mc{H}\to \mc{H}$ is an unbounded self-adjoint operator with domain $\mc{D}\subset \mathcal{H}$ such that 
\begin{equation}
\label{e:blackBox}
\begin{gathered}
1_{\mathbb{R}^d\setminus B(0,R_0)}\mc{D}\subset H^2(\mathbb{R}^d\setminus B(0,R_0)),\quad
1_{\mathbb{R}^d\setminus B(0,R_0)}(Pu)=-\Delta \big(u|_{\mathbb{R}^d\setminus B(0,R_0)}\big),\\
\big\{ u\in H^2(\mathbb{R}^d)\,:\, u|_{B(0,R_0+\e)}\equiv 0\text{ for some }\e>0\big\}\subset \mc{D},\\1_{B(0,R_0)}(P+i)^{-1}:\mc{H}\to \mc{H}\text{ is compact}.
\end{gathered}
\end{equation}
We equip $\mc{D}$ with the norm
\beq\label{eq:normD}
\|u\|_{\mc{D}}^2=\|u\|_{\mc{H}}^2+\|Pu\|_{\mc{H}}^2,\qquad u\in \mc{D},
\eeq
and define $\mc{D}^s$ for $s\in [0,1]$ by interpolation between $\mc{H}$ and $\mc{D}$.
We also define 
$$
\begin{gathered}
\mc{H}_{\comp}:=\big\{ u\in \mc{H}\,:\, u|_{\mathbb{R}^d\setminus B(0,R_0)}\in L^2_{\comp}\big\},\qquad \mc{H}_{\loc}:=\mc{H}_{R_0}\oplus L^2_{\loc}(\mathbb{R}^d\setminus B(0,R_0)),\\
\mc{D}_{\comp}:=\mc{D}\cap \mc{H}_{\comp},\qquad \mc{D}_{\loc}:=\big\{u\in \mc{H}_{\loc}\,:\, \chi u\in \mc{D},\,\text{for all}\,\chi\in C_c^\infty(\mathbb{R}^d),\,\chi\equiv 1\text{ on }B(0,R_0)\big\}.
\end{gathered}
$$
We now recall some properties of the resolvent of a black-box Hamiltonian.
\begin{prop}[Theorem 4.4 \cite{DyZw:19}]
\label{p:cutoff}
Suppose that $P$ is a black-box Hamiltonian. Then,
$$
R_P(\lambda):=(P-\lambda^2)^{-1}:\mc{H}\to \mc{D} \text{ is meromorphic  for }\Im \lambda >0
$$
with finite rank poles. Moreover, for all $\chi \in C_c^\infty(\mathbb{R}^d)$ with $\chi \equiv 1$ on $B(0,R_0)$, 
$$
R_P(\lambda):\mc{H}_{\comp}\to \mc{D}_{\loc}\text{ is meromorphic for }  -\tfrac{\pi}{2}<\Arg(\lambda)<\tfrac{3\pi}{2},
$$
with finite rank poles.
\end{prop}

\subsection{Examples}
\label{s:examples}
\begin{enumerate}
\item[1.] {\bf Scattering by a Dirichlet obstacle.} Let $\Omega_-\subset \overline{B(0,R_0)}$ be an open set such that $\Gamma_-$ is Lipschitz and $\Omega_+:=\Rea^d\setminus \overline{\Omega_-}$ is connected. If $\mc{H}=L^2(\Omega_+)$ and
$$
\mc{D}=\big\{ u\in H^1(\Omega_+)\,:\, u|_{\Gamma_-}=0,\\, -\Delta u\in L^2(\Omega_+)\big\},
$$
then $P=-\Delta$ is a black-box Hamiltonian by \cite[Lemma 2.1]{LaSpWu:20}. If $\Gamma_-$ is $C^\infty$ then by~\cite{Bu:98,Vo:00} $\Lambda(P)<\infty$. If $\Omega_-$ is nontrapping~\cite{Va:75,MeSj:82}~\cite[Theorem 4.43]{DyZw:19}, or $\Omega_-$ is star shaped~\cite{Mo:75,ChMo:08}, then $\Lambda(P)=0$.
\item[2.] {\bf Scattering by a Neumann obstacle.} 
Let $\Omega_-\subset \overline{B(0,R_0)}$ be an open set such that $\Gamma_-$ is Lipschitz and $\Omega_+:=\Rea^d\setminus \overline{\Omega_-}$ is connected.
If $\mc{H}=L^2(\Omega_+)$ and 
$$
\mc{D}=\big\{ u\in H^1(\Omega_+)\,:\, (\partial_\nu u)|_{\Gamma_-}=0,\,\, -\Delta u\in L^2(\Omega_+)\big\},
$$
then $P=-\Delta$ is a black-box Hamiltonian by \cite[Lemma 2.1]{LaSpWu:20}. If $\Gamma_- $ is $C^\infty$ then by~\cite{Bu:98,Vo:00} $\Lambda(P)<\infty$. If $\Omega_-$ is nontrapping~\cite{Va:75,MeSj:82}~\cite[Theorem 4.43]{DyZw:19}, or $\Gamma_- \in C^{3}$ and $\Omega_-$ is convex~\cite{Mo:75}, then $\Lambda(P)=0$.
\item[3.] {\bf Scattering by inhomogeneous media.} Let $\alpha>0$, $A\in C^{2,\alpha}(\mathbb{R}^d;\mathbb{M}_{d\times d})$ be real, symmetric, and positive definite, $b\in C^{1,\alpha}(\mathbb{R}^d; \mathbb{R}^d)$, and $c\in C^{0,\alpha}(\mathbb{R}^d;\mathbb{R})$ with $A|_{\mathbb{R}^d\setminus B(0,R_0)}\equiv \Id$, $\supp b,\supp c\subset B(0,R_0)$. If $\mc{H}=L^2(\mathbb{R}^d)$ and $\mc{D}=H^2(\mathbb{R}^d)$, then
$$
P= \partial_i A^{ij}(x)\partial_j + (b^i(x)D_i +D_ib^i(x))+c(x).
$$
 is a black-box Hamiltonian.
If the Hamiltonian flow for $A^{ij}\xi_i\xi_j$ is nontrapping, then $\Lambda(P)=0$~\cite{GaSpWu:20}. Moreover,  if $A^{ij},b^i,c\in C^\infty$, then $\Lambda(P)<\infty$~\cite{Bu:98,Vo:00}. 
We note that we could combine this example with either of Examples 1 and 2, with the result that scattering by an inhomogeneous media contained either a Dirichlet or Neumann obstacle is covered by the black-box framework.
\item[4.] {\bf Scattering by a penetrable obstacle.}
Let $\Omega_-\subset \overline{B(0,R_0)}$ be an open set such that $\Gamma_-$ is Lipschitz and $\Omega_+:=\Rea^d\setminus \overline{\Omega_-}$ is connected.
Let $A=(A_-,A_+)$ with $A_\pm \in C^{0,1} (\Omega_\pm, \mathbb{M}_{d\times d})$ real, symmetric, positive definite, and such that $A|_{\mathbb{R}^d\setminus B(0,R_0)}\equiv \Id$.
Let $c\in L^\infty(\Omega_-)$ be such that 
$c_{\rm min}\leq c \leq c_{\rm max}$ with $0<c_{\rm min}\leq c_{\rm max}<\infty$, and $\beta>0$.
Let $\nu$ be the unit normal vector field on $\partial \Omega_-$  pointing from $\Omega_-$ into $\Omega_+$, and let 
$\partial_{\nu,A}$ the corresponding conormal derivative from either $\Omega_-$ or $\Omega_+$.
If $\mc{H}=L^2(\Rea^d)$ and
\begin{align*}\nonumber
\mathcal D :=&\Big\{ v= (v_-,v_+) \quad\text{ where }\quad v_- \in 
H^1(\Omega_-), \quad\nabla\cdot(A_-\nabla v_-) \in L^2(\Omega_-),
\\ \nonumber
&\quad v_+ \in
H^1\big(\mathbb{R}^d\setminus \overline{\Omega_-}\big), \quad\nabla\cdot(A_+\nabla v_+)\big)\in 
L^2\big(\mathbb{R}^d\setminus \overline{\Omega_-}\,\big),\\ \nonumber
 & \quad  v_+ = v_-\quad\tand\quad \partial_{\nu,A_+} v_+ = \beta\, \partial_{\nu,A_-} v_- \quad\text{ on } \partial \Omega_- 
\Big\},
 \label{eq:domain_transmission}
\end{align*}
then
\beqs
P v:=- \Big(c^{2}\nabla\cdot(A_-\nabla  v_{-}),\nabla\cdot(A_+\nabla v_{+})\Big),
\eeqs
is a black-box Hamiltonian by \cite[Lemma 2.4]{LaSpWu:21}.
If $\partial\Omega_-\in C^\infty$ and  $A_{\pm},c\in C^\infty$, then
$\Lambda(P)<\infty$~\cite{Bel:03}.
\end{enumerate}

\section{Complex scaling and perfectly matched layers}
\label{s:scaling}

In \S\ref{sec:31} we review the method of complex scaling;
as discussed in the introduction, this plays a crucial role in our analysis of PML. 
In \S\ref{sec:32} we prove Theorem \ref{t:scaledResolve}. In \S\ref{sec:33} we formulate the PML problem in the black-box framework using the language of complex scaling.

\subsection{The scaled operator}\label{sec:31}
Let $R_2>R_1>R_0>0$ and $P$ a black-box Hamiltonian as in~\eqref{e:blackBox}. Let $f_\theta\in C^{2,\alpha}([0,\infty);\mathbb{R})$ satisfy
\beq\label{e:fProp2}
f_\theta(r)\equiv 0\text{ on } r\leq R_1,\qquad f_\theta'(r)\geq 0,\qquad f_\theta(r)=r\tan \theta \text{ on }r \geq R_2,
\eeq
and define $\Delta_\theta$ as in~\eqref{e:deltaTheta}.
The theory of complex scaling when $f_\theta$ is smooth is standard (see~\cite[\S4.5]{DyZw:19}) but when $f_\theta\in C^{2,\alpha}$, some modifications to the standard proofs are required. We record the main outputs of this theory for the operator $\Delta_\theta$ here and prove these results (plus necessary intermediate ones) for $ C^{2,\alpha}$ scalings in Appendix~\ref{a:scale}.

We now define the complex-scaled operator for a black-box Hamiltonian. 
With $\chi \in C_c^\infty(B(0,R_1))$ equal to 1 on $B(0,R_0)$, define $P_\theta:\mc{H}\to \mc{H}$ with domain $\mc{D}$ by
\begin{equation}
\label{e:defPa}
\begin{gathered}
P_\theta u= P(\chi u)+ (-\Delta_\theta)((1-\chi )u).
\end{gathered}
\end{equation}
(Strictly speaking, the domain and Hilbert space for $P_\theta$ involve the scaled manifold; see \eqref{e:defP} below. However, these can be naturally identified with $\mc{D}$ and $\mc{H}$, and so in the main body of this paper we make this identification to reduce notation.)
The following result is proved in Proposition \ref{p:scaledFredholm}.

\begin{prop}
\label{p:scaledFredholmA}
Let $P_\theta$, $\mc{D}$, and $\mc{H}$, $0\leq \theta <\pi/2$ be as in~\eqref{e:defPa}. If $\Im (e^{i\theta}\lambda)>0$, then 
$$
P_\theta -\lambda^2:\mc{D}\to \mc{H}
$$
is a Fredholm operator of index zero. Moreover,  for $R_0<R_1$ and $\chi \in C_c^\infty(B(0,R_1))$ with $\chi\equiv 1$ on $B(0,R_0)$, 
$$
{\indicator_{B(0,R_1)}}(P-\lambda^2)^{-1}{\indicator_{B(0,R_1)}} ={\indicator_{B(0,R_1)}}(P_\theta-\lambda^2)^{-1}{\indicator_{B(0,R_1)}},\qquad \Im (e^{i\theta}\lambda)>0.
$$
\end{prop}

We also need the following nontrapping estimate on the free resolvent of the scaled problem; since the proof of this result is somewhat long and technical, we postpone the proof to \S\ref{s:rough_res}.
\begin{theorem}
\label{t:nontrappingScale}
Suppose that $f_\theta$ is as in~\eqref{e:fProp2} and $0<\theta<\pi/2$. Then for all $\epsilon>0$ there are $C>0$ and $k_0>0$  such that for $k>k_0$, $-1\leq m\leq 0$, $0\leq s\leq 2$, {and $\epsilon\leq \theta \leq \pi/2 - \epsilon$}, 
$$
\|(-\Delta_\theta -k^2)^{-1}\|_{H^{m}\to H^{s+m}}\leq Ck^{s-1}.
$$
\end{theorem}

\subsection{From cutoff resolvent estimates estimates to scaled resolvent estimates}\label{sec:32}

We now prove Theorem \ref{t:scaledResolve}; i.e., we show that an estimate on the cutoff resolvent, $\chi R_P(\lambda)\chi$, can be transferred to one on $(P_\theta -\lambda^2)^{-1}$. Since most estimates in the literature are stated for the cutoff resolvent, this allows us to directly transfer those estimates to the scaled operator.
\begin{lem}
\label{l:truncToScaled}
Suppose there are $R>R_0$ and $g:[0,\infty)\to (0,\infty]$ such that, for all $\rho \in C_c^\infty(B(0,R);[0,1])$ with  $\rho \equiv 1$ in a neighborhood of $B(0,R_0)$ and $k>k_0$,
\beq\label{eq:ass1}
\|\rho R_{P}(k)\rho\|_{\mc{H}\to \mc{H}}\leq \resolvent (k).
\eeq
Then, given $\e>0$, there exists $C>0$ such that, for $\e<\theta<\pi/2-\e$, $k>k_0$, and $0\leq s\leq 1$,
$(P_\theta-k^2)^{-1}: \mc{H}\to \mc{D}$ exists and satisfies
\beq\label{e:truncToScaled}
\|(P_\theta-k^2)^{-1}\|_{\mc{H}\to \mc{D}^{s}}\leq Ck^{2s}\resolvent (k)
\eeq
\end{lem}
Theorem~\ref{t:scaledResolve} follows from Lemma \ref{l:truncToScaled} taking $\resolvent (k)=\|\chi R_{P}(k)\chi\|_{\mc{H}\to \mc{H}}$. 

\begin{rem}\label{rem:resolventLowerBound}
Note that one always has $\|\rho R_{P}(k)\rho\|_{\mc{H}\to \mc{H}}\geq ck^{-1}$. Indeed, 
given $\rho \in C_c^\infty(B(0,R);[0,1])$ with  $\rho \equiv 1$ in a neighborhood of $B(0,R_0)$,
let $\chi\in C_c^\infty(B(0,R)\setminus B(0,R_0))$ with $\supp \chi \subset \{\rho \equiv 1\}$.
 Let $u= \chi e^{ik x\cdot a}$ for some $a\in \mathbb{R}^d$ with $|a|=1$. Then,
 for some $C_1,C_2>0$ (independent of $k$),
$$
\|(P-k^2)u\|_{\mc{H}}=\|(-\Delta-k^2)u\|_{L^2}=\|[-\Delta,\chi]e^{ikx\cdot a}\|_{L^2}=\|(2ik\langle \partial \chi, a\rangle +\Delta \chi)e^{ikx\cdot a}\|_{L^2}\leq C_1 k,
$$
and $\| u\|_{\mc{H}}\geq C_2$. Therefore, since $\supp \chi \subset \{\rho \equiv 1\}$,
$$
\|\rho R_P(k)\rho (P-k^2)u\|_{\mc{H}}=  \|\rho u\|_{\mc{H}}\geq \|u\|_{\mc{H}}\geq (C_2/C_1)k^{-1}\|(P-k^2)u\|_{\mc{H}}.
$$
\end{rem}
\begin{proof}[Proof of Lemma \ref{l:truncToScaled}]
By Proposition \ref{p:scaledFredholmA}, $P_\theta-k^2$ is Fredholm of index zero; thus existence of $(P_\theta-k^2)^{-1}$ follows from injectivity, and injectivity follows once we establish the a priori bound \eqref{e:truncToScaled}.

The idea of the proof of \eqref{e:truncToScaled} is to approximate $(P_\theta-k^2)^{-1}$ away from the black-box using the free scaled resolvent, and near the black-box using the unscaled resolvent.
Let $\widetilde{R}:=\min(R,R_1)$,  $f\in \mc{H}$ and $\chi_0,\chi_1\in C_c^\infty(B(0,\widetilde{R}))$ with $\chi_1\equiv 1$ in a neighborhood of $\supp \chi_0$ and $\chi_0\equiv 1$ in a neighborhood of $B(0,R_0)$. Let $u=(P_\theta-k^2)^{-1}\rhs$ and $v=(-\Delta_\theta-k^2)^{-1}(1-\chi_1)\rhs$. Then, we define
$$
(P_\theta-k^2)(u-(1-\chi_0)v)=\rhs+[-\Delta,\chi_0]v-(1-\chi_0)(1-\chi_1)\rhs=\chi_1 \rhs+[-\Delta,\chi_0]v =: \widetilde{\rhs}
$$
and observe that $\widetilde{\rhs}$ satisfies $\supp \widetilde{\rhs}\Subset B(0,\widetilde{R})$.
Let
$\widetilde{u}=(P_\theta-k^2)^{-1}\widetilde{\rhs}$ so that $u = \widetilde{u} + (1-\chi_0)v$.
By Theorem \ref{t:nontrappingScale},
\begin{equation}
\label{e:est1}\|u-\widetilde{u}\|_{\mc{H}} =\|(1-\chi_0)(-\Delta_\theta-k^2)^{-1}(1-\chi_1)\rhs\|_{\mc{H}}\leq Ck^{-1}\|\rhs\|_{\mc{H}}.
\end{equation}
Therefore, we need only estimate $\widetilde{u}$.
By Theorem \ref{t:nontrappingScale} with $m=0$ and $s=1$,
\begin{equation}
\label{e:est2}
\|\widetilde{\rhs}\|_{\mc{H}}\leq \|\chi_1 \rhs\|_{\mc{H}}+\|[-\Delta,\chi_0]v\|_{L^2}\leq \|\rhs\|_{\mc{H}}+\|[-\Delta,\chi_0](-\Delta_\theta-k^2)^{-1}(1-\chi_1)\rhs\|_{L^2}\leq C\|\rhs\|_{\mc{H}}.
\end{equation}
{Since $\supp \widetilde{\rhs}\Subset B(0,\widetilde{R})$}, there is $\rho \in C_c^\infty(B(0,\widetilde{R}))$ such that $\rho \equiv 1$ on $\supp \widetilde{\rhs}\cup B(0,R_0)$ and hence
$$
\widetilde{u}=(P_\theta-k^2)^{-1}\rho \widetilde{\rhs}.
$$ 
Let $\rho_1\in C_c^\infty(B(0,\widetilde{R}))$ with $\rho_1\equiv 1$ in a neighborhood of $\supp \rho$. Then,
$$
(-\Delta_\theta-k^2)(1-\rho_1)\widetilde{u}=(1-\rho_1)\rho \widetilde{\rhs}-[-\Delta,\rho_1]\widetilde{u}=[\rho_1,-\Delta]\widetilde{u},
$$
and thus
$$
(1-\rho_1)\widetilde{u}=(-\Delta_\theta-k^2)^{-1}[\rho_1,-\Delta]\widetilde{u}.
$$
Therefore, for $\rho_2\in C_c^\infty(B(0,\widetilde{R})\setminus B(0,R_0))$ with $\rho_2\equiv 1$ on $\supp \partial \rho_1$, and $\rho_3\in C_c^\infty(B(0,\widetilde{R}))$ with $\rho_3\equiv 1$ on $\supp \rho_2\cup \supp \rho_1$,
\begin{align*}
\|(1-\rho_1)\widetilde{u}\|_{L^2}=\|(-\Delta_\theta-k^2)^{-1}[\rho_1,-\Delta]\rho_2\widetilde{u}\|_{L^2}\leq  C\|\rho_2\widetilde{u}\|_{L^2}&=C\|\rho_2\rho_3(P_\theta-k^2)^{-1}\rho_3\rho \widetilde{\rhs}\|_{\mc{H}}\\
&=C\|\rho_2\rho_3R_{P}(k)\rho_3\rho \widetilde{\rhs}\|_{\mc{H}}\\
&\leq C\resolvent (k)\|\rho\widetilde{\rhs}\|_{\mc{H}}\leq C\resolvent (k)\|\widetilde{\rhs}\|_{\mc{H}},
\end{align*}
where we have used Theorem \ref{t:nontrappingScale} (with $m=-1$ and $s=1$),  Proposition \ref{p:scaledFredholmA}, and the assumption \eqref{eq:ass1}.
Putting this together with
$$
\|\rho_1\widetilde{u}\|_{\mc{H}}=\|\rho_1\rho_3\widetilde{u}\|_{\mc{H}}\leq \|\rho_3(P_\theta-k^2)^{-1}\rho_3\rho\widetilde{\rhs}\|_{\mc{H}}=\|\rho_3R_{P}(k)\rho_3\rho\widetilde{\rhs}\|_{\mc{H}}\leq \resolvent (k)\|\rho\widetilde{\rhs}\|_{\mc{H}}\leq \resolvent (k)\|\widetilde{\rhs}\|_{\mc{H}},
$$
we have
$$
\begin{aligned}
\|\widetilde{u}\|_{\mc{H}}&\leq C\resolvent (k)\|\widetilde{\rhs}\|_{\mc{H}}.
\end{aligned}
$$
Finally, using~\eqref{e:est1} and~\eqref{e:est2} and the fact that $\resolvent (k)>ck^{-1}$ (by Remark \ref{rem:resolventLowerBound}) completes the proof of \eqref{e:truncToScaled} for $s=0$. 

By the definition of $\|\cdot\|_{\cD}$ \eqref{eq:normD},
to obtain the estimate for $s=1$, we need to bound $\|Pu\|_{\cH}$.
Let $\newchi_i\in C_c^\infty(B(0,R_1))$, $i=-1,0,1$ with $\newchi_i \equiv 1$ in a neighborhood of $B(0,R_0)$, and $\supp \newchi_{i}\subset \{\newchi_{i+1}\equiv 1\}$. 
It is then sufficient to bound $\|P\psi_1 u\|_{\cH}$ and $\|(1-\psi_0)u\|_{H^2}$.
Now, since $P=P_\theta$ on $B(0,R_1)$,
\beq\label{eq:BHM1}
P{\newchi_1} u=k^2\newchi_1 u+\newchi_1 \rhs+[-\Delta,\newchi_1]u,
\eeq
and 
$$
(-\Delta_\theta-k^2) (1-\newchi_0)u= [\Delta,\newchi_0]u+(1-\newchi_0)\rhs.
$$
A priori, we only have $u\in \mc{H}$, and thus the right-hand side of the last equation is, a priori, only in $H^{-1}$.
By two applications of Theorem \ref{t:nontrappingScale} (the first with $m=-1$ and $s=2$ and the second with $m=0$ and $s=1$),
\beq\label{eq:BHM2}
\|(1-\newchi_0)u\|_{H^1}\leq Ck \|
u\|_{\mc{H}}+C\|\rhs\|_{\mc{H}}.
\eeq
Since 
\beqs
\|[-\Delta,\newchi_1]u\|_{L^2}\leq C \|(1-\newchi_0)u\|_{H^1}, 
\eeqs
using 
these last two inequalities
in \eqref{eq:BHM1}, along with \eqref{e:truncToScaled} with $s=0$ (which we established above), we have
\beq\label{eq:comb1}
\|P\newchi_1u\|_{\mc{H}}\leq C\big( k^2\|u\|_{\mc{H}}+\|\rhs\|_{\mc{H}}\big)\leq Ck^2\big(\resolvent (k)+k^{-2}\big)\|\rhs\|_{\mc{H}},
\eeq
which is the required estimate on $\|P\newchi_1u\|_{\mc{H}}$; we therefore only need to bound 
$\|(1-\psi_0)u\|_{H^2}$.
If we can show that $\Delta_\theta((1-\psi_0)u)\in L^2$, then, by elliptic regularity (since $\Delta_\theta$ is elliptic by \cite[Theorem 4.32]{DyZw:19}), 
\beq\label{eq:ER1}
\|(1-\newchi_0)u\|_{H^2}\leq C\big(\|\Delta_\theta (1-\newchi_0)u\|_{L^2}+\|u\|_{\mc{H}}\big),
\eeq
with a uniform constant for $\theta \in [\e,\pi/2-\e]$.
Now
\beq\label{eq:combine2a}
\Delta_\theta((1-\psi_0)u)= (1-\psi_0)(k^2 u + \rhs) - [\Delta_\theta,\psi_0]u,
\eeq
\beq\label{eq:combine1}
\| [\Delta_\theta,\psi_0]u\|_{L^2} = \| [\Delta,\psi_0]u\|_{L^2} \leq C \|(1-\psi_{-1})u\|_{H^1},
\eeq
and exactly the same argument used to prove \eqref{eq:BHM2} shows that 
\beq\label{eq:BHM2a}
\|(1-\newchi_{-1})u\|_{H^1}\leq Ck \|u\|_{\mc{H}}+C\|\rhs\|_{\mc{H}}. 
\eeq
Therefore, combining \eqref{eq:ER1}--\eqref{eq:BHM2a}, and using the bound \eqref{e:truncToScaled} with $s=0$, we obtain that
\beq\label{eq:combine2}
\|(1-\newchi_0)u\|_{H^2}\leq C\big(k^2\|u\|_{\mc{H}}+\|\rhs\|_{\mc{H}}\big)\leq Ck^2\big(\resolvent (k)+k^{-2}\big)\|\rhs\|_{\mc{H}}.
\eeq
The combination of \eqref{eq:comb1} and \eqref{eq:combine2} proves the bound \eqref{e:truncToScaled} for $s=1$;
the bound \eqref{e:truncToScaled} for $0<s<1$ then follows by interpolation.
\end{proof}

\subsection{The PML operator}\label{sec:33}

In addition to the Fredholm property for $P_\theta$, we need Fredholm properties for the corresponding PML operator. Let $\Omega_{\tr}\Subset \mathbb{R}^d$ have Lipschitz boundary and $B(0,R_1)\subset \Omega_{\tr}$. We study the PML operator $P^{D}_\theta -\lambda^2$ on $\Omega_{\tr}$. That is, we define
\begin{equation}
\label{e:defPMLa}
\begin{gathered}
\mc{H}(\Omega_{\tr}):=\mc{H}_{R_0}\oplus L^2(\Omega_{\tr} \setminus B(0,R_0)),\\
\mc{D}(\Omega_{\tr}):=\Big\{u\in \mc{H}(\Omega_{\tr}) \,:\, 
{\text{for all}\,\chi\in C_c^\infty(B(0,R_1)),\,\chi\equiv 1\text{ on }B(0,R_0),}\\
\hspace{3cm} 
\chi u\in \mc{D},\, { (1-\chi)u \in H_0^1(\Omega_{\tr}),\, -\Delta_\theta ((1-\chi)u)\in L^2(\Omega_{\tr})}\Big\},\\
P^{D}_\theta u:= P(\chi u)+ (-\Delta_\theta)((1-\chi )u).
\end{gathered}
\end{equation}
We then consider $P^{D}_\theta:\mc{H}(\Omega_{\tr})\to \mc{H}(\Omega_{\tr})$ with domain $\mc{D}(\Omega_{\tr})$ 
and norm
\beq\label{eq:normDtr}
\|u\|_{\mc{D}(\Omega_\tr)}^2=\|u\|_{\mc{H}(\Omega_\tr)}^2+\|P_\theta^D u\|_{\mc{H}(\Omega_\tr)}^2,\qquad u\in \mc{D}(\Omega_\tr).
\eeq

\begin{prop}
\label{p:fredPMLa}
Let $P^{D}_\theta$, $\mc{H}(\Omega_{\tr})$, and $\mc{D}(\Omega_{\tr})$ be as in~\eqref{e:defPMLa}. Then, $P^{D}_\theta-\lambda^2:\mc{D}(\Omega_\tr)\to \mc{H}(\Omega_{\tr})$ is Fredholm with index zero. 
\end{prop}
Proposition~\ref{p:fredPMLa} is proved in Appendix~\ref{a:scale}; see Proposition~\ref{p:fredPML}.

\section{Elliptic estimates}\label{s:elliptic}
In this section, we prove the necessary bounds on the solutions to $(P_\theta-k^2)u=f$ and $(P^D_\theta -k^2)v=f$ for $k\in \mathbb{R}$, $k\gg 1$. The Carleman estimates in \S\ref{s:carleman} describe how both $u$ and $v$ propagate in the scaling region. The bound in \S\ref{s:boundaryEst} (obtained essentially by integration by parts) describes the behaviour of $v$ in a neighbourhood of $\Gamma_\tr$.

It is convenient to use the semiclassical rescaling $\hbar=k^{-1}$ 
\footnote{The semiclassical parameter is often denoted by $h$, but we use $\hbar$ to avoid a notational clash with the meshwidth of the FEM appearing in \S\ref{s:numerical}.}
and write these equations as 
$$
(\hbar^2P_\theta-1)u=\hbar^2\rhs,\qquad (\hbar^2P_\theta^D-1)v=\hbar^2\rhs,
$$
and we do so throughout the rest of the paper.
We use the semiclassically-scaled Sobolev norms for $\ell\in \mathbb{N}$ defined by
$$
\|u\|_{H_{\hbar}^\ell}^2:=\sum_{|\alpha|\leq \ell}\|(\hbar D)^\alpha u\|_{L^2}^2,
$$
{where $D:= -i \partial$.}
Then, for $\ell\in \mathbb{N}$, $H_{\hbar}^{-\ell}=(H_{\hbar}^\ell)^*$ and the norms for $s\in \mathbb{R}$ are defined by interpolation. 
With $\langle\cdot\rangle:= (1+ |\cdot|^2)^{1/2}$,
these norms satisfy
$$
\|u\|_{H_{\hbar}^s}\sim \|\langle \hbar D\rangle^s u\|_{L^2}.
$$
\subsection{Carleman estimates}
\label{s:carleman}

We start by proving an exponential estimate for solutions to 
$$
(-\hbar^2\Delta_\theta -1)u=\rhs,
$$
for $u$ supported in $r>R_1$. Our estimates are proved using Carleman estimates with weight $\psi=\psi(r)$. 
To this end, for $\psi \in C^\infty([0,\infty))$, we define
\begin{equation}
\label{e:pPsi}
P_\psi:=e^{\psi/\hbar}(-\hbar^2\Delta_\theta-1)e^{-\psi/\hbar},
\end{equation}
with {semiclassical principal} symbol 
$$
p_\psi(r,\omega,\xi_r,\xi_\omega):=\Big(\frac{\xi_r+i\psi'}{1+if_\theta'(r)}\Big)^2+\frac{|\xi_\omega|^2_{S^{d-1}}}{(r+if_\theta(r))^2}-1.
$$

\begin{lem}
\label{l:carleman1}
Let $\e>0$ and $\Phi_\theta$ be as in~\eqref{e:weightDer}. Then there is $c_{\e, f}>0$ such that for $r>R_1+\e$ and $\e\leq \theta\leq\pi/2-\e$, $\Phi_\theta(r)>c_{\e,f}$. Moreover, given $0\leq a<1$, there is $c>0$ such that for all $\e\leq \theta\leq \pi/2-\e$ and $r\geq R_1 + \e$ such that 
\beq\label{eq:carleman1}
|\psi'(r)|< a\Phi_\theta(r),
\eeq
$P_\psi$ is uniformly semiclassically elliptic in $r\geq R_1+\e$; i.e.,
$$
|p_\psi|\geq c\langle \xi\rangle^2,\qquad r\geq R_1+\e.
$$
\end{lem}
\begin{proof}
In the following arguments, $c_{f,\e}, C_{f,\e}>0$ are constants depending on $f$ and $\e$ whose values may change from line to line. Throughout the proof, $r>R_1+\e$ and $\theta\in [\e,\pi/2-\e]$.

The solutions, $s_{\pm}$ to 
$$
\widetilde{p}(s):=\Big(\frac{s}{1+if_\theta'(r)}\Big)^2+\frac{|\xi_\omega|^2}{(r+if_\theta(r))^2}-1=0
$$
are given by
$$
s_{\pm}= \pm (1+if_\theta'(r))\sqrt{1-\frac{|\xi_\omega|^2}{(r+if_\theta(r))^2}}.
$$
The definition of $\Phi_\theta(r)$ \eqref{e:weightDer} then implies that
\beq\label{eq:Phispm}
\Phi_\theta(r)= \inf_{|\xi_\omega|\geq 0} \min \big\{
|\Im s_+|, |\Im s_-|
\big\}
=\inf_{|\xi_\omega|\geq 0 
|\Im s_+|,}
\eeq
since $s_+ =- s_-$.

By considering the real and imaginary parts of $\widetilde{p}(s)$, we find that 
$$
|\widetilde{p}(s)|\geq 
c_{f,\e}\big(|\Re s|^2+|\xi_\omega|^2{/r^2}+1\big),\qquad \Im s=0,
$$
where we have used  the particular form of $f_\theta(r)$, i.e., $f_\theta(r) = f(r)\tan \theta$ and the fact that $\theta\in [\e,\pi/2-\e]$ to get uniformity in $\theta$.
Therefore, since there exists $c_{f,\e}>0$ such that,
$$
|\partial_s\widetilde{p}|\leq c_{f,\e}|s|, 
$$
there is $c_{f,\e}>0$ such that
\begin{equation}
\label{e:PhiIsntSmall}
|\widetilde{p}(s)|\geq c_{f,\e}\big(|\Re s|^2+|\xi_\omega|^2{/r^2}+1\big),\qquad |\Im s|\leq c_{f,\e},
\end{equation}
and thus $|\Im s_{\pm}|>c_{f,\e}$. Therefore, by \eqref{eq:Phispm}, 
$\Phi_\theta(r)>c_{f,\e}$.
Hence, if $|\Im s|<a\Phi_\theta(r)$, then
$$
\min_{\pm} |s -s_{\pm}|>c_{f,\e}(1-a).
$$

In particular, since 
$$
|\partial_s\widetilde{p}(s_{\pm})|=\Big|\frac{2s_{\pm}}{(1+if_\theta'(r))^2}\Big|\geq c_{f,\e},
$$
and 
$$
|\partial_s^2\widetilde{p}(s)|=\frac{1}{|1+if_\theta'|^2}\leq C_{f,\e},
$$
there is $c_{a,f,\e}>0$ such that 
$$
|\widetilde{p}(s)|\geq c_{a,f,\e}
$$
if $|\Im s|<a\Phi_\theta(r)$.
Finally, by considering the real and imaginary parts of $\widetilde{p}(s)$ again (this time with $s$ not necessarily real),
we see that there exists $C_{f,\e}>0$ such that, for $|\Re s|^2+|\xi_\omega|^2/r^2\geq C_{f,\e}$, 
$$
|\widetilde{p}(s)|\geq C_{f,\e}^{-1}\big(|\Re s|^2+|\xi_\omega|^2/r^2+1\big).
$$
Together, we have shown  that for $|\Im s|<a\Phi_\theta(r)$, 
$$
|\widetilde{p}(s)|\geq c_{a,f,\e}\big(|\Re s|^2+|\xi_\omega|^2/r^2+1\big)
$$
and the claim follows.
\end{proof}

In the rest of the paper we use the notation that $(a,b)_r := B(0,b)\setminus B(0,a)$.
\begin{lem}
\label{l:carleman}
Let $\e>0$, $\eta>0$. Then there are $C>0$, $\hbar_0>0$, and $0<\widetilde\eta<\e/6$ such that for all $\e\leq \theta \leq \pi/2-\e$, $\delta>\e$, $u\in L^2$, $0<\hbar<\hbar_0$,
\begin{multline}
\label{e:rightFromLeft}
\|u\|_{H_\hbar^2(R_1+\delta-2\widetilde\eta,R_1+\delta-\widetilde\eta)_r}\leq C\|(-\hbar^2\Delta_\theta-1)u\|_{L^2(R_1,R_1+\delta)_r}\\+ C\exp\Big(-\frac{(1-\eta)}{\hbar}\int_{R_1}^{R_1+\delta}\Phi_\theta(s)\,ds\Big)\hbar\|u\|_{H_{\hbar}^1(R_1, R_1+\widetilde\eta)_r }
+C\hbar\|u\|_{H_{\hbar}^1(R_1+\delta-\widetilde{\eta}, R_1+\delta)_r },
\end{multline}
and
\begin{multline}
\label{e:rightFromLeftNoRightBoundary}
\|u\|_{H_\hbar^2(R_1+\delta-2\widetilde\eta,\infty)_r}\leq C\|(-\hbar^2\Delta_\theta-1)u\|_{L^2(R_1,\infty)_r}\\
+C\exp\Big(-\frac{(1-\eta)}{\hbar}\int_{R_1}^{R_1+\delta}\Phi_\theta(s)\,ds\Big)
\hbar\|u\|_{H_{\hbar}^1(R_1, R_1+\widetilde\eta)_r } .
\end{multline}
\end{lem}
\begin{proof}
Let $P_\psi$ be as in~\eqref{e:pPsi}. To prove the lemma, we construct a $\psi$ {satisfying \eqref{eq:carleman1} with $a=1-\widetilde{\eta}$ for some, not yet specified, $\widetilde{\eta}$.}
Let $\psi_0\in C_c^\infty((2\widetilde{\eta},\delta-2\widetilde{\eta});[0,1])$ with $\psi_0\equiv 1$ on $(3\widetilde{\eta},\delta-3\widetilde{\eta})$. Then, let $0\leq \widetilde{\Phi}_{\theta}(r)\in C^\infty$ with 
$(1-{2}\widetilde{\eta})\Phi_\theta(r)\leq\widetilde{\Phi}_{\theta}(r)\leq  {(1-\widetilde{\eta})}\Phi_\theta(r)$ on $[R_1+\widetilde{\eta},\infty)$ and  $\supp \widetilde\Phi_\theta \subset (R_1+\widetilde{\eta}/2,\infty)$. Then define
\beq\label{eq:psi}
\psi(t)=-\int_{t}^\infty\widetilde{\Phi}_{\theta}(s)\psi_0(s-R_1)\,ds,
\eeq
and choose  $0<\widetilde{\eta}<\e/6$ small enough such that 
\beq\label{eq:carlemanproof1}
-\int_{-\infty}^\infty \widetilde{\Phi}_{\theta}(s){\psi_0}(s-R_1)\, d s\leq -(1-\eta)\int_{R_1}^{R_1+\delta}\Phi_\theta(s)\,ds;
\eeq
note that this choice can be made uniformly in $\delta>\e$.
By \eqref{eq:psi} and the support properties of $\psi_0$,
$$
|\psi'(t)|=\widetilde{\Phi}_{\theta}(t)|{\psi_0}(t-R_1)|\leq (1-\widetilde{\eta})\Phi_\theta(t),
$$
so that, by Lemma~\ref{l:carleman1}, $|p_\psi|\geq c\langle \xi\rangle^2$, {for all $t$.}
In addition
\beq\label{eq:carlemanproof2}
\psi(t)= -\int_{-\infty}^\infty\widetilde{\Phi}_{\theta}(s){\psi_0}(s-R_1)\,ds,\,\, t-R_1\leq 2\widetilde{\eta},\quad \tand\quad\psi(t)=0,\,\, t-R_1\geq \delta-2\widetilde{\eta},
\eeq
{see Figure \ref{fig:cutoff1},}
and
$$
|\partial^\alpha \psi(t)|\leq C_{\alpha\widetilde{\eta}\e}\quad{\tfa t}.
$$

To prove~\eqref{e:rightFromLeft}, let $\chi_1,\chi_2 \in C_c^\infty(R,R+\delta)$ with $\chi_1\equiv 1$ in a neighborhood of $[R_1+\widetilde{\eta},R_1+\delta-\widetilde{\eta}]$, $\chi_2\equiv 1$ on $\supp \chi_1$.
Let $\chi_-\in C_c^\infty(R_1,R_1+\widetilde{\eta})$, $\chi_ +\in C_c^\infty(R_1+\delta-\widetilde{\eta},R+\delta)$ with $\chi_-+\chi_+\equiv 1$ on $\supp (\chi_2-\chi_1)$; {see Figure \ref{fig:cutoff1}.}

\begin{figure}
\begin{center}
    \includegraphics[scale=0.9]{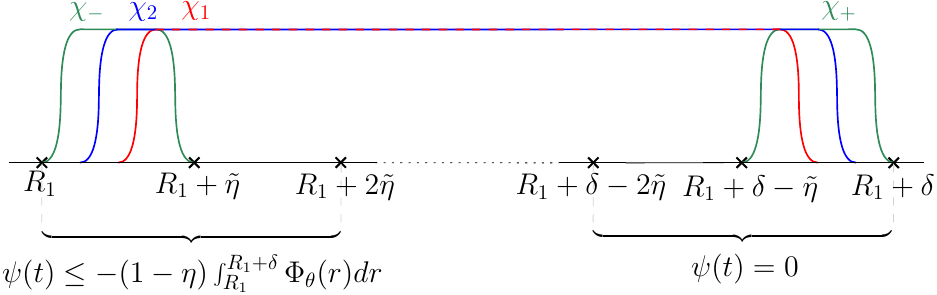}
  \end{center} 
      \caption{The cut-off functions and behaviour of the function $\psi(t)$ in the proof of Lemma \ref{l:carleman}.
      Although $\widetilde{\eta}$ appears large here (for readability), we emphasise that $\widetilde{\eta}\ll \delta$.
      }
\label{fig:cutoff1}
\end{figure}

Now, 
$$
P_\psi = \Op_{\hbar}(p_{0,\psi})+\hbar \Op_{\hbar}(p_{1,\psi}).
$$
with $p_{0,\psi}\in C^{1,\alpha}S^2$, $p_{1,\psi}\in C^{0,\alpha}S^1$, and 
$$
|p_{0,\psi}|\geq c_{\e \widetilde{\eta}}\langle \xi\rangle^2.
$$
Let $\rho =\frac{1}{1+\alpha}$. Then, by Lemmas~\ref{l:roughNorm} and~\ref{l:roughApprox}, there is $p_{\hbar,\psi}$ satisfying
$$|\partial_x^\gamma \partial_\xi^\beta p_{\hbar,\psi}(x,\xi)|\leq C_{\e\widetilde{\eta}\gamma \beta} \hbar^{-\rho\gamma }\langle \xi\rangle ^{m-|\beta|+\rho \gamma},$$
and 
$$
\|\Op_{\hbar}(p_{\hbar,\psi})-P_\psi\|_{H_{\hbar}^1\to L^2}\leq C_{\e \widetilde{\eta}}\hbar.
$$

By a standard elliptic-parametrix construction for $p_{\hbar}$ in an exotic symbol class (see Theorem \ref{t:ellip} for the standard elliptic-parametrix construction and~\cite[\S7.3-7.4]{Ta:96} for the construction in exotic calculi),  there is $E:L^2\to H_{\hbar}^2$, such that 
$$
\chi_1=E\Op_{\hbar}(p_{\hbar,\psi})+O(\hbar^\infty)_{\Psi^{-\infty}}.
$$
Moreover, both $E$ and the error are uniform over $\delta\geq \e$, $\e\leq \theta \leq \pi/2-\e$.
Therefore, 
\begin{equation}
\label{e:parametrix}
\begin{aligned}
\chi_1 e^{\psi/\hbar}u&=\chi_1 \chi_2e^{\psi/\hbar} u=E\Op_{\hbar}(p_{\hbar,\psi}) \chi_2e^{\psi/\hbar} u+ O_{\widetilde{\eta}\e}(\hbar^\infty)_{\Psi^{-\infty}}\chi_2 e^{\psi/\hbar} u\\
&=E(P_\psi +(\Op_{\hbar}(p_{\hbar,\psi})-P_\psi))\chi_2 e^{\psi/\hbar} u+ O_{\widetilde{\eta}\e}(\hbar^\infty)_{\Psi^{-\infty}}\chi_2 e^{\psi/\hbar} u\\
&=E \chi_2e^{\psi/\hbar}\rhs-Ee^{\psi/\hbar}[\hbar^2\Delta_\theta,\chi_2]u+ O_{\widetilde{\eta}\e}(\hbar)_{H_{\hbar}^1\to H_{\hbar}^2}\chi_2  e^{\psi/\hbar}u,
\end{aligned}
\end{equation}
where $f:=(-\hbar^2 \Delta_\theta-1) u.$
Therefore, since $\partial \chi_2$ is supported where $\chi_+ +\chi_-=1$,
\begin{equation}\label{eq:carlemanproof3}
\|\chi_1e^{\psi/\hbar}u\|_{H_{\hbar}^2}
\leq C\|\chi_2e^{\psi/\hbar}\rhs\|_{L^2}+C\hbar\|\chi_-e^{\psi/\hbar}u\|_{H_{\hbar}^1} +C\hbar\|\chi_+e^{\psi/\hbar}u\|_{H_{\hbar}^1} +C\hbar\|\chi_2 e^{\psi/\hbar}u\|_{H_{\hbar}^1}.
\eeq
Since $\chi_2 = (\chi_2-\chi_1)+ \chi_1$ and $\chi_+ + \chi_-=1$ on $\supp(\chi_2-\chi_1)$, 
\beq\label{eq:carlemanproof4}
\|\chi_2 e^{\psi/\hbar} u \|_{H^1_\hbar} \leq 
C\|\chi_- e^{\psi/\hbar} u \|_{H^1_\hbar} 
+ C\|\chi_+ e^{\psi/\hbar} u \|_{H^1_\hbar} 
+ \|\chi_1 e^{\psi/\hbar} u \|_{H^1_\hbar}.
\eeq
Combining \eqref{eq:carlemanproof3} and \eqref{eq:carlemanproof4} and taking $\hbar$ sufficiently small  (depending only on $\eta$ and $\e$), we have
\beq\label{eq:carlemanproof5}
\|\chi_1 e^{\psi/\hbar} u \|_{H^2_\hbar} \leq 
C\|\chi_2e^{\psi/\hbar}\rhs\|_{L^2}+C\hbar\|\chi_-e^{\psi/\hbar}u\|_{H_{\hbar}^1} +C\hbar\|\chi_+e^{\psi/\hbar}u\|_{H_{\hbar}^1}.
\eeq
Then,  since $\psi\leq -(1-\eta)\int_{R_1}^{R_1+\delta}\Phi_\theta(s)ds$ on $\supp  \chi_-$ {(by \eqref{eq:carlemanproof1} and \eqref{eq:carlemanproof2}; see Figure \ref{fig:cutoff1})}, and $\psi\leq 0$ everywhere (and thus, in particular, on $\supp \chi_+$),
$$
\|\chi_1e^{\psi/\hbar}u\|_{H_{\hbar}^2}\leq C\|\chi_2e^{\psi/\hbar}\rhs\|_{L^2}+C\exp\Big(-\frac{(1-\eta)}{\hbar}\int_{R_1}^{R_1+\delta}\Phi_\theta(s)ds\Big)
\hbar\|\chi_-u\|_{H_{\hbar}^1}+C\hbar\|\chi_+u\|_{H_{\hbar}^1}.
$$ 
The lemma now follows since $\chi_1\equiv 1$ and $\psi\equiv 0$ on $(R_1+\delta-2\widetilde{\eta}, R_1+\delta-\widetilde{\eta})$,  $\supp \chi_-\subset (R_1,R_1+\widetilde{\eta})$, and $\supp \chi_+\subset (R_1+\delta-\widetilde{\eta},R_1+\delta)$,  and $\psi\leq 0$ everywhere.

To prove~\eqref{e:rightFromLeftNoRightBoundary}, we make the same argument as above except that $\chi_1,\chi_2\in C^\infty(R_1,\infty)$ with $\chi_1\equiv 1$ in a neighborhood of $[R_1+\widetilde{\eta},\infty)$, $\chi_2\equiv 1$ on $\supp \chi_1$, and $\chi_+=0$.
\end{proof}

Next, we need an elliptic estimate away from the support of the right hand side.
\begin{lem}
\label{l:carleman2}
Let $\e,\eta>0$. Then there are $C>0$, $\hbar_0>0$, and $0<\widetilde\eta<\e/6$ such that for all $\e\leq \theta \leq \pi/2-\e$, $\e<s<\delta-\e$, $\delta>2\e$ and all $u\in L^2$ satisfying 
$$
(-\hbar^2\Delta_\theta-1)u=\rhs
$$
with $\supp \rhs \cap (R_1,R_1+\delta)_r\subset (R_1+\delta-\widetilde{\eta},R_1+\delta)_r,$ and all $0<\hbar<\hbar_0$,
\begin{multline}
\label{e:leftFromRight}
\|u\|_{H_\hbar^2(R_1+s-2\widetilde{\eta},R_1+ s+2\widetilde{\eta})_r}
\leq C 
\exp\Big(-\frac{(1-\eta)}{\hbar}\int_{R_1}^{{R_1+}s} \Phi_\theta(r)\, dr\Big)
\hbar \|u\|_{H_{\hbar}^1(R_1, R_1+\widetilde{\eta})_r }\\
+C \exp\Big(- \frac{(1-\eta)}{\hbar}\int_{{R_1+}s}^{{R_1}+\delta} \Phi_\theta(r)\, dr\Big)
\Big(\|\rhs\|_{L^2}+ \hbar\|u\|_{H_{\hbar}^1(R_1+\delta-\widetilde{\eta}, R_1+\delta)_r}\Big).
\end{multline}
\end{lem}

{
\noindent(Note that, since $\widetilde{\eta}<\epsilon/6$ and $s<\delta-\e$, $R_1 + s+ 2\widetilde{\eta}< R_1 + \delta -\widetilde{\eta}$, the norm on the left-hand side of \eqref{e:leftFromRight} is indeed away from $\supp \rhs$.)
}

\begin{proof}
As in the proof of Lemma \ref{l:carleman}, we use a Carleman estimate with $P_\psi$ as in~\eqref{e:pPsi}. Let $\psi_-\in C_c^\infty( (2\widetilde{\eta},s-2\widetilde{\eta});[0,1])$ with $\psi_-\equiv  1$ on $(3\widetilde{\eta},s-3\widetilde{\eta})$, and $\psi_+\in C_c^\infty( (s+2\widetilde{\eta},\delta-2\widetilde{\eta};[0,1]))$ with $\psi_+ \equiv  1$ on $(s+3\widetilde{\eta},\delta-3\widetilde{\eta})$. Then, 
{exactly as in the proof of Lemma \ref{l:carleman}, }
let $0\leq \widetilde{\Phi}_{\theta}(r)\in C^\infty$ with 
$(1-{2}\widetilde{\eta})\Phi_\theta(r)\leq \widetilde{\Phi}_{\theta}(r)\leq {(1-\widetilde{\eta})}\Phi_\theta(r)$ on $[R_1+\widetilde{\eta},\infty)$ and  $\supp \widetilde\Phi_\theta \subset (R_1+\frac{\widetilde{\eta}}{2},\infty)$, {for some, not yet specified, $\widetilde{\eta}$.}
Let
\beq\label{eq:psi2}
\psi(t)=\int_{{R_1+}s}^t( \psi_-(r-R_1)-\psi_+(r-R_1))\widetilde{\Phi}_{\theta}(r)\,dr,
\eeq
and choose $0<\widetilde{\eta}<\e/6$ such that 
\beq\label{eq:carlemanproof1a}
-\int_{{R_1+}s}^\infty \psi_+(r-R_1)\widetilde{\Phi}_{\theta}(r)\,dr\leq- (1-\eta)\int_{{R_1+}s}^{{R_1}+\delta} \Phi_\theta(r)\,dr,
\eeq
and
\beq\label{eq:carlemanproof1b}
-\int_{-\infty}^{{R_1+}s} \psi_-(r-R_1)\widetilde{\Phi}_{\theta}(r)\,dr\leq -(1-\eta)\int_{R_1}^{{R_1+}s} \Phi_\theta(r)\,dr;
\eeq
note that this choice can be made uniformly in $\delta>2\e$ and $\e<s<\delta-\e$. 
By \eqref{eq:psi2} and the support properties of $\psi_-$ and $\psi_+$,
$$
|\psi'(t)|\leq \big(|\psi_-(t-R_1)|+|\psi_+(t-R_1)|\big)\widetilde{\Phi}_{\theta}(t)\leq (1-\widetilde{\eta})\Phi_\theta(t),
$$
and 
\begin{align}\label{eq:carlemanproof2a}
\psi(t)\equiv -\int_{{R_1+}s}^\infty \psi_+(r-R_1)\widetilde{\Phi}_{\theta}(r)\,dr,\quad t-R_1\geq \delta-2\widetilde{\eta},
\\ \psi(t)\equiv 0,\qquad s-2\widetilde{\eta}\leq t-R_1\leq s+2\widetilde{\eta},
\nonumber
\\
\psi(t)\equiv -\int_{-\infty}^{{R_1+}s}\psi_-(r-R_1)\widetilde{\Phi}_{\theta}(r)\,dr,\quad t-R_1\leq 2\widetilde{\eta};
\label{eq:carlemanproof2b}
\end{align}
see Figure \ref{fig:cutoff2}.

\begin{figure}
\begin{center}
    \includegraphics[scale=0.9]{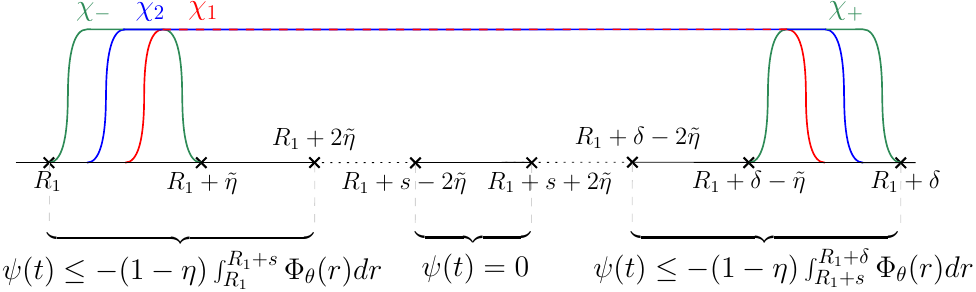}
  \end{center} 
      \caption{The cut-off functions and behaviour of the function $\psi(t)$ in the proof of Lemma \ref{l:carleman2}.
            Although $\widetilde{\eta}$ appears large here (for readability), we emphasise that $\widetilde{\eta}\ll \delta$.
      }
\label{fig:cutoff2}
\end{figure}

To prove the lemma, let $\chi_1,\chi_2, \chi_-, \chi_+$ be as in the proof of Lemma \ref{l:carleman}, i.e., $\chi_1,\chi_2 \in C_c^\infty(R_1,R_1+\delta)$ with $\chi_1\equiv 1$ in a neighborhood of $[R_1+\widetilde{\eta},R_1+\delta-\widetilde{\eta}]$, $\chi_2\equiv 1$ on $\supp \chi_1$, and $\chi_-\in C_c^\infty(R_1,R_1+\widetilde{\eta})$, $\chi_ +\in C_c^\infty(R_1+\delta-\widetilde{\eta},R_1+\delta))$ with $\chi_-+\chi_+\equiv 1$ on $\supp (\chi_2-\chi_1)$. Applying the same argument as in the proof of Lemma~\ref{l:carleman}, we obtain
\begin{align*}
\chi_1 e^{\psi/\hbar}u &=E \chi_2e^{\psi/\hbar}f -Ee^{\psi/\hbar}[\hbar^2\gamma_\theta^2\Delta,\chi_2]u+ O_{\widetilde{\eta}\e}(\hbar)_{H_{\hbar}^1\to H_{\hbar}^2}e^{\psi/\hbar}\chi_2 u
\end{align*}
(see \eqref{e:parametrix}). Arguing exactly as before, we obtain \eqref{eq:carlemanproof5}.
Therefore,  since 
$\psi\leq - (1-\eta)\int_{{R_1}}^{{R_1+}s} \Phi_\theta(r)\,dr$ on $\supp \chi_-$ 
{(by \eqref{eq:carlemanproof1b} and \eqref{eq:carlemanproof2b})} 
and
$\psi\leq -(1-\eta)\int_{{R_1+s}}^{{R_1+\delta}} \Phi_\theta(r)\,dr$ on $\supp  \chi_+\cup \supp f$
{(by \eqref{eq:carlemanproof1a} and \eqref{eq:carlemanproof2a})},
\begin{align*}
&\|\chi_1e^{\psi/\hbar}u\|_{H_{\hbar}^2}
\leq C\exp\Big(- \frac{(1-\eta)}{\hbar}\int_{{R_1}}^{{R_1+}s} \Phi_\theta(r)\, dr\Big)
\hbar\|\chi_-u\|_{H_{\hbar}^1}\\
&\hspace{3cm}
+C
\exp\Big(-\frac{(1-\eta)}{\hbar}\int_{{R_1+}s}^{{R_1+\delta}} \Phi_\theta(r)\, dr\Big)
\big(\|\rhs\|_{L^2}+\hbar\|\chi_+u\|_{H_{\hbar}^1}\big).
\end{align*}
The bound \eqref{e:leftFromRight} now follows using  the support properties of $\chi_{\pm}$ and the facts that $\chi_1\equiv 1$ and $\psi\equiv 0$ on $(R_1+s-2\widetilde{\eta},R+ s+2 \widetilde{\eta})$.
\end{proof}

\subsection{Estimate on the PML solution near the boundary}
\label{s:boundaryEst}

\begin{lem}
\label{l:nearBoundary}
For any $\epsilon>0$, there exists $\hbar_0>0$ and $C(\epsilon)>0$ so that for any $\e<\theta<\pi/2-\e$, $R_{\tr}>R_1+\e$, $B(0,R_1)\Subset\Omega_{\tr}\subset \mathbb{R}^d$ with Lipschitz boundary, if $v\in L^2$ is supported in 
$\Omega_{\tr}\setminus B(0,R_1+\e)$ and $v=0$ on $\partial \Omega_{\tr}$, then, for all $0<\hbar \leq \hbar_0$,
\begin{align}\label{eq:boundary1}
 &\Vert v \Vert_{H^1_\hbar(\Omega_{\tr})} \leq C( \epsilon) \Vert  (\hbar^2P^D_\theta -1)v \Vert_{L^2(\Omega_{\tr})}.
\end{align}
\end{lem}

\begin{proof}
We use results from Appendix \ref{a:scale}, and use that, by Lemma~\ref{l:tiger}, $F''_\theta(x) \geq \delta(\e) > 0$ in the sense of quadratic forms
for $x \in \operatorname{supp}v$.
Since $v$ is zero in a neighbourhood of $B(0,R_0)$,
\begin{equation} \label{eq:deep:Pstar}
\langle (\hbar^2P^D_\theta-1)v,  v \rangle_{L^2(\Omega_{\tr})} = \big\langle (-\hbar^2 \Delta_\theta - 1 )v ,v \big\rangle_{L^2(\Omega_{\tr})}.
\end{equation}
However, {by \eqref{eq:A41} and \eqref{e:absVal}, }
\begin{equation} \label{eq:deep:IPP}
\langle (-\hbar^2 \Delta_\theta - 1) v, v\rangle_{L^2} = \Vert w \Vert_{L^2}^2 - \Vert F''_\theta(x) w \Vert_{L^2}^2  - 2i \langle  F''_\theta(x) w, w \rangle+ \hbar \langle  A_\theta (x)\hbar\partial_x v,v\rangle - \Vert v \Vert^2_{L^2},
\end{equation}
where $A_\theta(x)\in C^{0,\alpha}$ and $w := (I+F''_\theta(x)^2)^{-1} \hbar \partial_x v$. Taking the imaginary part of (\ref{eq:deep:IPP}) and using the fact that $F''_\theta(x) \geq \delta(\e) > 0$
for $x \in \operatorname{supp}v$, and then using (\ref{eq:deep:Pstar}), we obtain that
\begin{align}  \label{eq:deep:im}
\Vert \hbar \partial_x v \Vert^2_{L^2} &\leq C \Vert w \Vert^2_{L^2} \leq C \langle  F''_\theta(x) w, w \rangle
 \leq C\Big| \operatorname{Im} \langle (-\hbar^2 \Delta_\theta - 1) v, v\rangle \Big| + C\Big| \operatorname{Im} \hbar \langle  A_\theta(x)\hbar\partial_x v,v\rangle \Big| \nonumber \\
& \leq C \Vert (\hbar^2P^D_\theta-1)v \Vert_{L^2} \Vert v \Vert_{L^2} + C\hbar \Vert \hbar \partial_x v \Vert_{L^2} \Vert v \Vert_{L^2},
\end{align}
where $C$ depends a-priori on $\theta$.
Now taking the real part of (\ref{eq:deep:IPP}), we get
\begin{equation} \label{eq:deep:real}
\Vert v \Vert^2_{L^2} \leq C \Vert \hbar \partial_x v \Vert^2_{L^2} +  C\Vert (\hbar^2P^D_\theta-1)v \Vert_{L^2} \Vert v \Vert_{L^2} + C\hbar \Vert \hbar \partial_x v \Vert_{L^2} \Vert v \Vert_{L^2}.
\end{equation}
Thus, combining (\ref{eq:deep:im}) and (\ref{eq:deep:real}), we have
\begin{align*}
\Vert v \Vert^2_{H^1_\hbar} \leq C  \Vert (\hbar^2P^D_\theta-1)v \Vert_{L^2} \Vert v \Vert_{L^2} + C \hbar \Vert \hbar \partial_x v \Vert_{L^2} \Vert v \Vert_{L^2}.
\end{align*}
{With $F_\theta$ and $f_\theta$ related by \eqref{e:basicF},} and $f_\theta(r) = f(r) \tan \theta$ satisfying \eqref{e:fProp2},  all the implicit constants appearing above depend continuously on $\tan \theta$. Hence, for $\epsilon<\theta<\pi/2 - \epsilon$, there is $C(\epsilon)>0$, depending only on $\epsilon$, such that
\begin{align*}
\Vert v \Vert^2_{H^1_\hbar} \leq  C(\epsilon) \Big[ \Vert (\hbar^2P^D_\theta-1)v \Vert_{L^2} \Vert v \Vert_{L^2} + \hbar \Vert \hbar \partial_x v \Vert_{L^2} \Vert v \Vert_{L^2}\Big];
\end{align*}
the  bound \eqref{eq:boundary1} then follows by taking $\hbar>0$ small enough depending only on $\epsilon$.
\end{proof}

\section{Proof of Theorems  \ref{t:blackBoxErr} and \ref{t:blackBoxResolve} (the main results in the black-box setting)}\label{s:proofs1}

\begin{proof}[Proof of Theorem~\ref{t:blackBoxResolve}]
By Proposition \ref{p:fredPMLa}, $P_\theta^D-k^2$ is Fredholm of index zero; thus existence and uniqueness of $v$ follow once we establish the a priori bound \eqref{e:blackBoxResolve}.

The overall idea to prove \eqref{e:blackBoxResolve} is to use the elliptic estimates in \S\ref{s:elliptic} to bound $v$ near $\Gamma_\tr$ in terms of $v$ away from $\Gamma_\tr$ and the data $\rhs$, and then use 
Lemma~\ref{l:truncToScaled} to bound $v$ away from $\Gamma_\tr$.
Given $\e>0$, let $\eta>0$ (to be fixed in terms of $\e$ later). Then, by \eqref{e:rightFromLeft} 
with $\delta =R_{\tr}-R_1$, there is $0<\widetilde{\eta}<\e/6$ and $C>0$ such that 
\begin{equation}
\label{e:innerScaledAnnulus}
\begin{aligned}
\|v\|_{H_{\hbar}^2(R_{\tr}-2\widetilde{\eta},R_{\tr}-\widetilde{\eta})_r}
&\leq C\hbar^2\|\rhs\|_{L^2(R_1,R_{\tr})_r}
+C \exp\Big(-\frac{(1-\eta)}{\hbar}\int_{R_1}^{R_{\tr}}\Phi_\theta(r)\, dr\Big)\hbar\|v\|_{H_{\hbar}^1(R_1,R_1+\widetilde{\eta})_r}\\
&\qquad+C\hbar\|v\|_{L^2(R_{\tr}-\widetilde{\eta},R_{\tr})_r}.
\end{aligned}
\end{equation}
Let $\chi \in C_c^\infty(\mathbb{R}^d\setminus \overline{B(0,R_{\tr}-2\widetilde{\eta})})$ with $\chi\equiv 1$ on $\Omega_{\tr}\setminus B(0,R_{\tr}-\widetilde{\eta})$. Then, by Lemma~\ref{l:nearBoundary}
\begin{equation}
\label{e:outerScaledAnnulus}
\begin{aligned}
\|v\|_{H_{\hbar}^1(\Omega_{\tr}\setminus B(0,R_{\tr}-\widetilde{\eta}))}\leq \|\chi v\|_{H_{\hbar}^1(\Omega_{\tr})}&\leq C\hbar^2\|(P_\theta -k^2)\chi v\|_{L^2(\Omega_\tr)}\\
&\leq C(\hbar^2\|\chi \rhs\|_{L^2(\Omega_{\tr})}+\|[-\hbar^2\Delta_\theta,\chi]v\|_{L^2(\Omega_{\tr})}).
\end{aligned}
\end{equation}

Combining~\eqref{e:innerScaledAnnulus} and~\eqref{e:outerScaledAnnulus}, using that the derivatives of $\chi$ are uniform in $R_{\tr}\geq R_2+\e$, and $\supp \partial\chi \subset B(0,R_{\tr}-\widetilde{\eta})\setminus B(0,R_{\tr}-2\widetilde{\eta})$, and shrinking $\hbar_0$ if necessary, we have
\begin{equation}
\label{e:scalingRegion}
\|v\|_{H_{\hbar}^1(\Omega_{\tr}\setminus B(0,R_{\tr}-2\widetilde{\eta}))}\leq C\hbar^2\|\rhs\|_{L^2(\Omega_{\tr}\setminus B(0,R_1))} +C\exp\Big(-\frac{(1-\eta)}{\hbar}\int_{R_1}^{R_{\tr}}\Phi_\theta(r)\, dr\Big)\hbar
\|v\|_{H_{\hbar}^1(R_1,R_1+\widetilde{\eta})_r}.
\end{equation}
Next, let $\chi_1\in C_c^\infty(B(0,R_{\tr}))$ with $\chi_1\equiv 1$ on $B(0,R_{\tr}-2\widetilde{\eta})$. Then,
$$
(\hbar^2P_\theta-1) \chi_1v= {\hbar^2}\chi_1 \rhs+[\hbar^2P_\theta,\chi_1 ]v= {\hbar^2}\chi_1\rhs+[-\hbar^2\Delta_\theta,\chi_1]v.
$$
Now, by \eqref{e:scalingRegion},
\begin{align*}
\|[\chi_1,-\hbar^2\Delta_\theta]v\|_{\mc{H}}
&\leq C\hbar \|v\|_{H_{\hbar}^1(R_{\tr}-2\widetilde{\eta},R_{\tr})_r}\\
&\leq C\hbar^3\|\rhs\|_{L^2(\Omega_{\tr}\setminus B(0,R_1))}+ C
\exp\Big(-\frac{(1-\eta)}{\hbar}\int_{R_1}^{R_{\tr}}\Phi_\theta(r)\, dr\Big)
\hbar^2 \|v\|_{H_{\hbar}^1(R_1,R_1+\widetilde{\eta})_r}.
\end{align*}
Therefore, by Lemma~\ref{l:truncToScaled} with $\resolvent (k)=\|\chi R_{P}(k)\chi\|_{\mc{H}\to \mc{H}}$ and $s=1$,
\begin{align*}
\|\chi_1v\|_{\mc{D}}&=\|(P_\theta-k^2)^{-1}(\chi_1 \rhs+[\chi_1,-\Delta_\theta]v)\|_{\mc{D}}\\
&\leq Ck^{2}\resolvent (k)\Big(\|\rhs\|_{\mc{H}}+
\exp\Big(-k(1-\eta)\int_{R_1}^{R_{\tr}}\Phi_\theta(r)\, dr \Big)
\|v\|_{H_{\hbar}^1(R_1,R_1+\widetilde{\eta})_r}\Big).
\end{align*}
Similarly, 
\begin{align*}
\|\chi_1v\|_{\mc{H}}&=\|(P_\theta-k^2)^{-1}(\chi_1 g+[\chi_1,-\Delta_\theta]v)\|_{\mc{H}}\\
&\leq 
C\resolvent (k)\Big(\|\rhs\|_{\mc{H}}+
\exp\Big(-k(1-\eta)\int_{R_1}^{R_{\tr}}\Phi_\theta(r)\, dr \Big)
\|v\|_{H_{\hbar}^1(R_1,R_1+\widetilde{\eta})_r}\Big).
\end{align*}
By the definition of $\theta_0(P,J,R_{\tr})$ \eqref{eq:theta0}, 
\beq\label{eq:ce}
c_\e:= -\Lambda(P,J) +  \inf_{\theta\in [\theta_0+\e, \pi/2-\e]}\int_{R_1}^{R_\tr}\Phi_{\theta}(r)\, dr >0.
\eeq
Now choose $\eta$ (as a function of $\e$) such that
\beq\label{eq:eta_ineq1}
0<\eta< \min \left\{1\,,\, \frac{c_\e}{2\inf_{\theta\in [\theta_0+\e, \pi/2-\e]}\int_{R_1}^{R_\tr}\Phi_{\theta}(r)\, dr} \right\}.
\eeq
This choice implies that, for all $\theta\in [\theta_0+\e,\pi/2-\e]$,
\beq\label{eq:ce2}
-\Lambda(P,J)+(1-\eta) \int_{R_1}^{R_\tr} \Phi_\theta(r)\, dr \geq \frac{c_\e}{2}.
\eeq
Then, using the definition of $\Lambda(P,J)$ \eqref{eq:LambdaPJ}, and
choosing $k$ large enough, depending only on $\e$ and $\eta$ (and hence only on $\e$), we have
\beq\label{eq:WPM1}
\|\chi_1v\|_{\mc{H}}+k^{-2}\|\chi_1v\|_{\mc{D}}\leq C\resolvent (k)\|\rhs\|_{\mc{H}}.
\eeq
The definition of $\chi_1$ and interpolation imply that
\beqs
\|v\|_{H^1_\hbar(R_1,R_1+\widetilde{\eta})_r} \leq C \big( \|\chi_1v\|_{\mc{H}}+k^{-2}\|\chi_1v\|_{\mc{D}}\big),
\eeqs
and thus combining this, \eqref{eq:WPM1}, and \eqref{e:scalingRegion}, we obtain that
\beqs
\|v\|_{\mc{H}(\Omega_\tr)}
\leq C\resolvent (k)\|\rhs\|_{\mc{H}} \quad \tfa k\geq k_0.
\eeqs
Since $\|P_\theta^D v\|_{\mc{H}(\Omega_\tr)} = k^2 \|v\|_{\mc{H}(\Omega_\tr)} + \|\rhs\|_{\mc{H}(\Omega_\tr)}$, the bound \eqref{e:blackBoxResolve} then follows from the definition of $\|v\|_{\mc{D}(\Omega_\tr)}$ \eqref{eq:normDtr}. 
\end{proof}

\begin{proof}[Proof of Theorem~\ref{t:blackBoxErr}]
To avoid writing $\|\chi R_P(k)\chi\|_{\mc{H}\to \mc{H}}$ repeatedly, we let $\resolvent (k):[0,\infty)\to (0,\infty]$ be such that, for all $\chi\in C_c^\infty(B(0,R_1);[0,1])$, 
$$
\|\chi R_P(k)\chi\|_{\mc{H}\to \mc{H}}\leq   \resolvent (k)\qquad \tfa k \geq k_1.
$$ 
Given $\epsilon>0, \eta>0$, let $\widetilde{\eta}$ equal the minimum of the $\widetilde{\eta}$s from Lemmas \ref{l:carleman} and \ref{l:carleman2}.
Observe that the bound \eqref{e:blackBoxErr} gets stronger when $\eta$ decreases. In the course of the proof, therefore, we can reduce $\eta$ (in an $\e$-dependent way) without loss of generality.

By \eqref{e:PMLBox} and \eqref{e:PBox}, $v=(P_\theta^D-k^2)^{-1}\rhs$ and $u= (P-k^2)^{-1}\rhs= R_P(k)\rhs$, where  $\rhs = \chi \widetilde{\rhs}$.
Since $\supp \chi \subset B(0,R_1)$, $\rhs = \indicator_{B(0,R_1)} \rhs$, and thus 
\beqs
\indicator_{B(0,R_1)} u =\indicator_{B(0,R_1)} R_P(k)\indicator_{B(0,R_1)}\rhs = \indicator_{B(0,R_1)}(P_\theta-k^2)^{-1}\indicator_{B(0,R_1)}\rhs
\eeqs
by Proposition~\ref{p:scaledFredholmA}.
Since the bound \eqref{e:blackBoxErr} only involves $\indicator_{B(0,R_1)} u$, without loss of generality we abuse notation slightly and let $u=(P_\theta-k^2)^{-1}g$ for the rest of the proof.

By~\eqref{e:scalingRegion}, together with the fact that $\rhs$ is supported in $B(0,R_1)$,
\begin{equation}
\label{e:PMLAnnulus}
\|v\|_{H_{\hbar}^1(\Omega_{\tr}\setminus B(0,R_{\tr}-2\widetilde{\eta}))}\leq C
\exp\Big(-\frac{(1-\eta)}{\hbar}\int_{R_1}^{{R_{\tr}}}\Phi_\theta(r)\, dr\Big)
\hbar\|v\|_{H_{\hbar}^1(R_1,R_1+\widetilde{\eta})_r}.
\end{equation}
Moreover, using~\eqref{e:rightFromLeftNoRightBoundary},
\begin{equation}
\label{e:scaledAnnulus}
\begin{aligned}
\|u\|_{H_{\hbar}^1(\Omega_{\tr}\setminus B(0,R_{\tr}-2\widetilde{\eta}))}\leq C
\exp\Big(-\frac{(1-\eta)}{\hbar}\int_{R_1}^{{R_{\tr}}}\Phi_\theta(r)\, dr\Big)
\hbar\|u\|_{H_{\hbar}^1(R_1,R_1+\widetilde{\eta})_r}.
\end{aligned}
\end{equation}
Therefore, by Theorem \ref{t:blackBoxResolve} and Lemma~\ref{l:truncToScaled},
\begin{equation}
\label{e:farEstimate}
\|v\|_{H_{\hbar}^1(\Omega_{\tr}\setminus B(0, R_{\tr}-2\widetilde{\eta}))}+\|u\|_{H_{\hbar}^1(\Omega_{\tr}\setminus B(0, R_{\tr}-2\widetilde{\eta})}\leq C
\exp\Big(-\frac{(1-\eta)}{\hbar}\int_{R_1}^{{R_{\tr}}}\Phi_\theta(r)\, dr\Big)
\hbar  \resolvent(\hbar^{-1}) \|\rhs\|_{\mc{H}}.
\end{equation}

Let $\delta=R_{\tr}-R_1$, let $0<\e'<\e$ (to be fixed later),
let $\e'<s<\delta-\e'$  (to be fixed later), and let $\chi_1\in C_c^\infty(B(0,R_{\tr});[0,1])$ with $\chi_1\equiv 1$ on $B(0,R_{\tr}-\widetilde{\eta})$.
Since  ${(-\hbar^2\Delta_\theta-1)}(\chi_1(u-v))=[-\hbar^2\Delta_\theta,\chi_1](u-v)$ and 
$$
\supp [-\hbar^2\Delta_\theta,\chi_1](u-v)\subset B(0,R_{\tr})\setminus B(0,R_{\tr}-\widetilde{\eta}),
$$
we can apply Lemma~\ref{l:carleman2} to $\chi_1(u-v)$, with $\e$ in this lemma 
replaced by $\e'$, and obtain
\begin{multline}\label{eq:WAM1}
\|u-v\|_{H_\hbar^2(R_1+s-2\widetilde{\eta},R_1+ s+2\widetilde{\eta})_r}\leq 
C \exp\Big(-\frac{(1-\eta)}{\hbar}\int_{R_1}^{{R_1+s}} \Phi_\theta(r)\, dr\Big)
\hbar
\|u-v\|_{H_{\hbar}^1(R_1, R_1+\widetilde{\eta})_r }\\
+
C \exp\Big(-\frac{(1-\eta)}{\hbar}\int_{{R_1+s}}^{{R_{\tr}}} \Phi_\theta(r)\, dr\Big)
\hbar\|u-v\|_{H_{\hbar}^1(R_{\tr} - \widetilde{\eta}, R_{\tr})_r},
\end{multline}
where $C$ depends on $\eta$ and $\e'$.
Let $\chi_2\equiv 1$ on $B(0,R_1+s-2\widetilde{\eta})$ with $\supp \chi_2\subset B(0,R_1+s+2\widetilde{\eta})$. Then
\begin{equation}
\label{e:cutCut}
(-\hbar^2\Delta_\theta-1)(\chi_2(u-v))=[-\hbar^2\Delta_\theta ,\chi_2](u-v).
\end{equation}
Hence, by Lemma~\ref{l:truncToScaled}, \eqref{eq:WAM1}, and~\eqref{e:farEstimate}
\begin{align}\nonumber
&\|\chi_2(u-v)\|_{\mc{H}}+\hbar^{2}\|\chi_2(u-v)\|_{\mc{D}}\\ \nonumber
 &\leq C\resolvent(\hbar^{-1}){\hbar}\|u-v\|_{H_\hbar^1(R_1+s-2\widetilde{\eta},R_1+s+2\widetilde{\eta})_r}\\ \nonumber
&\leq C\resolvent(\hbar^{-1})\hbar^{2}\bigg(C
\exp\Big(-\frac{(1-\eta)}{\hbar}\int_{R_1}^{{R_1+s}} \Phi_\theta(r)\, dr\Big)
 \|u-v\|_{H_{\hbar}^1(R_1, R_1+\widetilde{\eta})_r }\\ \nonumber
&\hspace{2.5cm}+C
\exp\Big(-\frac{(1-\eta)}{\hbar}\int_{{R_1+s}}^{R_{\tr}} \Phi_\theta(r)\, dr\Big)
\|u-v\|_{H_{\hbar}^1(R_{\tr}-\widetilde{\eta},R_{\tr})_r}
\bigg)\\ \nonumber
&\leq C\resolvent(\hbar^{-1})\hbar^2 \bigg(
\exp\Big(-\frac{(1-\eta)}{\hbar}\int_{R_1}^{{R_1+s}} \Phi_\theta(r)\,dr\Big)
\|u-v\|_{H_{\hbar}^1(R_1, R_1+\widetilde{\eta})_r }\\
&\hspace{2.5cm} + C \hbar \resolvent(\hbar^{-1})
\exp\Big(- \frac{(1-\eta)}{\hbar}\bigg[\int_{{R_1+s}}^{R_{\tr}} \Phi_\theta(r)\,dr+\int_{R_1}^{R_{\tr}}\Phi_\theta(r)\,dr\bigg]\Big)
\|\rhs\|_{\mc{H}}
\bigg)\label{eq:WAM2}
\end{align}
As at the end of the proof of Theorem \ref{t:blackBoxResolve}, $c_\e$ defined by \eqref{eq:ce} is $>0$ 
by the definition of $\theta_0$ \eqref{eq:theta0}.
Without loss of generality, let 
\beq\label{eq:eta_ineq2}
0<\eta< \min \left\{\frac12\,,\, \frac{c_\e}{2\inf_{\theta\in [\theta_0+\e, \pi/2-\e]}\int_{R_1}^{R_\tr}\Phi_{\theta}(r)\, dr} \right\},
\eeq
so that, for all $\theta\in [\theta_0+\e,\pi/2-\e]$, \eqref{eq:ce2} holds.
(We impose the condition \eqref{eq:eta_ineq2}, in contrast to \eqref{eq:eta_ineq1}, since we later require that $1-2\eta>0$.)
Let 
\beq\label{eq:BTP1}
\e_0:=\min\left\{\frac{c_{\e}}{4}\,,\, 2\eta\inf_{\theta\in [\theta_0+\e, \pi/2-\e]}\int_{R_1}^{R_{\tr}} \Phi_\theta(r)\,dr\right\}>0
\eeq
and choose $\e'$  such that
\beqs
\Lambda(P,J) -(1-\eta) \inf_{\theta\in [\theta_0+\e, \pi/2-\e]}\int_{R_1}^{R_1+\delta-\e'}\Phi_{\theta}(r)\, dr \leq -\frac{c_\e}{4}
\eeqs
and
\beqs
\Lambda(P,J) -(1-\eta) \sup_{\theta\in [\theta_0+\e, \pi/2-\e]}\int_{R_1}^{R_1+\e'}\Phi_{\theta}(r)\, dr \geq -\frac{\e_0}{2};
\eeqs
the latter is possible since $\Lambda(P,J)\geq 0$. Note that $\e'$ depends only on $\e$.
By the intermediate value theorem, there exists $\e'<s<\delta-\e'$, depending on $\theta$, such that
\beq\label{eq:BTP2}
\Lambda(P,J) -(1-\eta) \int_{R_1}^{R_1+ s} \Phi_\theta(r)\, dr = -\e_0.
\eeq
Then, using the definition of $\Lambda(P,J)$ \eqref{eq:LambdaPJ}, and
choosing $k$ large enough, depending only on $\e'$, $\eta$, and $\e_0$, and hence only on $\e$, we can absorb the term involving $\|u-v\|_{H^1_\hbar(R_1,R_1+\widetilde{\eta})_r}$ on the right-hand side of \eqref{eq:WAM2} into the left-hand side to obtain 
\begin{align*}
&\|\chi_2(u-v)\|_{\mc{H}}+\hbar^{2}\|\chi_2(u-v)\|_{\mc{D}}\\
&\hspace{3cm}\leq  C \hbar \resolvent(\hbar^{-1})
\exp\Big(- \frac{(1-\eta)}{\hbar}\bigg[\int_{{R_1+s}}^{R_{\tr}} \Phi_\theta(r)\,dr+\int_{R_1}^{R_{\tr}}\Phi_\theta(r)\,dr\bigg]\Big)
\|\rhs\|_{\mc{H}}.
\end{align*}
Now, by \eqref{eq:BTP2} and \eqref{eq:BTP1},
\begin{align*}
- \frac{(1-\eta)}{\hbar}\bigg[\int_{{R_1+s}}^{R_{\tr}} \Phi_\theta(r)\,dr+\int_{R_1}^{R_{\tr}}\Phi_\theta(r)\,dr\bigg]
&=- \frac{(1-\eta)}{\hbar}\bigg[2\int_{R_1}^{R_{\tr}}\Phi_\theta(r)\,dr - \int_{R_1}^{R_1+s} \Phi_\theta(r)\,dr\bigg]\\
&=- \frac{2(1-\eta)}{\hbar}\int_{R_1}^{R_{\tr}}\Phi_\theta(r)\,dr + \frac{1}{\hbar}\Big( \e_0 + \Lambda(P,J)\Big)\\
&\leq- \frac{2(1-2\eta)}{\hbar}\int_{R_1}^{R_{\tr}}\Phi_\theta(r)\,dr + \frac{1}{\hbar}\Lambda(P,J).
\end{align*}
We then use again the definition of $\Lambda(P,J)$ \eqref{eq:LambdaPJ} to obtain 
\beqs
\|\chi_2(u-v)\|_{\mc{H}}+\hbar^{2}\|\chi_2(u-v)\|_{\mc{D}}
\leq C \hbar 
\exp\bigg(- \frac{1}{\hbar}\Big(
2(1-2\eta)\int_{R_1}^{R_{\tr}}\Phi_\theta(r)\,dr
-2\Lambda(P,J)
\Big)\bigg)
\|\rhs\|_{\mc{H}}.
\eeqs
This obtains the bound \eqref{e:blackBoxErr} with $2-\eta$ replaced by $2- 4\eta$. Repeating the proof with $\eta$ replaced by $\eta/4$, the result follows.
\end{proof}

\section{Proof of Theorem \ref{t:expEst} (relative-error estimate for scattering by a plane wave)}
\label{s:relError}
Recall that $\Omega_{-}\subset \mathbb{R}^d$ is bounded and open with connected open complement, $\Omega_{\tr,+}=\Omega_{\tr}\setminus \overline{\Omega_-}$ is such that $B(0,R_{\tr})\subset \Omega_{\tr}$ for some $R_1<R_{\tr}$. 
Let $u^S$ and $v^S$ be the solutions to \eqref{e:helmholtz} and \eqref{e:PML}, respectively, and let
$
u^I(x) := \exp(ix\cdot a/\hbar).
$ 

The key ingredient for the proof of Theorem \ref{t:expEst}, on top of the result of Theorem \ref{t:blackBoxErr}, is the following lemma.

\begin{lem} \label{lem:contruI}
Let $R_0>0$ be such that $\Omega_-\Subset B(0,R_0)$. Given $R> R_0$ there is $C>0$ and $\hbar_0$ such that, for $0<\hbar<\hbar_0$,
$$
\Vert u^I \Vert_{L^2(B(0,R))} \leq C \Vert u^I  + u^S\Vert_{L^2(B(0,R) \backslash \Omega_-)}.
$$
\end{lem}
\begin{proof}
First observe that if 
$$
\|u^S\|_{L^2(B(0,R)\setminus \Omega_-)}\geq 2 \|u^I\|_{L^2(B(0,R))},
$$
then the claim follows from the triangle inequality.  Therefore,  without loss of generality, we can assume that $\|u^S\|_{L^2(B(0,R)\setminus \Omega_-)}\leq C<\infty$. 
Under this assumption, the argument involving the free resolvent in~\cite[Proof of Lemma 3.2]{GaSpWu:20} shows that, for any compact set $K\subset \mathbb{R}^d$, 
$$
\|u^S\|_{L^2(K\setminus \Omega_-)}\leq C_K.
$$

We now show that, for any $r>0$ and $\varphi \in C^\infty_c(\mathbb R^d \times \mathbb R^d;\mathbb{R})$ satisfying $\int_{\mathbb R^d} \varphi^2(x, a)  \,dx > 0$, there exists $C_{R, \varphi} >0$ such that,
for $\hbar>0$ sufficiently small,
\begin{equation} \label{eq:contrinc}
\Vert u^I \Vert_{L^2(B(0,R))} \leq C_{R, \varphi} \Vert \Op_{\hbar}(\varphi) u^I \Vert_{L^2}.
\end{equation}
Observe that, by the Fourier inversion fromula, for any $\psi \in C^\infty_c(\mathbb R^d \times \mathbb R^d)$, 
\begin{equation} \label{eq:defectplane}
\langle \Op_{\hbar}(\psi) u^I, u^I \rangle = \int_{\mathbb R^d} \psi(x, a) \, dx.
\end{equation}
Now, let $\phi \in C^\infty_c(\mathbb R^d)$ be such that $0\leq \phi \leq 1$, $\phi = 1$ in $B(0,R)$ and is supported in $B(0,2R)$. Using \eqref{eq:defectplane}, we obtain
$$
\Vert u^I\Vert^2_{L^2(B(0,R))} \leq \Vert \phi u^I\Vert^2_{L^2} = \int_{\mathbb R^d} \phi^2(x) dx \leq |B(0,2R)|.
$$
On the other hand, again using the Fourier inversion formula,
$$
\Vert \Op_{\hbar}(\varphi) u^I \Vert_{L^2}^2 = \int_{\mathbb{R}^d}\varphi^2(x, a) \, dx,
$$
and thus  (\ref{eq:contrinc}) follows with 
$$
C^2_{R,\varphi} := 2 |B(0,2R)| \Big( \frac 12 \int_{\mathbb{R}^d} \varphi^2(x, a) \, dx \Big)^{-1}.
$$

Let $x_0 \in \partial B(0, \frac{R_0+R}{2})$ and $V \subset T^* \mathbb R^d$ be such that 
$$
(x_0,a) \in V, \hspace{0.3cm} V \subset \big\{ (x,\xi)\,:\,\langle x,\xi\rangle<0\big\} \cap T^* \big( B(0, r) \backslash B(0,R_0) \big);
$$
{i.e., $a$ is an ``incoming'' direction at $x_0$.}
We take $\varphi \in C^\infty_c(\mathbb R_\xi^d )$, $\chi \in C^\infty_c(\mathbb R_x^d) $ so that $\operatorname{supp} \chi \subset B(0, R) \backslash B(0,R_0 )$, $\operatorname{supp}\varphi(\xi)\chi(x) \subset V$, and $\varphi = 1$ near $a$, $\chi = 1$ near $x_0$. Letting $\psi(x,\xi) := \varphi(\xi)\chi(x)$ and using (\ref{eq:contrinc}) we get
\begin{equation} \label{eq:uI1}
\Vert u^I \Vert_{L^2(B(0,R))}  \leq C_{R, \psi} \Vert \Op_{\hbar}(\psi)u^I \Vert_{L^2}.
\end{equation} 
We now write
\begin{equation} \label{eq:uI2}
\Op_{\hbar}(\psi)u^I = \Op_{\hbar}(\psi) (u^I + u^S) - \Op_{\hbar}(\psi) u^S.
\end{equation} 
By, e.g., \cite[Lemma 3.4]{GaLaSp:21}, 
 $\operatorname{WF}_{\hbar} (u^S) \cap \{(x,\xi)\,:\, \langle x,\xi\rangle <0,\,|x|>R_0\} = \emptyset $. Therefore, by, e.g., \cite[Proposition E.38]{DyZw:19},
  $$
\operatorname{WF}_h (\Op_{\hbar}(\psi)u^S) \subset \supp \psi \cap \operatorname{WF}_{\hbar}( u^S)  = \emptyset.
$$
By the definition of $\operatorname{WF}_h$ (see \S\ref{app:SCA}) and the fact that $u^S$ is uniformly bounded in $L^2_{\loc}$,
there is $C>0$ such that, {for $\hbar$ sufficiently small,}
$$
\Vert \Op_{\hbar}( \psi)u^S \Vert_{L^2} \leq C\hbar.
$$
Now, by (\ref{eq:defectplane}),
\beqs
\Vert u^I \Vert_{L^2(B(0,R))}  \geq \frac 12 |B(0,R)|.
\eeqs 
Therefore, with $C' := C \big( \frac 12 |B(0,R) | \big)^{-1}$, for $\hbar$ sufficiently small,
$$
\Vert \Op_{\hbar}(\psi) u^S \Vert_{L^2} \leq C'\hbar\Vert u^I \Vert_{L^2(B(0,R))}.
$$
Combining this last inequality with (\ref{eq:uI1}) and (\ref{eq:uI2}) and then using the fact that $\varphi(hD_x) \in \Psi^\infty$ together with the support properties of $\chi$, we obtain that, for $\hbar$ sufficiently small,
$$
\Vert u^I \Vert_{L^2(B(0,R))} \leq C \Vert \Op_{\hbar}(\psi) (u^I + u^S) \Vert_{L^2} \leq C\Vert u^I + u^S\Vert_{L^2(B(0,R) \backslash B(0, R_0 ))},
$$
and the proof is complete.
\end{proof}

\begin{rem}
The proof below shows that for any $0<R\leq R_1$ such that $\Omega_-\Subset B(0,R)$ we can replace the relative error
$$
\frac{\|u^S-v^S\|_{H^1(B(0,R_1)\setminus\Omega_-)}}{\|u^S+e^{ikx\cdot a}\|_{L^2(B(0,R_1)\setminus\Omega_-)}} \qquad \text{by}\qquad\frac{\|u^S-v^S\|_{H^1(B(0,R_1)\setminus\Omega_-)}}{\|u^S+e^{ikx\cdot a}\|_{L^2(B(0,R)\setminus\Omega_-)}}
$$
in Theorem~\ref{t:expEst}.
\end{rem}

\begin{proof}[Proof of Theorem \ref{t:expEst}]
Let   $0<R_0<R\leq R_1$ be such that $\Omega_-\Subset B(0,R_0)$ and let $\chi \in C^\infty_c(\mathbb R^d)$ be such that $\chi = 1$ near $B(0,R_0)$ and $\supp\chi \Subset B(0,R)$.
Observe that $u^S + \chi u^I$ and  $v^S + \chi u^I$ satisfy, respectively,
\begin{equation}
\begin{cases}
(-\hbar^2\Delta-1)(u^S + \chi u^I) = [-\hbar^2\Delta, \chi]u^I & \text{ in } \mathbb R^d \backslash \Omega_-, \\
B(u^S + \chi u^I )= 0 & \text{ on } \Gamma_-, \\
u^S +  \chi u^I \text{ is outgoing,}
\end{cases}
\end{equation}
and
\begin{equation}
\begin{cases}
(-\hbar^2\Delta_\theta-1)(v^S + \chi u^I) = [-\hbar^2\Delta_\theta, \chi]u^I & \text{ in } \mathbb R^d \backslash \Omega_-, \\
B(v^S + \chi u^I )= 0 & \text{ on } \Gamma_-, \\
v^S + \chi u^I = 0 & \text{ on } \Gamma_{\tr}.
\end{cases}
\end{equation}
Hence, by Theorem \ref{t:blackBoxErr}, there are $C,h_0>0$ such that, {with $\theta_0$ given by \eqref{eq:theta0}}, for $\theta_0+\e\leq \theta<\pi/2-\e$ and any $0<\hbar<\hbar_0$,
\begin{align}\nonumber
&\Vert (u^S + \chi u^I)- (v^S + \chi u^I ) \Vert_{H_{\hbar}^1(\Omega_{\tr} \backslash \Omega_-)} \\
&\hspace{1.5cm} \leq C\exp\bigg(-k\Big((2-\eta)\int_{R_1}^{R_\tr}\Phi_\theta(r)\,dr-2\Lambda(P,J)\Big)\bigg)
\Vert [-\hbar^2\Delta, \chi]u^I \Vert_{L^2(\Omega_{\tr} \backslash \Omega_-)}.
 \label{eq:plane1}
\end{align}
Since $\hbar \nabla u^I = ia u^I$,
\begin{equation} \label{eq:plane2}
\Vert [-\hbar^2\Delta, \chi]u^I \Vert_{L^2(\Omega_{\tr} \backslash \Omega_-)} \leq C\hbar \Vert u^I \Vert_{H^1_{\hbar}(B(0,R))}
\leq C \hbar \Vert u^I \Vert_{L^2(B(0,R))}.
\end{equation}
We now apply Lemma \ref{lem:contruI}.
 We obtain, reducing $\hbar_0$ again if necessary, that for $0<\hbar<\hbar_0$,
\begin{equation} \label{eq:plane3}
\Vert u^I \Vert_{L^2(B(0,R))} \leq C\Vert u^I  + u^S\Vert_{L^2(B(0,R) \backslash \Omega_-)}
\end{equation}
The result \eqref{e:relError1} then follows by combining (\ref{eq:plane1}), (\ref{eq:plane2}), and (\ref{eq:plane3}).
\end{proof}

\section{Nontrapping estimate on the free resolvent with rough scaling}
\label{s:rough_res}

The goal of this section is to prove Theorem \ref{t:nontrappingScale}. This section uses notions of rough semiclassical pseudo-differential operators  recapped in \S\ref{sec:rough}.
We first prove a propagation result.

\begin{lem} \label{lem:inva}
Assume that $Q \in C^{1,\alpha}\Psi^2 + \hbar C^{0,\alpha} \Psi^1$ is such that, for any $w \in H_\hbar^{2}$, 
\begin{equation} \label{eq:hypcom1}
 \langle \operatorname{Im} Q w, w\rangle \leq C_0 \hbar \Vert w\Vert_{H_\hbar^{\frac{1}{2}}}^{2},
\end{equation} 
and that $\sigma_{\hbar}(Q)\to q$ with $q$ satisfying
\begin{equation}
\label{e:classicalElliptic}
|q(x,\xi)|\geq c|\xi|^2,\qquad |\xi|\geq C.
\end{equation}
Given $\rhs\in L^2$, with $\|\rhs\|_{L^2}\leq C'$ with $C'$ independent of $\hbar$, let $u$ satisfy $Q u = \hbar \rhs$. Let 
$u$ have defect measure $\mu$ as $\hbar \to 0$ (in the sense of \eqref{eq:defectmeasure}) and let $u$ and $\rhs$ have joint measure $\mu^j$ (in the sense of \eqref{eq:jointmeasure}).

Then, (i) the measure $\mu$ is supported 
in $\{ q = 0\}$, (ii) for $b \in S^1$ and $\chi \in C_c^\infty$, as $\hbar\to0$,
\begin{equation} \label{eq:inva1}
 \Vert \operatorname{Op}_\hbar(b) \chi u \Vert^2_{L^2} \rightarrow \mu(|b|^2\chi^2),
\end{equation}
and (iii)  for any real-valued
$a\in C^\infty_c(T^*\mathbb R^d)$,
\begin{equation} \label{eq:inva2}
\mu(H_{\operatorname{Re} q}a^2 + C_0\langle \xi \rangle a^2) \geq -2 \operatorname{Im}\mu^j(a^2).
\end{equation}
\end{lem} 

\begin{proof}
The fact that $\operatorname{supp}\mu \subset \{ q = 0 \}$ and (\ref{eq:inva1}) are shown in \cite[Proof of Lemma 3.6]{GaSpWu:20}, where the only assumptions used are that (a) the
operator associated to the equation is in $C^{1,\alpha}\Psi^2 + \hbar C^{0,\alpha} \Psi^1$ and (b) the principal symbol satisfies the bound~\eqref{e:classicalElliptic}. We therefore only have to show (\ref{eq:inva2}).

Let $A := \operatorname{Op}_\hbar(a)$. 
Following the calculations in~\cite[Equation 2.38]{GaMaSp:21}, we have
\begin{align}\nonumber
-2\hbar^{-1} \operatorname{Im} \langle A^* A u, Qu \rangle &= \hbar^{-1} \operatorname{Im} \big\langle (A^*A \operatorname{Re}Q - \operatorname{Re}Q A^*A)u, u\big\rangle + 2\hbar^{-1}\operatorname{Re} \big\langle A^*A \operatorname{Im}Qu, u\big\rangle \nonumber \\ \nonumber
&= \hbar^{-1} \operatorname{Im} \big\langle [A^*A, \operatorname{Re}Q] u, u\big\rangle + 2\hbar^{-1}\operatorname{Re} \big\langle \operatorname{Im}QAu, Au\big\rangle\\ \nonumber
&\hspace{5cm} + 2\hbar^{-1}\operatorname{Re} \big\langle A^*[A,\operatorname{Im}Q]u , u\big\rangle.\\ \nonumber
&\leq  \hbar^{-1} \operatorname{Im} \big\langle [A^*A, \operatorname{Re}Q] u, u\big\rangle  
+2C_0 \Vert Au\Vert_{H_\hbar^{\frac{1}{2}}}^2\\
&\hspace{5cm} + 2\hbar^{-1}\operatorname{Re} \big\langle A^*[A,\operatorname{Im}Q]u , u\big\rangle.
\label{eq:Sat2}
\end{align}
by \eqref{eq:hypcom1}.
We now examine each of the terms in \eqref{eq:Sat2}, starting with the term on the left-hand side. By \eqref{eq:symbol} and the fact that $a$ is real, $\sigma_\hbar(A^*A) = a^2$; 
using this and the fact that $\rhs$ is bounded in $L^2$ uniformly in $\hbar$, we have
$$
\big| 2\hbar^{-1} \operatorname{Im} \langle \big(A^* A - \operatorname{Op}_\hbar(a^2) \big) u, Qu \rangle \big| \leq  2 \Vert A^* A - \operatorname{Op}_\hbar(a^2) \Vert_{L^2} \Vert \rhs\Vert_{L^2} \rightarrow 0;
$$
hence, by the definition of $\mu^j$ \eqref{eq:jointmeasure},  as $\hbar\to 0$,
\begin{equation} \label{eq:com2}
-2\hbar^{-1} \operatorname{Im} \langle A^* A u, Qu \rangle  = -2 \operatorname{Im} \langle \operatorname{Op}_\hbar(a^2) u, f \rangle + o(1)  \rightarrow -2 \operatorname{Im} \mu^j(a^2).
\end{equation}
For the first term on the right-hand side of \eqref{eq:Sat1},
by Lemma~\ref{l:commutator}, as $\hbar\to 0$,
\begin{equation}\label{eq:Sat3}
\hbar^{-1} \operatorname{Im} \langle [A^*A, \operatorname{Re}Q] u, u \rangle \rightarrow \mu(H_{\operatorname{Re} q} a^2),
\end{equation}
By the definition of $\mu$ \eqref{eq:defectmeasure} and of the semi-classical Sobolev norms, as $\hbar\to 0$,
\beq\label{eq:Sat4}
\Vert Au \Vert^2_{H_\hbar^{\frac{1}{2}}} \rightarrow \mu(\langle \xi \rangle a^2)
\eeq
 By Lemma~\ref{l:commutator} and the fact that $a$ is real, as $\hbar\to 0$,
\beq\label{eq:Sat1}
\hbar^{-1}\langle \Re  A^*[A,\Im Q]u,u\rangle\to 0.
\eeq
The result \eqref{eq:inva2} then follows from using  in \eqref{eq:Sat2}  the limits \eqref{eq:com2}, \eqref{eq:Sat3},  \eqref{eq:Sat4}, and \eqref{eq:Sat1}.
\end{proof}

We now show that when $\operatorname{Re} q$ is sufficiently regular, invariance statements of type (\ref{eq:inva2}) can be translated to
 invariance statements at the level of the Hamiltonian flow. In this lemma, the assumption $p\in C^2$ ensures that the Hamiltonian flow is well defined; this is where the assumption $f \in C^3$ in our main results originates.

\begin{lem} \label{lem:inva_C2}
Let $\mu$ be a Radon measure on $T^* \mathbb R^d$ such that for any real-valued
$a\in C^\infty_c(T^*\mathbb R^d)$ and $p \in C^2$,
\begin{equation} \label{eq:inva3}
\mu(H_{p}a^2 + C_0\langle \xi \rangle a^2) \geq 0.
\end{equation}
Let $\varphi_t$ be the Hamiltonian flow associated to $p$.
Then, for any measurable $B$, and for all $t\geq 0$,
\begin{equation*}
\mu(\varphi_t(B)) \leq \mu(B) +  C_0 \sup_{(x,\xi) \in B}\langle \xi \rangle \int_{0}^{t} \mu(\varphi_{s}(B)) \, ds.
\end{equation*}
\end{lem}

\begin{proof}
We first show that (\ref{eq:inva3}) remains valid for $a \in C^1_c$. 
To do so, let $a \in C^1_c$. Let $\phi \in C^\infty_c$ be such that $\phi \geq0 $, $\operatorname{supp} \phi \subset B(0,1)$,
and $\int \phi = 1$. For $\epsilon >0$, let $\phi_\epsilon := \epsilon^{-d} \phi(\cdot/\epsilon)$, and define $a_\epsilon := a * \phi_\epsilon \in C^\infty_c$. Since $H_p a$ is continuous, $H_p a_\epsilon = (H_p a)* \phi_\epsilon \rightarrow H_p a$ pointwise. Similarly, $a_\epsilon \rightarrow a$ pointwise. Hence $H_p a^2_\epsilon = 2 a_\epsilon H_ p a_\epsilon \rightarrow 2 a H_p a = H_p a^2$ pointwise. In addition, since the derivatives of $p$ are bounded on $\operatorname{supp} a$, for $0<\epsilon \leq 1$,
$$
|H_p a_\epsilon(\rho)| = \Big| \int H_p a(\rho - \epsilon \zeta) \phi(\zeta) \, d\zeta \Big| \leq C \indicator_{\rho \in \operatorname{supp}a + B(0,1)}.
$$
Similarly $|a_\epsilon(\rho)| \leq C' \indicator_{\rho \in \operatorname{supp}a + B(0,1)}$. Hence $|H_p a_\epsilon^2(\rho)| \leq 2CC'  \indicator_{\rho \in \operatorname{supp}a + B(0,1)}$ and thus, by dominated convergence, $\mu(H_p a_\epsilon^2) \rightarrow \mu (H_p a^2)$. In a similar way, $\mu(\langle \xi \rangle a_\epsilon^2) \rightarrow \mu (\langle \xi \rangle a^2)$; hence
$$
\mu(H_p a_\epsilon^2 + C_0 \langle \xi \rangle a_\epsilon^2) \rightarrow \mu (H_p a^2 + C_0 \langle \xi \rangle a^2).
$$
By (\ref{eq:inva3}), the left-hand side is non-negative; since $a_\epsilon \in C^\infty_c$, so is the right-hand side, and hence  (\ref{eq:inva3}) remains true for $a \in C^1_c$.

Now let $a \in C^\infty_c$.
Since the derivatives of $p$ are bounded on $\operatorname{supp} a$, by Hamilton's equations $\partial_s \varphi_s$ is
bounded on $\big\{ \varphi_s \in \operatorname{supp} a \big\}$ independently of time, and hence
 $$
 |\partial_s (a^2 \circ \varphi_s)| \leq C \indicator_X, \quad  \tfa (s, (x,\xi)) \in [-t, 0] \times T^*\mathbb R^d,
 $$
 where
  $$
  X:= \bigcup_{s \in [0,t]} \varphi_s(\operatorname{supp} a).
 $$
By the dominated convergence theorem, interchanging the derivative and integral, we have
$$
\mu(a^2 \circ \varphi_{-t}) - \mu(a^2) = - \int_{-t}^0 \partial_s \Big( \int a^2 \circ \varphi_s\, d\mu \Big) \, ds = - \int_{-t}^0 \int \partial_s(a^2 \circ \varphi_s)\, d\mu \, ds.
$$
Since $p \in C^2$ and $\varphi_s \in C^1_c$ for any $s$, $a^2 \circ \varphi_s \in C^1_c$ for any $s$. Therefore, using (\ref{eq:inva3}),
$$
\mu(a^2) - \mu(a^2 \circ \varphi_{-t})  = \int_{-t}^{0} \int H_p a^2 \circ \varphi_s \, d\mu \, ds \geq - C_0  \int_{-t}^{0} \int \langle \xi\rangle \; a^2 \circ \varphi_s\, d\mu \, ds.
$$
The result follows by approximating $\mathbf{1}_B$ by squares of smooth, compactly-supported symbols.
\end{proof}

As a consequence, we obtain the following resolvent estimate.
\begin{lem} \label{lem:res_gen}
Let $(Q_{\theta})_{\theta \in \Theta}$ be a family of (rough) semiclassical pseudo-differential operators with $\Theta \subset \mathbb R$ compact.
We assume that  $Q_\theta \in C^{1,\alpha}\Psi^2 + \hbar C^{0,\alpha} \Psi^1$ uniformly in $\theta \in \Theta$. We assume further that (i) there exists $C_0>0$ such that for any $\theta \in \Theta$ and any $w \in H_\hbar^{2}$, 
\begin{equation} \label{eq:hypcom}
\langle \operatorname{Im} Q_\theta w, w\rangle \leq C_0 \hbar \Vert w\Vert^2_{H_{\hbar}^{\frac{1}{2}}},
\end{equation} 
(ii) $\sigma_{\hbar}(Q_\theta) \rightarrow q_\theta$ where $q_\theta \in C^2$ and depends smoothly on $\theta \in \Theta$ together with its derivatives, (iii) $q_\theta$ satisfies~\eqref{e:classicalElliptic} uniformly in $\theta \in \Theta$, and (iv)
\begin{multline}\label{eq:esc_ell}
\exists \eta>0, \; \forall \theta_0 \in \Theta, \; \forall (x_0, \xi_0) \in \big\{ q_{\theta_0} = 0 \big\}, \; \exists \tau_{\theta_0}^*(x_0,\xi_0)>0, \\
 \varphi^{\theta_0}_{-\tau_{\theta_0}^*(x_0,\xi_0)}(x_0, \xi_0) \in  \bigcap_{\theta \in \Theta} \Big\{ \langle \xi \rangle^{-2} |q_\theta(x,\xi)| \geq \eta \Big\},
\end{multline}
where $\varphi^\theta_t$ is the Hamiltonian flow associated with $\operatorname{Re} q_\theta$.

Then, there exists $C>0$ and $\hbar_0>0$ such that, for any $\theta \in \Theta$,
if $u \in L^2$ is a solution of
$$
Q_\theta u = \hbar \rhs,
$$
with $\rhs \in L^2$,
then, for $0<\hbar\leq h_0$, $1\leq s \leq 2$,
$$
\Vert u \Vert_{H_{\hbar}^{s}} \leq C\Vert \rhs \Vert_{H_{\hbar}^{s-2}}.
$$
\end{lem}

\begin{proof} 
For $\delta>0$, let 
$$
\mathcal E_\delta := \bigcap_{\theta \in \Theta} \Big\{ \langle \xi \rangle^{-2} |q_\theta(x,\xi)| \geq \delta \Big\}.
$$

We begin by showing two elliptic estimates (\eqref{eq:res_ell_conc} and \eqref{e:res_ell_freq} below). Let $b \in S^0(T^*\mathbb R^d)$ be such that $b = 1$ on $\mathcal E_{\eta/2}$ and $\operatorname{supp}b \subset \mathcal E_{\eta / 4}$.
We write $Q_\theta= \operatorname{Op}_{\hbar} q^\theta_0 + \hbar \operatorname{Op}_{\hbar} q^\theta_1$
with $q^\theta_0 \in C^{1,\alpha} S^2$ and $q^\theta_1 \in C^{0,\alpha} S^1$ uniformly in $\hbar \rightarrow 0$. Let $\psi \in C^\infty_c(\mathbb R)$ be such that $\psi = 1$ on $[-2, 2]$, and for $\epsilon >0$ we define
$q^\theta_{0, \epsilon}(x, \xi) := (\psi(\epsilon |D_x|) q^\theta_0)(x, \xi)$. Then $q^\theta_{0,\epsilon} \in S^{2}$ and by Littlewood-Paley
(see, e.g., \cite[\S7.5.2]{Zworski_semi}),
\begin{equation} \label{eq:reg_LP}
\sup_{\theta \in \Theta} \Vert D^\beta_\xi (q^\theta_{0,\epsilon}(\cdot, \xi) - q^\theta_0(\cdot, \xi)) \Vert_{C^{0,\alpha}} \leq C\epsilon \langle \xi \rangle^{2-|\beta|},
\end{equation}
where $C$ is independent of $\epsilon$ and the uniformity in $\theta$ comes from the fact that all the involved quantities depend continuously on $\theta$ and $\Theta$ is compact.
In particular, by (\ref{eq:reg_LP}), for $\epsilon>0$ and $0<h<h_0$ small enough, $q^\theta_{0,\epsilon}$ is elliptic on  $\operatorname{supp} b$, uniformly in $\epsilon>0$ and $\theta \in \Theta$. Therefore, by the elliptic parametrix (Theorem \ref{t:ellip}),
there exists $S_{\epsilon, \theta} \in  \Psi^{s-2}$, bounded uniformly from $H^m_\hbar$ to $H^{m-s+2}_\hbar$ in $\epsilon>0$ and $\theta \in \Theta$, and such that
$$
\langle \hbar D\rangle^s b(x, \hbar D_x) = S_{\epsilon, \theta} \operatorname{Op}_h(q^\theta_{0,\epsilon})  + O(\hbar^\infty)_{\Psi^{-\infty}},
$$
and thus
\begin{equation} \label{eq:res_ell_1}
\langle \hbar D\rangle^{s} b(x, \hbar D_x)  = S_{\epsilon, \theta}  Q - S_{\epsilon, \theta}  \hbar \operatorname{Op}_\hbar(q^\theta_{1}) + S_\epsilon (\operatorname{Op}_\hbar(q^\theta_{0, \epsilon}) - \operatorname{Op}_\hbar(q^\theta_{0})) + O(\hbar^\infty)_{\Psi^{-\infty}}.
\end{equation}
But, by (\ref{eq:reg_LP}) together with Lemma \ref{lem:rough_calc},
\begin{equation} \label{eq:res_reg1}
\sup_{\theta \in \Theta}\Vert \operatorname{Op}_\hbar (q^\theta_{0, \epsilon}) - \operatorname{Op}_\hbar(q^\theta_{0}) \Vert_{H^2_\hbar \rightarrow L^2} \leq C \epsilon,
\end{equation}
where $C$ is independent of $\e$ and $\hbar$. In addition, by Lemma  \ref{lem:rough_calc} again, $\operatorname{Op}_\hbar(q^\theta_{1})  \in \mathcal L(H^1_\hbar, L^2)$ uniformly in $\hbar$ and $\theta \in \Theta$.
Thus, using the fact that $S_{\epsilon,\theta} \in  \Psi^{0}$ uniformly in $\epsilon>0$ small and $\theta \in \Theta$, (\ref{eq:res_ell_1}), and (\ref{eq:res_reg1}), we find that
$$
\langle \hbar D\rangle^s b(x, \hbar D_x)  = S_{\epsilon, \theta} Q + O(\hbar )_{H_\hbar ^1 \rightarrow H_\hbar ^{2-s}} + O(\epsilon)_{H_\hbar ^2\rightarrow H_\hbar ^{2-s}}.
$$
Evaluating in $w \in H^1_\hbar$ and letting $\epsilon \rightarrow 0$, we conclude that there exists $C>0$ such that for $\hbar$ small enough and any $\theta \in \Theta$
\begin{equation} \label{eq:res_ell_conc}
\Vert b(x, \hbar D_x)  w\Vert_{H_\hbar ^{s}} \leq C \Vert Q_\theta w \Vert_{H_\hbar ^{s-2}} + C \hbar  \Vert w \Vert_{H_\hbar ^1}, 
\quad \tfa w \in H^1_\hbar.
\end{equation}
A near-identical argument, using \eqref{e:classicalElliptic},  shows that for $\psi\in C_c^\infty([-2,2])$ with $\psi\equiv 1$ in $[-1,1]$, and $K$ large enough,
for any $\theta \in \Theta$
\begin{equation}
\label{e:res_ell_freq}
\|(1-\psi(K^{-1}|\hbar D_x|))(1-b)(x, \hbar D_x)  w\|_{H_\hbar^{s}}\leq C'\Vert Q_\theta w \Vert_{H_\hbar ^{s-2}} + C'\hbar \Vert w \Vert_{H_\hbar^1} \quad \tfa w \in H^1_\hbar.
\end{equation}

Now, if the conclusion of the Lemma fails, there exists $w_n$, $\rhs_n$, $\theta_n \in \Theta$ and $\hbar_n \rightarrow 0$ such that
$$
Q_{\theta_n}(\hbar_n)w_n = \hbar_n \rhs_n, \hspace{0.5cm}\Vert  w_n \Vert_{H_{\hbar_n}^s} > n  \Vert  \rhs_n \Vert_{H_{\hbar_n}^{s-2}}.
$$
Normalising, we can assume that
\begin{equation} \label{eq:res_norma}
\Vert  w_n \Vert_{H_{\hbar_n}^s} = 1, \hspace{0.5cm}
\Vert  \rhs_n \Vert_{H_{\hbar_n}^{s-2}} = o(1).
\end{equation}
Therefore, extracting subsequences, we can assume that $w_n$ has defect measure $\omega$. In addition, as $\Theta$ is compact, we can assume that
$\theta_n \rightarrow \bar \theta \in \Theta$.

Now, by~\eqref{eq:res_ell_conc} and~\eqref{e:res_ell_freq}, 
\begin{multline*}
\|(1-\psi(K^{-1}|\hbar_nD_x|))(1-b(x, \hbar D_x)) w_n\|_{H_{\hbar_n}^{s}}+\|b(x, \hbar D_x)w_n\|_{H_{\hbar_n}^s} \\ 
\leq \hbar_n(\|\rhs_n\|_{H_{\hbar_n}^{s-2}}+\|w_n\|_{H_{\hbar_n}^s})=O(\hbar_n),
\end{multline*}
and in particular
\begin{multline*}
1+O(\hbar_n)=\| \psi(K^{-1}|\hbar_nD_x|)  (1-b(x, \hbar D_x)) w_n\|_{H_{\hbar_n}^s} \\ 
\leq C_K\|
 \psi(K^{-1}|\hbar_nD_x|)(1 - \tilde b(x, \hbar D_x)) w_n\|_{L^2}\leq C_K.
\end{multline*}
Thus, by the support properties of $b$ and $\psi$
\begin{equation}
\label{e:contradictMe}
\begin{gathered}
\omega \big(\mathcal E_{\eta/4} ^c \cap \{ |\xi|\leq 2K \} \big)>c_K,\qquad \omega(\mathcal E_{\eta/2})=0, \qquad \omega(|\xi| \geq 2K) = 0.
\end{gathered}
\end{equation}

 Next, observe that letting $u_n :=\tilde\psi(K^{-1}|\hbar_n D_x|)(1-\tilde b(x, \hbar D_x)) w_n$, with $\tilde{\psi}\in C_c^\infty(\mathbb{R})$, $\tilde{\psi}\equiv 1$ on $[-2,2]$ and $\tilde{b}\in S^0(\mathbb{R}^d)$ with $b = 1$ on $ \operatorname{supp} b$, we have
$$
\begin{gathered}
Q_{\theta_n}u_n=\tilde\psi(K^{-1}|\hbar_nD|)(1 - \tilde b(x, \hbar D_x)) \hbar_n\rhs_n+[Q,\tilde \psi(K^{-1}|\hbar D_n|)(1 - \tilde b(x, \hbar D_x)])w_n=:\hbar_n\widetilde{\rhs}_n.
\end{gathered}
$$
and $u_n$ has defect measure $\mu:=\tilde{\psi}^2(K^{-1}|\xi|)(1 - \tilde b)^2(x,\hbar D_x)\omega$. 
Now, by Lemma~\ref{l:commutator} 
\begin{multline*}
\hbar_n^{-1}\Big\langle \big[Q_{\theta_n},\tilde\psi(K^{-1}|\hbar D_n|)(1 - \tilde b(x, \hbar D_x) \big]w_n, \; \tilde\psi(K^{-1}|\hbar_nD|)(1 - \tilde b(x, \hbar D_x)w_n\Big\rangle \\
\to \omega(\overline{\psi(K^{-1}|\xi|)(1 - \tilde b(x, \hbar D_x))}H_{q}\psi(K^{-1}|\xi|)(1 - \tilde b(x, \hbar D_x))=0,
\end{multline*}
since $\supp H_{q_0}\psi(K^{-1}|\xi|)(1-b(x, \hbar D_x))\cap \{|\xi|\leq 2K\}=\emptyset$ and $\omega(|\xi|\geq 2K)=0$.

In particular, this implies
$$
\|\widetilde{\rhs}_n\|_{L^2}\leq C\|\rhs_n\|_{H_{\hbar_n}^{s-2}}+o(1)=o(1).
$$
Therefore, $u_n$ and $\widetilde{\rhs}_n$ have joint defect measure equal to 0, and hence, by Lemma \ref{lem:inva} applied with $q:=q_{\bar \theta}$, together with Lemma \ref{lem:inva_C2}, 
for any measurable $B$, denoting $\varphi := \varphi^{\bar \theta}$
\begin{equation*}
\mu(\varphi_t(B)) \leq \mu(B) +  C_0 \sup_{(x,\xi) \in B}\langle \xi \rangle \int_{0}^{t} \mu(\varphi_{s}(B)) \, ds, \quad\tfa t>0,
\end{equation*}
and thus, by a Gr\"onwall inequality
\begin{equation} \label{eq:res_mes1}
\mu(\varphi_t(B)) \leq \mu(B) \times \exp\big({C_0 \sup_{(x,\xi) \in B} \langle \xi \rangle \; t} \big) \quad\tfa t>0.
\end{equation}
But, by~\eqref{e:contradictMe}, 
$\mu(\mathcal E_{\eta/2}) = 0$.
Together with (\ref{eq:res_mes1}), this implies that $\mu$ is identically zero. Indeed, let  $(x,\xi) \in \{ q = 0\}$ be arbitrary. By (\ref{eq:esc_ell}), if $B = \mathcal V(x,\xi) \cap \{ q=0 \}$ where $\mathcal V(x,\xi)$ is a sufficiently small neighbourhood of $(x,\xi)$, there exists $\tau^* = \tau^*(B)>0$ such that $\varphi_{-\tau^*}(B) \subset \mathcal E_{\eta/2}$ , and hence $\mu (\varphi_{-\tau^*}(B))= 0$, from which 
$$
\mu(B) = \mu(\varphi_{\tau^*} (\varphi_{- \tau^*} (B))) \leq \mu(\varphi_{-\tau^*}(B)) \times \exp\big({C_0 \sup_{\{ q_{\bar \theta} = 0\}}\langle \xi \rangle \; \tau^*} \big) = 0,
$$
where we used the fact that $\sup_{\{ q_{\bar \theta} = 0\}} |\xi|  < \infty$. This is a contradiction with the fact that, by \eqref{e:contradictMe}, $\mu( \mathcal E^c_{\eta/4}) \geq c>0$.
\end{proof}

We now show that the scaled operator satisfies the uniform escape-to-ellipticity condition (\ref{eq:esc_ell}) under
suitable uniformity assumptions for $F_\theta$.

\begin{lem} \label{lem:esc_ell} 
Let $-\Delta_\theta$ be as in \S\ref{sec:A2}, and
let $p_\theta$ be its principal symbol. 
Let 
$\Theta \Subset (0, \pi/2)$. Assume that $F_\theta \in C^{4}$ uniformly in $\theta \in \Theta$. 
Then, there exists $\nu = \nu(\Theta)>0$ such that, for any $\theta_0 \in \Theta$ and any $(x_0, \xi_0) \in \{ \operatorname{Re} p_{\theta _0}= 1 \} $, there exists $\tau_{\theta_0}^*(x_0,\xi_0)>0$ so that the trajectory $\varphi^{\theta_0}_t(x_0,\xi_0)$ of the Hamiltonian flow associated to $\operatorname{Re} p_{\theta_0}$ and starting from $(x_0, \xi_0)$ satisfies
$$
\varphi^{\theta_0}_{ - \tau_{\theta_0}^*(x_0, \xi_0)}(x_0, \xi_0) \in \bigcap_{\theta \in \Theta} \Big\{ \langle \xi \rangle^{-2} |p_\theta(x,\xi) - 1| \geq \nu \Big\}.
$$
\end{lem}
\begin{proof}
Suppose the conclusion fails. Then there are $\{\theta_n\}_{n=1}^\infty$ and $\{(x_n,\xi_n)\}_{n=1}^\infty$ such that 
$$
\varphi_{-t}^{\theta_n}(x_n,\xi_n)\subset \big\{\langle \xi\rangle^{-2}|p_{\theta_n}-1|\leq n^{-1}\big\} \tfa t\geq 0.
$$
Since $\Theta $ is compact, we can assume $\theta_n\to \theta\in \Theta$. Moreover, since there exist $c,C>0$ such that, for all $\theta\in \Theta$,
\begin{equation}
\label{e:eventualEllipticity}
|p_\theta(x,\xi)-1|\geq c\langle \xi\rangle^2-C \tfa (x,\xi),\qquad \text{and}\qquad|p_\theta(x,\xi)-1|\geq c \tfa |x|\geq C,
\end{equation}
we can assume that $(x_n,\xi_n)\to (x_0,\xi_0)$. Now, for any fixed $t\geq 0$, $\varphi_{-t}(x_n,\xi_n)\to \varphi_{-t}(x_0,\xi_0)$. Therefore, 
$$
\varphi_{-t}(x_0,\xi_0)\subset \{|p_\theta-1|=0\}\qquad \tfa t\geq 0.
$$

Now, by~\eqref{e:absVal}
$$
\operatorname{Im} p_{\theta} (x,\xi) = -2 \big\langle F''_\theta(x) (I+F''_\theta(x)^2)^{-1}\xi,  (I+F''_\theta(x)^2)^{-1}\xi \big\rangle.
$$
Therefore, when $\Im p_\theta(x,\xi)=0$, since $F_\theta''(x)\geq 0$, this implies $F_\theta''(x)(I+F''_\theta(x)^2)^{-1}\xi=0$ and hence,  
$$
\xi =(I+(F''_\theta(x))^2)(I+(F''_\theta(x))^2)^{-1}\xi=(I+(F''_\theta(x))^2)^{-1}\xi.$$
Now, again by~\eqref{e:absVal}
$$
\Re p_\theta(x,\xi)= \langle (I+(F_\theta''(x))^2)^{-1}\xi,(I+(F_\theta''(x))^2)^{-1}\xi\rangle.
$$
Therefore, when $p_\theta(x,\xi)=1$,
$$
\begin{gathered}
\Re p_\theta(x,\xi)=|\xi|^2=1,\qquad \partial_\xi \Re p_\theta= 2(I+(F_\theta''(x))^2)^{-1}(I+(F_\theta''(x))^2)^{-1}\xi=2\xi,
\end{gathered}
$$
and, since $F_\theta''(x)$ is symmetric, and $F_\theta''(x)\xi=0$,
$$
\begin{aligned}
\partial_{x_i}\Re p_\theta& = -2\langle (I+(F''_\theta(x))^2)^{-1}( \partial_{x_i}F_\theta'' F''_\theta(x)+F_\theta'' \partial_{x_i}F''_\theta(x))(I+(F''_\theta(x))^2)^{-1}\xi,(I+(F''_\theta(x))^2)^{-1}\xi\rangle\\
&= -2\langle ( \partial_{x_i}F_\theta'' F_\theta(x)''+F_\theta'' \partial_{x_i}F_\theta(x)'')\xi,\xi\rangle=0.
\end{aligned}
$$
In particular, $H_{\Re p_\theta}=2\langle \xi,\partial_x\rangle$ and $|\xi|=1$ on $\{p_\theta-1=0\}$. Thus, we have 
$$
\varphi_{-t}(x_0,\xi_0)=(x_0-t\xi_0,\xi_0)\subset \{p_\theta=1\}\qquad \tfa t\geq0,
$$
which contradicts~\eqref{e:eventualEllipticity}.
\end{proof}

\begin{proof}[Proof of Theorem \ref{t:nontrappingScale}]
 We let $Q_\theta := - \hbar^2\Delta_\theta - 1$ and check that $Q_\theta$ satisfies the assumptions of Lemma \ref{lem:res_gen} with $\Theta := [\epsilon, \pi/2 - \epsilon]$. Lemma \ref{lem:esc_ell} shows that the escape-to-ellipticity condition \eqref{eq:esc_ell} is satisfied, where $F_\theta \in C^3$ uniformly in $\theta \in [\epsilon, \pi/2 -\epsilon]$ since $f_\theta(r) = \tan \theta f(r)$ with $f$ satisfying \eqref{e:fProp} and the functions $F_\theta$ and $f_\theta$ are related by Lemma~\ref{l:tiger}.
Moreover, since for such a scaling function $\sup_{\epsilon\leq\theta\leq \pi/2 - \epsilon} \|(I+iF_\theta''(x))\|\leq C$,
$$
\inf_{\epsilon\leq\theta\leq \pi/2 - \epsilon} |\xi|^{-2}|\sigma(-\hbar^2\Delta_\theta-1)| >0,\qquad |\xi|\geq C,
$$
and hence~\eqref{e:classicalElliptic} holds uniformly in $\theta \in [\epsilon, \pi/2 - \epsilon]$. 
Finally,~\eqref{eq:hypcom} follows from~\eqref{eq:A41} and~\eqref{e:absVal}; indeed, for $u\in H_\hbar^2$ 
\begin{equation}
\label{e:Hhalf}
\Im \langle -\hbar^2\Delta_\theta u,u\rangle \leq  \Im \langle \hbar^2 A(x)\partial_x u,u\rangle \leq C\hbar\|u\|_{H_\hbar^{1/2}}^2,
\end{equation}
where $C>0$ can be taken uniform in $\theta$ thanks again to the particular form of the scaling function.
To see the last inequality in~\eqref{e:Hhalf}, observe that
$$
\langle A(x)\hbar \partial_x u,u\rangle=\langle \langle\hbar D\rangle^{-1/2} A(x)\langle \hbar D\rangle^{1/2} \langle \hbar D\rangle^{-1/2}\hbar \partial_x u,\langle \hbar D\rangle^{1/2}u\rangle,
$$
and thus it suffices to observe that, since $A(x)\in C^1$, $A:H_{\hbar}^{-1/2}\to H_{\hbar}^{-1/2}$ is bounded by Lemma~\ref{lem:rough_calc}.

Therefore, Lemma \ref{lem:res_gen} applies to $Q_{\theta} := - \hbar^2\Delta_\theta - 1$, $\Theta := [\epsilon, \pi/2 - \epsilon]$. Let $-1\leq s\leq 0$, and $\lambda>\hbar_0^{-1}$ where $\hbar_0$ is given by  Lemma \ref{lem:res_gen}.  Then Lemma \ref{lem:res_gen}  implies 
$$
\|u\|_{H_\hbar^{s+2}}\leq C\hbar^{-1}\|(-\hbar^2\Delta_{\theta}-1)u\|_{H_\hbar^s},
$$
which implies that
$$
\|u\|_{H^s}+k^{-2}\|u\|_{H^{s+2}}\leq Ck^{-1}\|(-\Delta_\theta-k^2)u\|_{H^{s}}.
$$
In particular, $(-\Delta_\theta-k^2)^{-1}$ has no poles in $k>h_0^{-1}$ and the required estimates hold.
\end{proof}

\appendix
\section{Complex scaling for rough scaling functions}
\label{a:scale}

We follow the treatment of complex scaling in~\cite[Chapter 4]{DyZw:19}, making the necessary changes to allow for $C^{2,\alpha}$ scaling functions. 

\subsection{The scaled manifold and operator}\label{sec:A1}
For $0\leq \theta<\pi$, let $\Gamma_\theta\subset \mathbb{C}^d$ be a deformation of $\mathbb{R}^d$ satisfying the following properties
\begin{equation}
\label{e:GammaTheta}
\begin{gathered}
\Gamma_\theta\cap B_{\mathbb{C}^d}(0,R_1)=B_{\mathbb{R}^d}(0,R_1),\qquad \Gamma_\theta \cap (\mathbb{C}^d\setminus B_{\mathbb{C}^d}(0,R_2))=e^{i\theta}\mathbb{R}^d\cap (\mathbb{C}^d\setminus B_{\mathbb{C}^d}(0,R_2)), \\
\Gamma_\theta=\widetilde{\rhs}_\theta(\mathbb{R}^d),\quad \widetilde{\rhs}_\theta:\mathbb{R}^d\to \mathbb{C}^d,\text{ is injective}.
\end{gathered}
\end{equation}

Recall that for $\ell\geq 1$, a $C^{\ell,t}$ manifold $M\subset \mathbb{C}^d$ is called \emph{totally real} if for all $m\in M$,
$$
T_mM\cap iT_m M=\{0\}.
$$
(Note that we identify $T_mM$ with a subspace of $\mathbb{R}^{2d}\cong \mathbb{C}^d$ in this definition). 

Furthermore, if $u\in C^{\ell,t}(M)$, we call $\widetilde{u}\in C^{\ell,t}(\mathbb{C}^d)$ a $(\ell,t)$-\emph{almost analytic extension of $u$} if 
$$
\bar{\partial}_{z_j}\widetilde{u}(z)=O_s(d(z,M)^{\ell-1+s}),\qquad s<t
$$
where, if $z_j=x_j+iy_j$, 
$$
\partial_{z_j}:=\frac{1}{2}(\partial_{x_j} -i\partial_{y_j}),\qquad \bar{\partial}_{z_j}:=\frac{1}{2}(\partial_{x_j}+i\partial_{y_j}).
$$
Recall that a $C^1$ function, $u$, on $\Omega\subset \mathbb{C}^d$ is holomorphic if and only if $\bar{\partial}_{z_j}u=0$ for all $j=1,\dots, d$.

We next need the analog of~\cite[Lemma 4.30]{DyZw:19} for $C^{\ell,t}$ manifolds. To do this, we first need a lemma which gives $(\ell,t)$-almost analytic extensions of functions in $C_c^{\ell,t}(\mathbb{R}^d)$ functions. For this, we need to use the $C_*^{s,t}$ norm:
$$
\|u\|_{C_*^{\ell,t}}:=\sup_{k}2^{k(\ell+t)}\|\varphi^2_k(|D|)u\|_{L^\infty}
$$
where $\varphi_0\in C_c^\infty(-1,1)$, $\varphi_1\in C_c^\infty( (\frac{1}{2},2))$, $\varphi_k(x)=\varphi_1(2^{1-k}x)$,  $k\geq 1$, and $\sum_k\varphi^2_k=1$. We also recall that  for all $s, t$, 
$$
C^{\ell,t}\subset C_*^{\ell,t}
$$
and for $0<t<1$, $C^{\ell,t}=C_*^{\ell,t}.$ 
\begin{lem}
Let $\ell\in \mathbb{Z}_+$, $0<t<1$ and suppose that $u\in C_{c}^{\ell,t}(\mathbb{R}^d)$. Then, there is $\widetilde{u}\in C_{c,*}^{\ell,t}(\mathbb{C}^n)$ such that $\widetilde{u}|_{\mathbb{R}^d}=u$ and for all $s<t$,
$$
\bar{\partial}_z\widetilde{u}=O_s(|\Im z|^{\ell+s-1}).
$$
\end{lem}
\begin{proof}
Let $\chi \in C_c^\infty(B(0,2))$ with $\chi\equiv 1$ on $B(0,1)$ and $\psi\in C_c^\infty(\mathbb{R}^d)$ with $\psi \equiv 1$ on $\supp u$. Define
$$
\widetilde{u}(x+iy)= \frac{\psi(x)}{(2\pi)^d}\int e^{i\langle x-x'+iy,\xi\rangle} \chi(\langle \xi\rangle y)u(x')dx'd\xi. 
$$
Note that when $y=0$, $\widetilde{u}(x)=u(x)$ by the Fourier inversion formula and the support property of $\psi$.
Next, observe that for $0<t\leq 1$
$$
\sup_{y,y'}\frac{|\partial_\xi^\alpha \partial_y^\beta e^{-\langle y,\xi\rangle}\chi(y\langle \xi\rangle)- \partial_{\xi}^\beta\partial_y^\beta e^{-\langle y',\xi\rangle}\chi(y'\langle \xi\rangle)|}{|y-y'|^\gamma}\leq C_{\alpha\beta \gamma} \langle \xi\rangle^{|\beta|+\gamma-|\alpha|}.
$$
and 
$$
y\mapsto \partial_y^\beta e^{-\langle y,\xi\rangle}\chi(y\langle \xi\rangle)\in S^{|\beta|}
$$
is continuous.
Therefore, by~\cite[Theorem 13.8.3]{Ta:11}, $\widetilde{u}\in \bigcap_{s\leq \ell+t}C_{c}^{\ell+t-s}(\mathbb{R}_y^d;C_{c,*}^{s}(\mathbb{R}^d_x))$ and is compactly supported. In particular, $\widetilde{u}\in C_c^{\ell,t}.$

Finally, we compute 

$$
\begin{aligned}
\bar{\partial}_{z}\widetilde{u}(x+iy)&= \frac{1}{2}\frac{\partial \psi(x)}{(2\pi)^d}\int e^{i\langle x-x'+iy,\xi\rangle} \chi(\langle \xi\rangle y)u(x')dx'd\xi\\
&\qquad +\frac{i}{2}\frac{\psi(x)}{(2\pi)^d}\int e^{i\langle x-x'+iy,\xi\rangle} \langle \xi\rangle\partial\chi(\langle \xi\rangle y)u(x')dx'd\xi
&=:I+II
\end{aligned}
$$
Now, to estimate $I$, we observe that $|x-x'|>0$ on the support of the integrand, and hence we can integrate by parts in $\xi$. In particular, 
$$
I=\frac{1}{2}\frac{\partial \psi(x)}{(2\pi)^d}\int e^{i\langle x-x'+iy,\xi\rangle} \Big(\frac{\langle x-x',D_\xi\rangle}{|x-x'|^2}\Big)^N\chi(\langle \xi\rangle y)u(x')dx'd\xi=O(|y|^N).
$$
On the other hand, to estimate $II$, observe that $|\langle\xi\rangle y|>1$ on $\supp \partial \chi(\langle\xi\rangle y)$. Therefore,
$$
II=\frac{i|y|^{\ell-1+s}}{2}\frac{\psi(x)}{(2\pi)^d}\int e^{i\langle x-x'+iy,\xi\rangle}\langle \xi\rangle^{\ell+s}\frac{\partial\chi(\langle \xi\rangle y)}{(|y|\langle\xi\rangle)^{\ell-1+s}}u(x')dx'd\xi,
$$
and since
$$
\sup_y\Big|\partial_\xi^\beta e^{-\langle y,\xi\rangle}\langle \xi\rangle^{\ell+s}\frac{\partial\chi(\langle \xi\rangle y)}{(|y|\langle\xi\rangle)^{\ell-1+s}}\Big|\leq \langle \xi\rangle^{\ell+s-|\beta|},
$$
 for all $s<t$,
$$
|II|\leq C|y|^{\ell-1+s}.
$$
\end{proof}

We now give the analog of~\cite[Lemma 4.30]{DyZw:19}.
\begin{lem}
\label{l:holExtend}
Let $0<t<1$ and suppose $M\subset\mathbb{C}^d$ is a $C^{k,t}$ totally real submanifold. Then every $u\in C^{\ell,t}(M)$ has a $(\ell+t)$-almost analytic extension, $\widetilde{u}$, to $\mathbb{C}^d$. If $\widetilde{P}=\sum_{|\alpha|\leq k}a_\alpha \partial_z^\alpha $ is a holomorphic differential operator near $M$ , then $\widetilde{P}$ defines a unique differential operator $P_M$ whose action on $C^{\ell,t}(M)$ is given by 
$$
P_Mu=\big(\widetilde{P}(\widetilde{u})\big)|_M
$$
\end{lem}
\begin{proof}
The proof follows that of~\cite[Lemma 4.30]{DyZw:19} where we replace references to almost analytic by $(\ell+t)$-almost analytic.
\end{proof}

We now recall~\cite[Lemma 4.29]{DyZw:19}.
\begin{lem}
Let $\Gamma_\theta$ be as in~\eqref{e:GammaTheta}. Then $\Gamma_\theta$ is totally real if and only if 
$$
\det (\partial_x \widetilde{\rhs}_\theta)\neq 0.
$$
In particular, if $0\leq \theta<\pi/2$, and 
\begin{equation}
\label{e:basicF}
\widetilde{\rhs}_\theta(x)=x+i\partial_xF_\theta(x):\mathbb{R}^d\to \mathbb{C}^d,
\end{equation}
where $F_\theta:\mathbb{R}^d\to \mathbb{R}$ is convex, then $\Gamma_\theta$ is totally real.
\end{lem} 

Throughout the paper we work in the case~\eqref{e:basicF} as shown in the following lemma.
\begin{lem}
\label{l:tiger}
Let $\widetilde{\rhs}_\theta(x):=x+if_\theta(|x|)\frac{x}{|x|}$ with $f_\theta$ as described in~\eqref{e:fProp}. Then there is $F(x)$
satisfying
$$F''(x)\geq 0,\qquad F''(x)>0 \text{ on }|x|>R_1$$
such that $\widetilde{\rhs}_\theta(x)$ is given by~\eqref{e:basicF} with $F_\theta(x)=\tan \theta F(x)$.
\end{lem}
\begin{proof}
 We follow~\cite[Example on Page 269]{DyZw:19}. If
$$
g (r)=\int_{0}^r f(s)\,ds,
$$
then $\widetilde{\rhs}_\theta(x)= x+i\tan \theta \partial_x g(|x|)$. With $F(x)= g(|x|)$, direct calculation shows that
$$
\partial_x^2F(x)= \frac{f(|x|)}{|x|^3}(|x|^2I-x\otimes x)+\frac{f'(|x|)}{|x|^2}x\otimes x
$$
which is positive semi-definite everywhere and positive definite on $|x|>R_1$.
\end{proof}

We can now define the complex-scaled operator for a black-box Hamiltonian. Suppose that $\Gamma_\theta$ is given by~\eqref{e:GammaTheta}, with $\widetilde{\rhs}_\theta \in C^{2,t}$, for some $0<t<1$, and $\widetilde{\rhs}_\theta$ satisfying~\eqref{e:basicF}, and that $P$ is a black-box Hamiltonian as in \S\ref{s:blackBox}. With $\chi \in C_c^\infty(B(0,R_1))$ equal to 1 on $B(0,R_0)$, define 
\begin{equation}
\label{e:defP}
\begin{gathered}
\mc{H}_\theta=\mc{H}_{R_0}\oplus L^2(\Gamma_\theta \setminus B(0,R_0)),\\
\mc{D}_\theta=\{u\in \mc{H}_\theta \,:\, \chi u\in \mc{D},\, (1-\chi)u\in H^2(\Gamma_\theta)\},\\
P_\theta u= P(\chi u)+ (-\Delta_\theta)((1-\chi )u),
\end{gathered}
\end{equation}
with 
$
\Delta_\theta:=\Delta_{\Gamma_\theta}
$
defined as in Lemma~\ref{l:holExtend}.
\subsection{Fredholm properties of the scaled operator}\label{sec:A2}

Throughout this section we use the following standard characterization of Fredholm operators.
\begin{lem}
\label{l:fredEst}
Let $X$ and $Y$, $Z_X$ and $Z_{Y^*}$ be Banach spaces such that $X\subset Z_X$ is compact and $Y^*\subset Z_{Y^*}$ is compact. Suppose that there is $C>0$ such that  $P:X\to Y$ satisfies
$$
\|u\|_{X}\leq C(\|Pu\|_{Y}+\|u\|_{Z_X})\qquad \text{and}\qquad 
\|u\|_{Y^*}\leq C(\|P^*u\|_{X^*}+\|u\|_{Z_{Y^*}}).
$$
Then $P:X\to Y$ is Fredholm.
\end{lem}

 It is easy to check that $\Delta_\theta$ is an elliptic second order differential operator given by 
\beq\label{eq:Delta_theta}
\Delta_\theta u=((I+iF_\theta''(x))^{-1}\partial_x)\cdot ((I+iF_\theta''(x))^{-1}\partial_x u),\qquad u\in C^{\ell,t}(\Gamma_\theta);
\eeq
see \cite[Equation 4.5.13 and Theorem 4.32]{DyZw:19}.
\begin{lem}
\label{l:intByParts}
For $u\in H^1(\mathbb{R}^d)$, and all $\e>0$
\begin{equation}
\label{e:basicFred0}
\begin{gathered}
\Im \langle -\Delta_\theta u, u\rangle \leq \e\|u\|^2_{H^1}+C\e^{-1}\|u\|^2_{L^2},\qquad \|u\|_{H^1}^2\leq C\big|\langle -\Delta_\theta u,u\rangle\big|+C\|u\|_{L^2}^2.
\end{gathered}
\end{equation}
Furthermore, 
\begin{equation}
\label{e:basicFred}
\|u\|^2_{H^1}\leq C\big(\big|\Re\langle -\Delta_\theta u,u\rangle\big|-\Im \langle -\Delta_\theta u,u\rangle+\|u\|_{L^2}^2\big).
\end{equation}
\end{lem}
\begin{proof}
By the definition of the operator $-\Delta_\theta$ \eqref{eq:Delta_theta} acting on $H^1$,
\begin{align}\label{eq:A41}
\langle -\Delta_\theta u,u\rangle_{_{H^{-1},H^{1}}}&= \langle (I+iF_\theta''(x))^{-1}\partial_x u,(I-iF_\theta''(x))^{-1}\partial_x u\rangle +\langle A(x)\partial_x u,u\rangle 
\end{align}
where $A(x)\in C^{0,\alpha}$. First, note that
$$
\big|\langle A(x)\partial_x u,u\rangle\big|\leq C\|u\|_{H^1}\|u\|_{L^2}.
$$
Next, put $v= (I +F_\theta''(x)^2)^{-1}\partial_x u$.  (Note that the inverse exists and is bounded since $F_\theta''(x)$ is real, symmetric, and tends to $\tan\theta \,I$.) Then, 
\begin{equation}
\label{e:absVal}
\begin{aligned}
\langle (I+iF_\theta''(x))^{-1}\partial_x u,(I-iF_\theta''(x))^{-1}\partial_x u\rangle&=\langle (I-iF_\theta''(x))v,(I+iF_\theta''(x))v\rangle\\
&=\langle (I-iF_\theta''(x))^2v,v\rangle\\
&=\|v\|_{L^2}^2-\|F_\theta''(x) v\|_{L^2}^2 -2i\langle F_\theta''(x)v,v\rangle.
\end{aligned}
\end{equation}
Therefore, since $F_\theta''$ is positive semi-definite, the first inequality in~\eqref{e:basicFred0} holds.

To obtain the second inequality in~\eqref{e:basicFred0}, observe that if 
$$
-\Im \big\langle (I+iF_\theta''(x))^{-1}\partial_x u,(I-iF_\theta''(x))^{-1}\partial_x u\big\rangle=2\langle F_\theta''(x)v,v\rangle \geq 2\e\|v\|^2,
$$
then~\eqref{e:basicFred0} holds. On the other hand, since $F_\theta''$ is positive semi-definite,
\begin{equation}
\label{e:posSemi}
\langle F_\theta''(x)v,v\rangle \leq \e\|v\|_{L^2}^2 \qquad \text{ implies that } \qquad  \|F_\theta''(x)v\|^2_{L^2}\leq C\e^{\frac{2}{3}}\|v\|^2_{L^2}.
\end{equation}
Indeed, the multiplication operator $F_\theta'':L^2(\mathbb{R}^d;\mathbb{C}^d)\to L^2(\mathbb{R}^d;\mathbb{C}^d)$ is positive semidefinite and self adjoint. Therefore, letting $\Pi_\e$ be the spectral projector onto the spectrum $\leq \e^{\frac{1}{3}}$, we have 
\begin{align*}
\e \|v\|^2\geq \langle F_\theta'' v,v\rangle& = \langle F_\theta'' \Pi_\e v,\Pi_\e v\rangle +\langle F_\theta'' (I-\Pi_\e)v,(I-\Pi_\e )v\rangle \\
&\geq \e^{\frac{1}{3}}\|(I-\Pi_\e)v\|_{L^2}^2.
\end{align*}
Therefore, 
$$
\|F_\theta''v\|_{L^2}^2=\|F_\theta'' \Pi_\e v\|_{L^2}^2+\|F_\theta''(I-\Pi_\e)v\|_{L^2}^2\leq C\e^{\frac{2}{3}}\|v\|_{L^2}^2.
$$
Thus, using~\eqref{e:absVal} together with~\eqref{e:posSemi} with $\e>0$ small enough, we have
$$
\big|\big\langle (I+iF_\theta''(x))^{-1}\partial_x u,(I-iF_\theta''(x))^{-1}\partial_x u\big\rangle\big|\geq c\|v\|_{L^2}^2\geq c\|\partial_x u\|_{L^2}^2,
$$
and~\eqref{e:basicFred0} follows.

To obtain~\eqref{e:basicFred}, we use the second equation in \eqref{e:basicFred0} to obtain that
\begin{align*}
\|u\|_{H^1}^2&\leq C|\Re\langle -\Delta_\theta u,u\rangle|+C|\Im \langle -\Delta_\theta u,u\rangle|+C\|u\|_{L^2}^2\\
&\leq  C|\Re\langle -\Delta_\theta u,u\rangle|+C\big|\Im \langle -\Delta_\theta u,u\rangle-\e\|u\|_{H^1}^2-C\e^{-1}\|u\|_{L^2}^2\big|+C(1+\e^{-1})\|u\|_{L^2}^2+C\e \|u\|_{H^1}^2\\
&=C|\Re\langle -\Delta_\theta u,u\rangle|- C\Im \langle -\Delta_\theta u,u\rangle+C\e\|u\|_{H^1}^2+C^2\e^{-1}\|u\|_{L^2}^2+C(1+\e^{-1})\|u\|_{L^2}^2+C\e \|u\|_{H^1}^2,
\end{align*}
and $\e>0$ small enough.
\end{proof}

\begin{lem}
\label{l:Fred1}
The operator 
$$
-\Delta_\theta-\lambda^2:H^1(\mathbb{R}^d)\to H^{-1}(\mathbb{R}^d)
$$
is an analytic family of Fredholm operators with index zero in $\Im (e^{i\theta} \lambda)>0$. Furthermore,
$$
R_{0,\theta}(\lambda):=(-\Delta_\theta-\lambda^2)^{-1}:H^{-1}(\mathbb{R}^d)\to H^1(\mathbb{R}^d)
$$
is a meromorphic family of operators with finite rank poles and there is $t_0>0$ such that for $t>t_0$,
$$
\|R_{0,\theta}(e^{i\frac{\pi}{4}}t)\|_{H^{-1}\to L^2}\leq \frac{C}{t}.
$$
\end{lem}
\begin{proof}
First note that 
$$
-e^{-2i\theta}\Delta-\lambda^2= e^{-2i\theta}(-\Delta -(\lambda e^{i\theta})^2):H^{s}(\mathbb{R}^d)\to H^{s-2}(\mathbb{R}^d)
$$
is invertible for $\Im (\lambda e^{i\theta})\neq 0$ since $-\Delta:L^2\to L^2$ is self adjoint. 

Suppose that 
$$
(-\Delta_\theta-\lambda^2)u=f. 
$$
Let $\chi \in C_c^\infty(\mathbb{R}^d)$ with $\chi\equiv 1$ on $B(0,R_2)$. 
Then,
$$
(1-\chi)f=(1-\chi) (e^{-2i\theta}\Delta-\lambda^2)u=(e^{-2i\theta}\Delta-\lambda^2)(1-\chi)u +e^{-2i\theta}[-\Delta,\chi]u.
$$
Therefore, 
\begin{equation}
\label{e:outside}
\|(1-\chi)u\|_{H^1}\leq C\big(\|(1-\chi)\rhs\|_{H^{-1}}+\|u\|_{L^2(\supp \partial \chi)}\big).
\end{equation}
On the other hand, by Lemma~\ref{l:intByParts} for $\psi \in C_c^\infty(\mathbb{R}^d)$,  and $\psi_1\in C_c^\infty(\mathbb{R}^d)$ with $\psi_1\equiv 1$ on $\supp\psi$. 
\begin{equation}
\label{e:inside}
\| \psi u\|_{H^1}\leq C\big(\|\psi_1 \rhs\|_{H^{-1}}+\|\psi_1 u\|_{L^2}\big).
\end{equation}
In particular, combining~\eqref{e:outside} and~\eqref{e:inside}, there is $\psi_1\in C_c^\infty$ such that 
\beq\label{eq:kernelcontrol}
\|u\|_{H^1}\leq C\big(\|(-\Delta_\theta-\lambda^2)u\|_{H^{-1}}+\|\psi _1u\|_{L^2}\big).
\eeq

Now, since 
$$
(-e^{2i\theta}\Delta-\bar{\lambda}^2):H^1(\mathbb{R}^d)\to H^{-1}(\mathbb{R}^d)
$$
is invertible, an identical argument shows that 
$$
\|u\|_{H^1}\leq C\big(\|(-\Delta_\theta-\lambda^2)^*u\|_{H^{-1}}+\|\psi u\|_{L^2}\big).
$$
Lemma~\ref{l:fredEst} now shows that $(-\Delta_\theta-\lambda^2):H^1\to H^{-1}$ is Fredholm for $\Im (e^{i\theta}\lambda)>0$.

Finally, we check the index of this operator. For $u\in H^1$, $\lambda=e^{\frac{i\pi}{4}}t$, and $c=1/\sqrt{2}$,
by~\eqref{e:basicFred},
\begin{equation}
\label{e:L2Bound}
\begin{aligned}
\big|\langle (-\Delta_\theta-\lambda^2)u,u\rangle\big|&\geq c\big| \Re\langle -\Delta_\theta u,u\rangle\big| +c\big|(\Im \langle -\Delta_\theta u,u\rangle-t^2\|u\|^2\big|\\
&\geq  c\big| \Re\langle -\Delta_\theta u,u\rangle\big| -c\Im \langle -\Delta_\theta u,u\rangle+ ct^2\|u\|^2  \\
&\geq c\|u\|_{H^1}^2+ (ct^2-C)\|u\|_{L^2}^2.
\end{aligned}
\end{equation}
Thus,
$$
 c\|u\|_{H^1}^2+ (ct^2-C)\|u\|_{L^2}^2\leq \tfrac{1}{2\e}\|(-\Delta-\lambda^2)u\|^2_{H^{-1}}+\tfrac{\e}{2}\|u\|^2_{H^{1}}
$$
and hence, choosing $\e>0$ small enough,
$$
\sqrt{(ct^2-C)}\|u\|_{L^2}+c\|u\|_{H^1}\leq \|(-\Delta_\theta-\lambda^2)u\|_{H^{-1}}. 
$$
Similarly,
$$
\sqrt{(ct^2-C)}\|u\|_{L^2}+c\|u\|_{H^1}\leq \|(-\Delta_\theta-\lambda^2)^*u\|_{H^{-1}},
$$
and hence, for $t$ sufficiently large, $(-\Delta_\theta - (e^{\frac{i\pi}{4}}t)^2):H^1\to H^{-1}$ is invertible. 
\end{proof}

\begin{lem}
\label{l:freeScaledEstimate}
For $\Im (e^{i\theta}\lambda)>0$, $R_{0,\theta}(\lambda):L^2(\mathbb{R}^d)\to H^2(\mathbb{R}^d)$ and there are $C>0$ and $t_0>0$ such that for $t>t_0$, and $\ell=0,1,2$,
$$
\|R_{0,\theta}(e^{i\frac{\pi}{4}}t)\|_{L^2\to H^\ell }\leq C t^{\ell-2}.
$$
\end{lem}
\begin{proof}
Suppose $f\in L^2$. Then, $R_{0,\theta}(\lambda)f\in H^1(\mathbb{R}^d)$ and $(-\Delta_\theta-\lambda^2)R_{0,\theta}(\lambda)f=f$, and using~\eqref{e:L2Bound}, we obtain
$$
\|R_{0,\theta}(\lambda)\|_{L^2\to L^2}\leq Ct^{-2}.
$$
By $H^2$ elliptic regularity (see, e.g.,~\cite[Section 6.3, Theorem 1]{Eva}), for $u\in H^1$,
\begin{equation}
\label{e:elliptic}
\|u\|_{H^2}\leq C(\|(-\Delta_\theta-\lambda^2)u\|_{L^2}+\| u\|_{L^2}).
\end{equation}
Therefore $R_{0,\theta}(\lambda):L^2\to H^2$ and
$$
\|R_{0,\theta}(\lambda)\|_{L^2\to H^2}\leq C;
$$
the bound $L^2\to H^1$ follows by interpolation.
\end{proof}

\begin{lem}
\label{l:merInverse}
Suppose that $P(\lambda):X\to Y$ is an analytic family of operators in $\Omega\subset \mathbb{C}$ and there are $Q(\lambda):Y\to X$ and $S(\lambda):Y\to X$ meromorphic families of operators with finite rank poles such that 
$$
P(\lambda)Q(\lambda)=I+K_1(\lambda),\qquad S(\lambda)P(\lambda)=I+K_2(\lambda)
$$
with $K_1:Y\to Y$ compact and $K_2:X\to X$ compact. Then, $P(\lambda)$ is Fredholm.
\end{lem}
\begin{proof}
Let $\lambda_0\in \Omega$. 
By the definition of a meromorphic family of operators (see, e.g., \cite[Definition C.7]{DyZw:19}),
there are $J\geq 0$, $A_0(\lambda):Y\to X$ and $A_j:Y\to X$ such that $A_0(\lambda)$ is holomorphic near $\lambda_0$, $A_j$ is finite rank, $j=1,\dots, J$, and 
$$
Q(\lambda)=A_{0}(\lambda)+\sum_j \frac{A_j}{(\lambda-\lambda_0)^j}.
$$
Then, we claim that 
\beq\label{eq:Taylor}
\sum_{j=1}^J \frac{P(\lambda)A_j}{(\lambda-\lambda_0)^j}=P(\lambda)(Q(\lambda)-A_{0}(\lambda))=I+K_1(\lambda)-P(\lambda)A_0(\lambda).
\eeq
is an analytic family of compact operators. Indeed, the left hand side of this equality is a meromorphic family of operators with uniformly bounded rank. On the other hand, the right hand side $I+K_1(\lambda)-P(\lambda)A_0(\lambda)$ is analytic. In fact, 
by Taylor-expanding $P(\lambda)$ about $\lambda=\lambda_0$ and demanding that the coefficients of $(\lambda-\lambda_0)^{k-J}$ on the left-hand side of \eqref{eq:Taylor} equal zero for $k=0,\ldots,J-1$, we see that, for $0\leq k\leq J-1$,
$$
\sum_{n=0}^k \frac{\partial_{\lambda}^{n}P|_{\lambda=\lambda_0}}{n!}A_{J-k+n}=0.
$$
 Thus,
$$
\sum_{j=1}^J \frac{P(\lambda)A_j}{(\lambda-\lambda_0)^j}= \sum_{j=1}^J\big[ \partial_\lambda^jP|_{\lambda=\lambda_0}A_j +O(|\lambda-\lambda_0|)_{X\to Y}\big]A_j,
$$
and this operator is an analytic family of compact operators as claimed. 
We then observe that 
$$
P(\lambda)A_0(\lambda)=I+K_1(\lambda)-\sum_j \frac{P(\lambda)A_j}{(\lambda-\lambda_0)^j}=I+\tilde{K}_1(\lambda)
$$
with $\tilde{K}_1(\lambda)$ an analytic family of compact operators.

Writing 
$$
S(\lambda)=B_0(\lambda)+\sum_{j=1}^J\frac{B_j}{(\lambda-\lambda_0)^j}
$$ 
with $B_0(\lambda):Y\to X$ analytic and $B_j:Y\to X$ finite rank, $j=1,\dots, J$, and applying the same argument shows that $B_0(\lambda)$ is an approximate left inverse for $P(\lambda)$.
Since $P(\lambda)$ has both an approximate left and right inverse, it is Fredholm (see, e.g., \cite[(C.2.8)]{DyZw:19}).
\end{proof}

\begin{prop}
\label{p:scaledFredholm}
Let $P_\theta$, $\mc{D}_\theta$, and $\mc{H}_\theta$, $0\leq \theta <\pi/2$, be as in~\eqref{e:defP}. If $\Im (e^{i\theta}\lambda)>0$, then 
$$
P_\theta -\lambda^2:\mc{D}_\theta\to \mc{H}_\theta 
$$
is a Fredholm operator of index zero and there is $t_0>0$ such that for $t>t_0$, and $0\leq s \leq 1$,
\beq\label{eq:RAM2}
\|(P_\theta-it^2)^{-1}\|_{\mc{H}_\theta\to \mc{D}_\theta^s}\leq Ct^{2s-2}.
\eeq
Moreover,  let $R_0<R_1$ with $R_1$ as in~\eqref{e:GammaTheta}. Then 
\beq\label{eq:RAM1}
\indicator_{B(0,R_1)}(P-\lambda^2)^{-1}\indicator_{B(0,R_1)} =\indicator_{B(0,R_1)}(P_\theta-\lambda^2)^{-1}\indicator_{B(0,R_1)},\qquad \Im (e^{i\theta}\lambda)>0.
\eeq
\end{prop}

\bpf
Together with Lemma~\ref{l:merInverse}, the proofs of~\cite[Theorems 4.36, 4.37]{DyZw:19} prove the result with \eqref{eq:RAM1} replaced by 
\begin{equation}
\label{e:chiRAM1}
\chi(P-\lambda^2)^{-1}\chi =\chi(P_\theta-\lambda^2)^{-1}\chi,\qquad \Im (e^{i\theta}\lambda)>0,
\end{equation}
for $\chi \in C_c^\infty(B(0,R_1))$ with $\chi\equiv 1$ on $B(0,R_0)$. 
(Although the bound \eqref{eq:RAM2} is not explicitly stated in \cite[Theorems 4.36, 4.37]{DyZw:19}, it is essentially contained in Step 3 of the proof of \cite[Theorem 4.36]{DyZw:19}.)

Replacing $\chi$ on the left of both sides of~\eqref{e:chiRAM1} by the indicator functions in~\eqref{eq:RAM1} follows by the unique continuation principle since $P=P_\theta$ on $B(0,R_1)$. To replace $\chi$ on the right of both sides of~\eqref{e:chiRAM1}, we approximate $f\in \mc{H}_{R_0}\oplus L^2(B(0,R_1)\setminus B(0,R_0))$ by $f_n\in \mc{H}_{R_0}\oplus L^2(B(0,R_1-n^{-1})\setminus B(0,R_0))$ and use continuity of $(P_\theta-\lambda^2)^{-1}:\mc{H}_\theta\to \mc{H}_\theta$ and $R_P(\lambda):\mc{H}_{\comp}\to \mc{H}_{\loc}$.
\epf

\subsection{Fredholm properties for the PML operator}

Now that we have obtained the Fredholm property of $P_\theta$, we study the Fredholm properties of the corresponding PML operator. Let $\Omega_\theta\Subset \Gamma_\theta$ have Lipschitz boundary and $B(0,R_1)\subset \Omega_\theta$. We study the PML operator $P_\theta -\lambda^2$ on $\Omega_\theta$. Let
\begin{equation}
\label{e:defPML}
\begin{gathered}
\mc{H}_{\theta}(\Omega_\theta):=\mc{H}_{R_0}\oplus L^2(\Omega_\theta \setminus B(0,R_0)),\\
\mc{D}_\theta(\Omega_\theta):=\big\{u\in \mc{H}_\theta(\Omega_\theta) \,:\, \chi u\in \mc{D},\, (1-\chi)u\in H_0^1(\Omega_\theta),\, -\Delta_\theta ((1-\chi)u)\in L^2(\Omega_\theta)\big\},\\
P_\theta u:= P(\chi u)+ (-\Delta_\theta)((1-\chi )u),
\end{gathered}
\end{equation}

We start by showing the Fredholm property when there is no black-box Hamiltonian i.e. when $P_\theta=-\Delta_\theta$. 
\begin{lem}
\label{l:freePMLFredholm}
The operator 
$$
(-\Delta_\theta-\lambda^2): H_0^1(\Omega_\theta)\to  H^{-1}(\Omega_\theta)
$$
is Fredholm with index zero.  Let $R_{0,\theta}^D(\lambda):=(-\Delta_\theta-\lambda^2)^{-1}:H^{-1}(\Omega_\theta)\to  H_0^1(\Omega_\theta)$.
Then 
there is $t_0>0$ such that for $t>t_0$, and $0\leq s\leq 1$,
\beq\label{eq:RAM3}
\|R_{0,\theta}^D(e^{i\frac{\pi}{4}}t)\|_{L^2(\Omega_\theta)\to H^s(\Omega_\theta)}\leq Ct^{s-2}.
\eeq
\end{lem}
\begin{proof}
Repeating the arguments in the proof of Lemma~\ref{l:intByParts} for $u\in H_0^{1}(\Omega_\theta)$ instead of $u\in H^1(\Rea^d)$, we obtain that, for $u\in H_0^{1}(\Omega_\theta)$,
\begin{equation}
\label{e:basicFred0Omega}
\begin{gathered}
\Im \langle -\Delta_\theta u, u\rangle_{\Omega_\theta} \leq \e\|u\|^2_{H^1(\Omega_\theta)}+C\e^{-1}\|u\|^2_{L^2(\Omega_\theta)},\qquad \|u\|_{H^1(\Omega_\theta)}^2\leq C|\langle -\Delta_\theta u,u\rangle_{\Omega_\theta}|+C\|u\|_{L^2(\Omega_\theta)}^2,
\end{gathered}
\end{equation}
and
\begin{equation}
\label{e:basicFredOmega}
\|u\|^2_{H^1(\Omega_\theta)}\leq C(|\Re\langle -\Delta_\theta u,u\rangle_{\Omega_\theta}|-\Im \langle -\Delta_\theta u,u\rangle_{\Omega_\theta}+\|u\|_{L^2(\Omega_\theta)}^2).
\end{equation}
The second estimate in~\eqref{e:basicFred0Omega} together with the fact that $-\Delta_\theta:H^1_0(\Omega_\theta)\to H^{-1}(\Omega_\theta)$ is bounded, implies the Fredholm property for $(-\Delta_\theta-\lambda^2):H^1_0(\Omega_\theta)\to H^{-1}(\Omega_\theta)$
(similar to in the proof of Lemma \ref{l:Fred1}). To check that the index of the operator is 0 we argue as in~\eqref{e:L2Bound}.
The estimate \eqref{e:basicFredOmega} implies that, for $\lambda =e^{\frac{i\pi}{4}}t$, the estimate \eqref{e:L2Bound} holds. 
The bound \eqref{eq:RAM3} for $s=0,1,$ then follows from  \eqref{e:L2Bound}  (exactly as in Lemma~\ref{l:freeScaledEstimate}), and the bound 
\eqref{eq:RAM3} for $0<s<1$ then follows via interpolation.
\end{proof}

Finally, we show that the black-box PML operator~\eqref{e:defPML} is Fredholm with index zero.
\begin{prop}
\label{p:fredPML}
Let $P_\theta$, $\mc{H}_\theta(\Omega_\theta)$, and $\mc{D}_\theta(\Omega_\theta)$ be as in~\eqref{e:defPML}. Then, $P_\theta-\lambda^2:\mc{D}_\theta(\Omega_\theta)\to \mc{H}_{\theta}(\Omega_\theta)$ is Fredholm with index zero. 
\end{prop}
\begin{proof}
To show that $P_\theta-\lambda^2$ is Fredholm, we find meromorphic families of operators giving both an approximate left and right inverse for $P_\theta-\lambda^2$. Lemma~\ref{l:merInverse} then shows that $P_\theta-\lambda^2$ is Fredholm. To show $P_\theta-\lambda^2$ has index zero we find $\lambda_0$ where $P_\theta-\lambda_0^2$ is invertible (since the index is constant in $\lambda$ by, e.g., \cite[Theorem C.5]{DyZw:19}). 

\smallskip
\noindent
\emph{Approximate right inverse.}

Let $\chi_0\in C_c^\infty(\mathbb{R}^d;[0,1])$ with $\chi_0 \equiv 1$ on $B(0,R_0+\e)$ for some $\e>0$. Then choose $\chi_j\in C_c^\infty(\mathbb{R}^d;[0,1])$, $j=1,2$ such that 
\beq\label{eq:supports}
\chi_j\equiv 1\,\,\text{on}\,\,\supp \chi_{j-1,},\qquad \supp \chi_j\subset B(0,R_1). 
\eeq
Let
$$
Q_0:=(1-\chi_0)R_{0,\theta}^D(\lambda)(1-\chi_1),\qquad Q_1:=\chi_2(P_\theta-\lambda^2)^{-1}\chi_1
$$
where $(P_\theta-\lambda^2)^{-1}:\mc{H}_\theta \to \mc{D}_\theta$.
Then
\begin{gather*}
(P_\theta-\lambda^2)Q_0=(1-\chi_1)+[\Delta_\theta,\chi_0]R_{0,\theta}^D(\lambda)(1-\chi_1),\\ 
(P_\theta-\lambda^2)Q_1=\chi_1+[-\Delta_\theta,\chi_2](P_\theta-\lambda^2)^{-1}\chi_1,
\end{gather*}
and thus
\begin{gather*}
(P_\theta-\lambda^2)(Q_0+Q_1)=\Id +K(\lambda),\quad \text{ where } \quad K(\lambda):=K_0(\lambda)+K_1(\lambda),\\ K_0(\lambda):=[\Delta_\theta,\chi_0]R_{0,\theta}^D(\lambda)(1-\chi_1),\quad K_1(\lambda):=[-\Delta_\theta,\chi_2](P_\theta-\lambda^2)^{-1}\chi_1.
\end{gather*}
By Lemma~\ref{l:freePMLFredholm}, $R_{0,\theta}^D:L^2(\Omega_\theta)\to \mc{D}_\theta(\Omega_\theta)$.
Since $\Omega_\theta$ is Lipschitz, $\mc{D}_\theta(\Omega_\theta)\subset H^{3/2}(\Omega_\theta)$ by \cite[Lemme 2]{CoDa:98}, \cite[Corollary 5.7]{JeKe:95}.
Therefore, since $(1-\chi_1):\mc{H}_{\theta}(\Omega_\theta)\to L^2(\Omega_\theta)$ and $[-\Delta_\theta,\chi_0]:H^{3/2}(\Omega_\theta)\to H^{1/2}(B(0,R_1)\setminus B(0,R_0+\e))$, 
$K_0(\lambda):\mc{H}_{\theta}(\Omega_\theta)\to H^{1/2}(B(0,R_1)\setminus B(0,R_0+\e))$. Thus, $K_0(\lambda):\mc{H}_\theta(\Omega_\theta)\to \mc{H}_\theta(\Omega_\theta)$ is compact. 

Next, by Proposition~\ref{p:scaledFredholm}, $(P_\theta-\lambda^2)^{-1}:\mc{H}_\theta\to \mc{D}_\theta$. Therefore, since $\chi_1:\mc{H}_\theta(\Omega_\theta)\to \mc{H}_\theta$, and $[-\Delta_\theta,\chi_2]:\mc{D}_\theta\to H^1_{\comp}(B(0,R_1)\setminus B(0,R_0+\e)$, 
$$
K_1(\lambda):\mc{H}_\theta(\Omega_\theta)\to H^1_{\comp}(B(0,R_1)\setminus B(0,R_0+\e).
$$
In particular, $K_1(\lambda):\mc{H}_\theta(\Omega_\theta)\to \mc{H}_\theta(\Omega_\theta)$ is compact and thus $K(\lambda):\mc{H}_\theta(\Omega_\theta)\to \mc{H}_{\theta}(\Omega_\theta)$ is compact.

\smallskip
\noindent
\emph{Invertibility of the right inverse.}

We now show that for $\lambda=e^{\frac{i\pi}{4}}t$ and $t$ sufficiently large, $\Id+K(\lambda)$ is invertible. 
By  Lemma \ref{l:freePMLFredholm} and Proposition \ref{p:scaledFredholm}, respectively, for $t>t_0$,
\beq\label{eq:boundsappendix1}
 \|R_{0,\theta}^D(e^{\frac{\pi i}{4}}t)\|_{L^2(\Omega_\theta)\to H^1(\Omega_\theta)}\leq Ct^{-1}
\quad\tand\quad
\|(P_\theta-it^2)^{-1}\|_{\mc{H}_\theta\to \mc{D}_\theta^{\frac{1}{2}}}\leq Ct^{-1}.
\eeq
Furthermore, $[\Delta_\theta,\chi_0]:H^1(\Omega_\theta)\to \mc{H}_\theta(\Omega_\theta)$ and 
$[\Delta_\theta,\chi_2]:\mc{D}^{\frac{1}{2}}_\theta\to \mc{H}_\theta(\Omega_\theta)$.
Using these bounds and mapping properties in the definition of $K(\lambda)$, we find that, for $t>t_0$,
$$\|K(e^{\frac{i\pi}{4}}t)\|_{\mc{H}_\theta(\Omega_\theta)\to \mc{H}_\theta(\Omega_\theta)}\leq Ct^{-1};$$
hence $\Id+K(e^{\frac{i\pi}{4}}t)$ is invertible for $t$ sufficiently large.

\noindent
\emph{Approximate left inverse.}

For the left inverse, let 
$$
S_\theta(\lambda):=(1-\chi_1)R_{0,\theta}^D(\lambda)(1-\chi_0)+\chi_1(P_\theta-\lambda^2)^{-1}\chi_2,
$$
and observe that
\beqs
S_\theta(\lambda)(P_\theta-\lambda^2)=I+ L_\theta(\lambda),
\eeqs
where
\beqs
L_\theta(\lambda):= (1-\chi_1)R_{0,\theta}^D(\lambda)[\chi_0,\Delta_\theta]+\chi_1(P_\theta-\lambda^2)^{-1}[\chi_2,-\Delta_\theta].
\eeqs
Note that $S_\theta:\mc{H}_\theta(\Omega_\theta)\to \mc{D}_\theta(\Omega_\theta)$ and hence,  $L_\theta:\mc{D}_\theta(\Omega_\theta)\to \mc{D}_\theta(\Omega_\theta)$.

The fact that $L_\theta:\mc{H}_\theta(\Omega_\theta)\to \mc{H}_\theta(\Omega_\theta)$ is compact follows from the mapping properties
\beq\label{eq:mapping}
R_{0,\theta}^D(\lambda):H^{-1}(\Omega_\theta)\to H^1_0(\Omega_\theta),\qquad (P_\theta-\lambda^2)^{-1}:\mc{D}_\theta^{-1/2}\to \mc{D}_\theta^{1/2},
\eeq
with the former coming from Lemma \ref{l:freePMLFredholm}, and the latter coming from Proposition \ref{p:scaledFredholm} plus duality and interpolation.
Therefore, by the definition of $\|\cdot\|_{\mc{D}_\theta(\Omega_\theta)}$ (inherited from \eqref{eq:normD}), to show that $L_\theta:\mc{D}_\theta(\Omega_\theta)\to \mc{D}_\theta(\Omega_\theta)$ is compact, it is enough to show that 
$(P_\theta-\lambda^2)L_\theta(\lambda): \mc{D}_\theta(\Omega_\theta)\to \mc{H}_\theta(\Omega_\theta)$ is compact. Now, 
using \eqref{eq:supports}, we obtain that
\beq\label{eq:PL}
(P_\theta-\lambda^2)L_\theta(\lambda)=[\Delta_\theta,\chi_1]R_{0,\theta}^D(\lambda)[\chi_0,\Delta_\theta]+[-\Delta_\theta,\chi_1](P_\theta-\lambda^2)^{-1}[\chi_2,-\Delta_\theta].
\eeq
The compactness of $(P_\theta-\lambda^2)L_\theta(\lambda): \mc{D}_\theta(\Omega_\theta)\to \mc{H}_\theta(\Omega_\theta)$ then follows since $[-\Delta_\theta, \chi_i]:\mc{D}_\theta\to H^1(B(0,R_1)\setminus B(0,R_0+\e))$, 
$$
[-\Delta,\chi_i](P_\theta-\lambda^2)^{-1},\,[-\Delta,\chi_i]R_{0,\theta}^D(\lambda): H^1(B(0,R_1)\setminus B(0,R_0+\e))\to  H^1(B(0,R_1)\setminus B(0,R_0+\e)),
$$
and  $I:H^1(B(0,R_1)\setminus B(0,R_0+\e))\to \mc{H}_\theta(\Omega_\theta)$ is compact.

\noindent
\emph{Invertibility of the left inverse.}

Finally, we show that for $\lambda = e^{\frac{i\pi}{4}}t$ and $t$ sufficiently large, $I+L_\theta(\lambda):\mc{D}_\theta(\Omega_\theta)\to \mc{D}_\theta(\Omega_\theta)$ is invertible.
As a map $\mc{H}_\theta(\Omega_\theta)\to \mc{H}_\theta(\Omega_\theta)$, $(I+L_\theta(\lambda))^{-1}$ exists by the same argument used to show that $I+K(\lambda)$ was invertible (and the  corresponding
estimates on $R_{0,\theta}^D(\lambda):H^{-1}(\Omega_\theta)\to L^2(\Omega_\theta)$, and $(P_\theta-\lambda^2):\mc{D}_\theta^{-1/2}\to  \mc{H}_\theta$ obtained from \eqref{eq:boundsappendix1} by duality).

Therefore, by the definition of $\|\cdot\|_{\mc{D}_\theta(\Omega_\theta)}$, to show that $(I+L_\theta(\lambda))^{-1}:\mc{D}_{\theta}(\Omega_\theta)\to \mc{D}_\theta(\Omega_\theta)$, it is sufficient to show that $(P_\theta-\lambda^2)(I+L_\theta(\lambda))^{-1}: \mc{D}_{\theta}(\Omega_\theta) \to \mc{H}_{\theta}(\Omega_\theta)$.
Since
$$
(P_\theta-\lambda^2)(I+L_\theta(\lambda))^{-1}= P_\theta-\lambda^2-(P_\theta-\lambda^2)L_\theta(\lambda)(I+L_\theta(\lambda))^{-1},
$$
it is enough to prove that $(P_\theta-\lambda^2)L_\theta(\lambda):\mc{H}_\theta(\Omega_\theta)\to \mc{H}_\theta(\Omega_\theta)$, and this follows from \eqref{eq:PL} and the mapping properties \eqref{eq:mapping}.
\end{proof}

\section{Semiclassical analysis}\label{app:SCA}

\subsection{Semiclassical pseudo-differential operators}

We review here the notation and definitions for semiclassical pseudodifferential operators on $\mathbb{R}^d$ used in this paper.

\paragraph{\textbf{Semiclassical Sobolev spaces}}
We say that $u\in H_\hbar^s(\mathbb{R}^d)$ if 
$$
 \|\langle \xi\rangle ^s \mathcal{F}_\hbar(u)(\xi)\|_{L^2}<\infty, \quad\text{ where } \quad\langle \xi\rangle:=(1+|\xi|^2)^{\frac{1}{2}}
\quad\tand\quad
\mathcal{F}_\hbar(u)(\xi):=\int e^{-\frac{i}{\hbar}\langle y,\xi\rangle}u(y)\,dy
$$
is the \emph{semiclassical Fourier transform}.

\paragraph{\textbf{Symbols and operators}} We say that $a\in C^\infty(T^*\mathbb{R}^d)$ is a symbol of order $m$ if 
$$
|\partial_x^\alpha \partial_\xi^\beta a(x,\xi)|\leq C_{\alpha\beta}\langle \xi\rangle^m,
$$
and write $a\in S^m(T^*\mathbb{R}^d)$. 
Throughout this section we fix $\chi_0\in C_c^\infty(\mathbb{R}))$ to be identically 1 near 0.  
We then say that an operator $A:C_c^\infty(\mathbb{R}^d)\to \mathcal{D}'(\mathbb{R}^d)$ 
is a \emph{semiclassical pseudodifferential operator} of order $m$, and write $A\in \Psi_\hbar^m(\mathbb{R}^d)$, if $A$ can be written as
\begin{equation}
\label{e:basicPseudo}
Au(x)=\frac{1}{(2\pi \hbar)^d}\int_{\Rea^d} e^{\frac{i}{\hbar}\langle x-y,\xi\rangle}a(x,\xi)\chi_0(|x-y|)
u(y)
dyd\xi +E
\end{equation}
where $a\in S^m(T^*\mathbb{R}^d)$ and $E=O(\hbar^\infty)_{\Psi^{-\infty}}$, i.e.~for all $N>0$ there exists $C_N>0$ such that
$$
\|E\|_{H_\hbar^{-N}(\mathbb{R}^d)\to H_\hbar^N(\mathbb{R}^d)}\leq C_N\hbar^N. 
$$
We use the notation $\operatorname{Op}_\hbar a$ or $a(x,hD_x)$ for the operator $A$ in~\eqref{e:basicPseudo}  with $E=0$.  
  We then define 
$$
\Psi^{-\infty}:=\bigcap_m \Psi^m,\qquad S^{-\infty}:=\bigcap_m S^m,\qquad  \Psi^\infty:=\bigcup_m \Psi^m, \qquad S^\infty:=\bigcup_m S^m.
$$

\begin{theorem}\mythmname{\cite[Propositions E.17 and E.19]{DyZw:19}} If $A\in \Psi_{\hbar}^{m_1}$ and $B  \in \Psi_{\hbar}^{m_2}$, then
\begin{itemize}
\item[(i)]  $AB \in \Psi_{\hbar}^{m_1+m_2}$,
\item[(ii)]  $[A,B] \in \hbar\Psi_{\hbar}^{m_1+m_2-1}$,
\item[(iii)]  For any $s \in \mathbb R$, $A$ is bounded uniformly in $\hbar$ as an operator from $H_\hbar^s$ to $H_\hbar^{s-m_1}$.
\end {itemize}
\end{theorem}

\paragraph{\textbf{Principal symbol}}  There exists a map 
$$
\sigma^m_\hbar:\Psi^m \to S^m/hS^{m-1}
$$
called the \emph{principal symbol map} and such that the sequence 
$$
0\to hS^{m-1}\overset{\operatorname{Op}_\hbar}{\rightarrow} \Psi^{m}\overset{\sigma^m_\hbar}{\rightarrow} S^m/hS^{m-1}\to 0
$$
is exact where $\operatorname{Op}_\hbar(a)=a(x,hD)$. 
When applying the map $\sigma^m_\hbar$ to elements of $\Psi^m$, we denote it by $\sigma_\hbar$ (i.e.~we omit the $m$ dependence). Key properites of $\sigma_\hbar$ are the following
\beq\label{eq:symbol}
\sigma_\hbar(AB)=\sigma_\hbar(A)\sigma_\hbar(B),\qquad \sigma_\hbar(A^*)=\overline{\sigma}_\hbar(A),\qquad i\hbar^{-1}\sigma_\hbar([A,B])=\{\sigma_\hbar(A),\sigma_\hbar(B)\}
\eeq
where $\{\cdot,\cdot\}$ denotes the Poisson bracket; see \cite[Proposition E.17]{DyZw:19}.

\paragraph{\textbf{Operator wavefront set}} 
To introduce a notion of wavefront set that respects both decay in $\hbar$ as well as smoothing properties of pseudodifferential operators, we introduce the set 
$$
\overline{T^*\mathbb{R}^d}:=T^*\mathbb{R}^d\sqcup ( \mathbb{R}^d\times S^{d-1})
$$
where $\sqcup$ denotes disjoint union and we view $\mathbb{R}^d\times S^{d-1}$ as the `sphere at infinity' in each cotangent fiber (see also~\cite[\S E.1.3]{DyZw:19} for a more systematic approach where $\overline{T^*\mathbb{R}^d}$ is introduced as the fiber-radial compactification of $T^*\mathbb{R}^d$). We endow $\overline{T^*\mathbb{R}^d}$ with the usual topology near points $(x_0,\xi_0)\in T^*\mathbb{R}^d$ and define a system of neighbourhoods of a point $(x_0,\xi_0)\in \mathbb{R}^d\times S^{d-1}$ to be
\begin{align*}
U_\e:=&\Big\{ (x,\xi)\in T^*\mathbb{R}^d\,\big|\, |x-x_0|<\e, |\xi|>\e^{-1}, \big|\tfrac{\xi}{\langle \xi\rangle }-\xi_0\big|<\e\Big\}\\
&\qquad\quad \sqcup \big\{ (x,\xi)\in \mathbb{R}^d\times S^{d-1}\,:\, |x-x_0|<\e.,\, |\xi-\xi_0|<\e\big\}.
\end{align*}

We now say that a point $(x_0,\xi_0)\in \overline{T^*\mathbb{R}^d}$ is not in the wavefront set of an operator $A\in \Psi^m$, and write $(x_0,\xi_0)\notin \WF(A)$, if there exists a neighbourhood $U$ of $(x_0,\xi_0)$ such that $A$ can be written as in~\eqref{e:basicPseudo} with 
$$
\sup_{(x,\xi)\in U} | \partial^\alpha_x \partial_\xi^\beta a(x,\xi)\langle \xi\rangle^N|\leq C_{\alpha \beta N} \hbar^N.
$$

\paragraph{\textbf{Elliptic set and elliptic parametrix}} 
We say that $(x_0,\xi_0)\in \overline{T^*\mathbb{R}^d}$ is in the elliptic set of $A$, and write $(x_0,\xi_0)\in \Ell(A)$, if there exists a neighbourhood $U$ of $(x_0,\xi_0)$ such that $A$ can be written as in~\eqref{e:basicPseudo} with 
$$
\inf_{(x,\xi)\in U} |a(x,\xi)\langle \xi\rangle^{-m}|\geq c >0.
$$
The motivation behind this definition is that semiclassical pseudo-differential operators are, up to a negligible term, micro-locally invertible on their elliptic set, as appears in the following elliptic parametrix construction. 
\begin{theorem}\mythmname{\cite[Proposition E.32]{DyZw:19}}
\label{t:ellip} 
Suppose that $A\in \Psi^{m_1}$ and $B\in \Psi_\hbar^{m_2}$ with $\WF(A)\subset \Ell(B)$. Then there exist $E_1,E_2\in \Psi_\hbar^{m_1-m_2}$ such that 
$$
A=E_1B+O(\hbar^\infty)_{\Psi^{-\infty}},\qquad A=BE_2+O(\hbar^\infty)_{\Psi^{-\infty}}.
$$
\end{theorem}

\paragraph{\textbf{Wavefront set of a tempered family of distributions}} 
We say that $u_\hbar$ is \emph{tempered} if for all $\chi\in C_c^\infty(\mathbb{R}^d)$ there exists $N>0$ such that
$$
\|\chi u\|_{H_\hbar^{-N}}<\infty. 
$$
For a tempered family of functions, $u_\hbar$  we say that $(x_0,\xi_0)\in \overline{T^*\mathbb{R}^d}$ is \emph{not} in the wavefront set of $u_\hbar$ and write $(x_0,\xi_0)\notin \WF(u_\hbar)$ if there exists $A\in \Psi^0$ with $(x_0,\xi_0)\in \Ell(A)$ such that for all $N$ there is $C_N>0$ such that 
$$
\|Au_\hbar\|_{H_\hbar^N}\leq C_N\hbar^N.
$$

\paragraph{\textbf{Semiclassical defect measures}} If $\hbar_n \rightarrow 0$, we say that a sequence  $(u_{n})_{n\geq 0} \subset L^2_{\rm loc}$ has 
semiclassical defect measure $\mu$ as $n \rightarrow \infty$ (associated to $\hbar_n$) if $\mu$ is a positive Radon measure on $T^* \mathbb R^d$
such that, as $n \rightarrow \infty$
\beq\label{eq:defectmeasure}
\tfa a\in C^\infty_c(T^*\mathbb R^d), \hspace{0.5cm}\langle a(x, \hbar_n D_x) u_n, u_n \rangle \rightarrow \int a \, d\mu.
\eeq
In addition, if $(f_{n})_{n\geq 0} \subset L^2_{\rm loc}$, we say that $u_n$ and $f_n$ have joint measure $\mu^j$ if $\mu^j$ is a Radon measure such that
\beq\label{eq:jointmeasure}
\tfa a\in C^\infty_c(T^*\mathbb R^d), \hspace{0.5cm}\langle a(x, \hbar_n D_x) u_n, \rhs_n \rangle \rightarrow \int a \, d\mu^j.
\eeq

\begin{theorem}\mythmname{\cite[Theorem 5.2]{Zworski_semi}} 
Assume that $(u_{n})_{n\geq 0} \subset L^2_{\rm loc}$ is uniformly bounded in $L^2_{\rm loc}$, that is, for any $\chi \in C^\infty_c(\mathbb R^d)$,
there exists $C>0$ such that for any $n$, $\Vert \chi u_n \Vert_{L^2} \leq C$. Then, $(u_{n})_{n\geq 0} $ has a subsequence $(u_{n_\ell})_{\ell\geq 0}$ admitting a semi-classical defect measure. If, in addition, $(f_{n})_{n\geq 0} \subset L^2$ is bounded in $L^2$ independently of $n$, $n_\ell$ can be taken such that $(u_{n_\ell})_{\ell\geq 0}$ and $(f_{n_\ell})_{\ell\geq 0}$ have a joint defect measure.
\end{theorem}

\subsection{Rough calculus}\label{sec:rough}

We need a semi-classical pseudo-differential calculus for $C^{r, \alpha}$ symbols. We collect here
the definition and properties of such operators that we use throughout the paper.
For $r\in\mathbb N$, $0<\alpha<1$ and $0\leq \rho <1$, we say that $p\in C^{r, \alpha}S^m$ if
$$
\Vert D^\beta_\xi p(\cdot, \xi)\Vert_{C^{r, \alpha}} \leq C_\beta \langle \xi \rangle^{m-|\beta|}.
$$
Moreover, we say that $B \in C^{r, \alpha}\Psi^m$ if $B = \operatorname{Op}_\hbar(b)$ with $b \in C^{r, \alpha}S^m$. 
\begin{lem}\mythmname{\cite[Lemma 3.8]{GaSpWu:20}}\label{lem:rough_calc}
For any $r\geq 0$, $0<\alpha<1$, $-r-\alpha<s<r+\alpha$, and $m \in \mathbb R$,
the map $\operatorname{Op}_{\hbar} : {C^{r, \alpha}}S^m \rightarrow \mathcal L({H^{s+m}_{\hbar}, H_{\hbar}^s})$ 
is bounded independently of $\hbar$. Moreover, for $a\in C_c^\infty$,
$$
 C^{1,\alpha}S^m\ni p\mapsto \hbar^{-1}[\operatorname{Op}_{\hbar}(p),\operatorname{Op}_{\hbar}(a)]\in \mc{L}(L^2,L^2)
$$
is bounded independently of $\hbar$.
\end{lem}

\begin{lem}
\label{l:commutator}
Let $0<\alpha<1$ and $Q = \operatorname{Op}_{\hbar}(q_0)+\hbar \operatorname{Op}_{\hbar}(q_1)$  with $q_0\in C^{1,\alpha}S^2$ and $q_1\in C^{0,\alpha}S^0$ and suppose that $u$ has defect measure $\mu$. Then for $a,b\in C_c^\infty$, 
$$
i\langle \hbar^{-1}\operatorname{Op}_{\hbar}(b)[\operatorname{Op}_{\hbar}(a),Q]u,u\rangle \to \mu(bH_{q_0}a),\qquad -i\langle u,\hbar^{-1}\operatorname{Op}_{\hbar}(b)[\operatorname{Op}_{\hbar}(a),Q]u\rangle \to \overline{\mu(bH_{q_0}a)}
$$
\end{lem}
\begin{proof}
Let $\psi \in C^\infty_c(\mathbb R)$ be such that $\psi = 1$ on $[-2, 2]$, and for $\epsilon >0$ we define
$$
q_{i, \epsilon}(x, \xi) := (\psi(\epsilon |D_x|) q_i)(x, \xi),
$$
where $q_0 \in S^2$ and $q_1 \in S^1$, and
\begin{equation}
\label{e:qeps}
Q_\epsilon := \operatorname{Op}_{\hbar} (q_{0,\epsilon}) + \hbar \operatorname{Op}_{\hbar} (q_{1,\epsilon}), \hspace{0.5cm} \tilde q_\epsilon := \lim_{\hbar\rightarrow 0} q_{0,\epsilon}.
\end{equation}
By~\cite[Equations 3.8 and 3.9]{GaSpWu:20},
$$
\begin{cases} 
\Vert [\hbar\operatorname{Op}_{\hbar}(q_1 - q_{1, \epsilon}), \operatorname{Op}_{\hbar}(a)u \Vert_{L^2} \leq C\hbar\epsilon^{\frac{\alpha}{2}}\Vert u \Vert_{L^2} + O_{\epsilon}(\hbar^2), \\ 
\Vert [\operatorname{Op}_\hbar(q_0 - q_{0, \epsilon}),\operatorname{Op}_{\hbar}(a)]u \Vert_{L^2} \leq C\hbar\epsilon \Vert u \Vert_{L^2}.
\end{cases} 
$$
Therefore,
\begin{equation} \label{eq:inva_approx1}
\Big| \hbar^{-1} \langle [\operatorname{Op}_{\hbar}(a), Q-Q_\e] u, u \rangle u,  \Big|+\Big|\langle u, \hbar^{-1} \langle [\operatorname{Op}_{\hbar}(a), Q-Q_\e] u \rangle u,  \Big| \leq C \epsilon^{\frac {\alpha}{2}} + O_\epsilon(\hbar).
\end{equation}
On the other hand, since, for any $T, U\in \Psi$, 
\begin{equation} \label{eq:symb_com}
\hbar^{-1} \sigma_\hbar([T, U]) = -i \{ \sigma_\hbar(T), \sigma_\hbar(U) \}, 
\end{equation}
we have that, as $\hbar\rightarrow 0$,
$$
i\hbar^{-1} \langle \operatorname{Op}_{\hbar}(b)[\operatorname{Op}_{\hbar}(a), Q_\epsilon] u, u \rangle \rightarrow \mu(bH_{\tilde q_\epsilon} a),\qquad -i\hbar^{-1} \langle u,  \operatorname{Op}_{\hbar}(b)[\operatorname{Op}_{\hbar}(a), Q_\epsilon] u \rangle \rightarrow \overline{\mu(H_{\tilde q_\epsilon} a)}.
$$
Therefore, sending $h \rightarrow 0$ in (\ref{eq:inva_approx1}) we obtain, by the above,
$$
\Big| i\lim_{\hbar\rightarrow 0} \hbar^{-1} \langle \operatorname{Op}_{\hbar}(b)[\operatorname{Op}_{\hbar}(a), Q]u, u \rangle-  \mu(bH_{\tilde q_\epsilon} a) \Big| +\Big| -i\lim_{\hbar\rightarrow 0} \hbar^{-1} \langle u, \operatorname{Op}_{\hbar}(b)[\operatorname{Op}_{\hbar}(a), Q]u \rangle-  \overline{\mu(bH_{\tilde q_\epsilon} a)} \Big| \leq C \epsilon^{\frac{\alpha}{2}}.
$$
Finally, since $q_0 \in C^{1,\alpha}S^2$ uniformly in $\hbar$, $H_{\tilde q_\epsilon} \rightarrow H_{ q_0}$. Sending $\epsilon \rightarrow 0$ and applying the dominated convergence theorem then proves the lemma.
\end{proof}

\begin{lem}
\label{l:roughNorm}
Suppose $0\leq \delta<1$ and
$$
|D_{\xi}^\beta a|\leq C_\beta \hbar^{(r+\alpha)\delta}\langle \xi\rangle^{m-|\beta|-(r+\alpha)\delta},\qquad \|D_{\xi}^\beta a\|_{C^{r,\alpha}_x}\leq C_\beta \langle \xi\rangle^{m-|\beta|}.
$$
Then, $\operatorname{Op}_{\hbar}(a):H_{\hbar}^{m}\to L^2$, and
$$
\|\operatorname{Op}_{\hbar}(a)\|_{H_{\hbar}^{m}\to L^2}\leq C\hbar^{(r+\alpha)\delta}.
$$
\end{lem}
\begin{proof}
It is enough to check this for $m=\delta (r+\alpha)$. For this, we unitarily transform to the case $\hbar=1$. Let $T_{\hbar}u(x)=\hbar^{d\delta/2}u(\hbar^\delta x)$. Then,  $T_{\hbar}:L^2\to L^2$ is unitary and $T\operatorname{Op}_{\hbar}(a)T^*= \operatorname{Op}_1(a_{\hbar})$ with
$$
a_{\hbar}(x,\xi)= a(\hbar^\delta x,\hbar^{1-\delta}). 
$$
It is now easy to check that 
$$
|D_{\xi}^\beta a_{\hbar}|\leq C_\beta \hbar^{\delta(r+\alpha)}\langle \xi\rangle^{-|\beta|},\qquad \|D_{\xi}^\beta a\|_{C^{r,\alpha}}\leq C_\beta \hbar^{\delta r}\langle \xi\rangle ^{\delta |\alpha|-\beta}.
$$
Therefore, the lemma follows from~\cite[Chapter 13, Proposition 9.10]{Ta:11}.
\end{proof}

\begin{lem}
\label{l:roughApprox}
Suppose that $a\in C^{r,\alpha}S^m$. Then there is $a_{\hbar}$ satisfying
$$
\begin{gathered}
|\partial_x^\gamma \partial_\xi^\beta a_{\hbar}(x,\xi)|\leq C_{\gamma \beta} \hbar^{-\delta\gamma }\langle \xi\rangle ^{m-|\beta|+\delta \gamma},\\
|D_\xi^\beta (a-a_\hbar)|\leq \hbar^{(r+\alpha)\delta}\langle \xi\rangle^{m-|\beta|-(r+\alpha)\delta},\quad \|D_{\xi}^\beta(a-a_{\hbar})(\cdot,\xi)\|_{C^{r,\alpha}}\leq C_\beta \langle \xi\rangle ^{m-|\beta|}.
\end{gathered}
$$
\end{lem}
\begin{proof}
Let $\varphi_0\in C_c^\infty(-1,1)$, $\varphi\in C_c^\infty(\frac{1}{2},2)$ such that 
$$
\varphi^2_0(|s|)+\sum_{j\geq 0} \varphi^2(2^{-j}|s|)\equiv 1.
$$ 
Then put
$$
a_{\hbar}(x,\xi)=(\varphi^2_0(\hbar^\rho |D_x|)a)(x,\xi)\varphi^2_0(\xi)+\sum_{j\geq 0}(\varphi^2_0(\hbar^\rho 2^{-j\rho}|D_x|)a)(x,\xi)\varphi^2_j(\xi).
$$
The estimates now follow as in the proof of~\cite[Chapter 13, Proposition 9.9]{Ta:11}.
\end{proof} 

\section{Properties of $\Phi_\theta(r)$}\label{app:Phi}
\begin{proof}[Proof of Lemma~\ref{lem:Phi}]
We first note that, using the principle square root,
$$
\Big\{ z\in \mathbb{C}\,:\, \Im \sqrt{z}=a\Big\}:=\Big\{ \frac{y^2}{4a^2}-a^2+iy\,:\, y\in \mathbb{R}\Big\}=:\mc{Z}_a.
$$
Therefore, if $z$ lies to the left of $\mc{Z}_a$, then $\Im \sqrt{z}>a$.

We are interested in 
$$
z(t,r)=(1+if_\theta'(r))^2-\frac{t(1+if_\theta'(r))^2}{(r+if_\theta(r))^2},\qquad t\geq 0.
$$
Note in particular, that $z(0,r)\in \mc{Z}_{f_\theta'(r)}$ and the tangent to  $\mc{Z}_{f_\theta'(r)}$ at $z(0,r)$ is given by 
$$
\frac{2f_\theta'(r)}{2f_\theta'(r)^2}+i=\frac{1}{f_\theta'(r)}(1+if_\theta'(r)).
$$

Next, observe that
$$
\partial_t z(t,r)=-(1+if_\theta'(r))\frac{(1+if_\theta'(r))}{(r+if_\theta(r))^2}.
$$
Hence, since $\mc{Z}_{f_\theta'(r))}$ is convex, $z(t,r)$ lies to the left of $\mc{Z}_{f_\theta'(r)}$ for $t>0$ (and thus $\displaystyle{\min_{t\geq 0}}\Im \sqrt{z(t,r)}=z(0,r)$)
 if and only if
$$
\Im-\frac{(1+if_\theta'(r))}{(r+if_\theta(r))^2}=\Im- \frac{(1+i\tan \theta f'(r))}{(r+i\tan \theta f (r))^2}\geq 0\qquad \Leftrightarrow\qquad \tan^2 \theta\geq \frac{r^2}{f(r)^2}-\frac{2r}{f'(r)f(r)},
$$
and Point \eqref{i:iffTheta} follows. Point \eqref{i:evLinear} then follows from Point \eqref{i:iffTheta}

Now, fix $\delta>0$ and let $g(r)$ denote the right-hand side of \eqref{e:PhitCond}. Then, there is $c_\delta>0$ such that both $f(r)>c_\delta$ and $f'(r)>c_\delta$ on $r>R_1+\delta$, and thus $g(r)<C_\delta$. Then by~\eqref{e:PhitCond}, since $\tan\theta \to \infty$ as $\theta\uparrow \pi/2$, there is $\theta_\delta$ such that for $\theta>\theta_\delta$, $\Phi_\theta(r)=f_\theta'(r)$ and hence~\eqref{i:rSmall} holds.  

To obtain~\eqref{i:thetaDep}, observe that by~\eqref{i:rSmall}, for $r>R_1+\delta$, and $\theta>\theta_\delta$, $\Phi_\theta(r)= f(r)\tan\theta>c_\delta \tan \theta$. Therefore, the result follows if $\Phi_\theta(r)>c_{\delta}$ for $\delta \leq \theta \leq \theta_\delta$, which was proved in Lemma~\ref{l:carleman1}.

Finally, we prove~\eqref{i:continuous}. Indeed, for $r\leq R_1$, $\Phi_\theta(r)\equiv 0$, and for $r\geq R_2$, $\Phi_\theta(r)=r\tan\theta$. Therefore, we need only consider $(r,\theta)\in[R_1,R_2]\times(0,\pi/2)$.

Since we are using the principle square root and $f_\theta\geq0$, $f'_\theta \geq0$, we have, for $t\geq 0$,
\beqs
\Arg \sqrt{1-\frac{t}{(r+if_\theta(r))^2}}\in [0,\pi/2),
\eeqs
and thus
$$\Phi_\theta(r)=\inf_{t\geq0 }\tilde{\Phi}_\theta(r,t)\quad\text{ where } \quad\tilde{\Phi}_\theta(r,t):=\Im \Big((1+if_\theta'(r))\sqrt{1-\frac{t}{(r+if_\theta(r))^2}}\,\Big).$$
Next, for $r>R_1$, $\theta>0$
$$
\lim_{t\to \infty}\tilde{\Phi}_{\theta}(r,t)=\infty;
$$
therefore, the infimum is achieved at some finite $t$, which we denote by $t_m=t_m(r,\theta)$. It is easy to check that, when~\eqref{e:PhitCond} does not hold,
\begin{align}
t_m(r,\theta)
=\max\bigg(\frac{\Im\big( (1+if_\theta')^2(r-if_\theta)^4\big)}{\Im\big((1+if_\theta')^2(r-if_\theta)^2\big)},0\bigg).
\label{e:tCrit}
\end{align}
Therefore
$$
t_m(r,\theta):=\begin{cases} 0& \text{if }\Im (1+if_\theta')(r-if_\theta)^2\leq 0,\\
\max\bigg(\dfrac{\Im \big((1+if_\theta')^2(r-if_\theta)^4\big)}{\Im\big((1+if_\theta')^2(r-if_\theta)^2\big)},0\bigg)
&\text{otherwise}
\end{cases}
$$
Note that $t_m(r,\theta)$ is continuous since 
the numerator of the left entry of the maximum in  \eqref{e:tCrit} is zero when $\Im\big( (1+if_\theta')(r-if_\theta)^2\big)=0$, and
the singularity in the left entry of the maximum in~\eqref{e:tCrit} occurs when $\Im\big( (1+if_\theta')(r-if_\theta)^2\big)\geq 0$;
this completes the proof.
\end{proof}

\bibliographystyle{amsalpha}
\bibliography{biblio_GLS}

\end{document}